\documentclass[10pt,electronic]{kthesis}
\usepackage[T1]{fontenc}
\usepackage{lmodern,amsfonts,amsthm,amsmath,amssymb,verbatim,enumerate,eepic}
\usepackage{mathrsfs}
\usepackage[utf8]{inputenc}
\usepackage{epsfig}
\usepackage{pdfpages}
\usepackage{color}
\usepackage[swedish,french,german,english]{babel}
\usepackage{enumerate}
\usepackage{macros}
\usepackage{cite}
\usepackage{float}
\usepackage{algorithm, algorithmicx}
\usepackage[hidelinks]{hyperref}
\usepackage{multirow}
\usepackage{cite}

 \newcommand{\inv}{^{-1}}
 
 \newcommand{\Lambdait}{\mathit \Lambda}
  \newcommand{\muhat}{\hat{\mu}}
  \newcommand{\mubar}{\bar{\mu}}

\theoremstyle{plain}
\newtheorem{theorem}{Theorem}[section]
\newtheorem{lemma}[theorem]{Lemma}
\newtheorem{proposition}[theorem]{Proposition}

\newtheorem{assumption}{Assumption}

\theoremstyle{definition}
\newtheorem{definition}[theorem]{Definition}

\theoremstyle{remark}

\makepagestyle{minech}
\makeoddfoot{minech}{}{\thepage}{}
\makeevenfoot{minech}{}{\thepage}{}

\copypagestyle{mine}{ruled}
\makeoddhead{ruled}{}{}{Approximate solution of linear system arising in interior-point methods}
\makeevenhead{ruled}{}{}{}
\makeevenfoot{mine}{}{\thepage}{}
\makeoddfoot{mine}{}{\thepage}{}

\counterwithout{section}{chapter}
\counterwithout{figure}{section}
\counterwithout{table}{section}
\setsecnumdepth{subsection}
 \makeatletter
  \renewcommand{\fnum@figure}{\textbf
  {\figurename~\thefigure}}
  \renewcommand{\fnum@table}{\textbf
  {\tablename~\thetable}}
 \makeatother    

\captionnamefont{\footnotesize}
\captiontitlefont{\footnotesize}

\begin{document}
\chapterstyle{article}
\chapter*[]{Approximate solution of system of equations arising in interior-point methods for bound-constrained optimization}
\thispagestyle{minech} 
David Ek\footnote{Optimization and Systems Theory, Department of
  Mathematics, KTH Royal Institute of Technology, SE-100 44 Stockholm,
  Sweden (\texttt{daviek@kth.se,andersf@kth.se}).\label{fn} } and Anders Forsgren\footref{fn}   \footnote[0]{Research partially supported by the Swedish Research Council (VR).}

\begin{abstract}
The focus in this paper is interior-point methods for bound-constrained nonlinear optimization, where the system of nonlinear equations that arise are solved with Newton's method. There is a trade-off between solving Newton systems directly, which give high quality solutions, and solving many approximate Newton systems which are computationally less expensive but give lower quality solutions. We propose partial and full approximate solutions to the Newton
systems. The specific approximate solution depends on estimates of the active and inactive constraints at the solution. These sets are at each iteration estimated by basic heuristics. The partial approximate solutions are computationally inexpensive, whereas a system of linear equations needs to be solved for the full approximate solution. The size of the system is determined by the estimate of the inactive constraints at the solution. In addition, we motivate and suggest two Newton-like approaches which are based on an intermediate step that consists of the partial approximate solutions. The theoretical setting is introduced and asymptotic error bounds are given. We also give numerical results to investigate the performance of the approximate solutions within and beyond the theoretical framework.  

\medskip\noindent
\textbf{Keywords:} interior-point methods, bound-constrained optimization, approximate solution of system of linear equations, Newton-like approaches.
\end{abstract}

\section{Introduction} \label{intro}
This work is intended for bound-constrained nonlinear optimization problems on the form
\begin{equation} \tag{NLP} \label{eq:NLP} 
\begin{array}{cl} 
\textrm{minimize }  & \>  f(x) \\ 
\textrm{subject to } & \>  l \leq x \leq u, 
\end{array} 
\end{equation}
where $f: \mathbb{R}^n \rightarrow \mathbb{R}$ is twice continuously differentiable, $\nabla^2 f(x)$ is locally Lipschitz continuous and $l, u\in \left\{\mathbb{R} \cup \{ -\infty, \infty \} \right\}^n$ are such that $l< u$. However, to make the work and its ideas more comprehensible, we initially describe the theoretical framework and the corresponding results for problems on the form
\begin{equation} \tag{P} \label{eq:P} 
\begin{array}{cl} 
\textrm{minimize }  & \>    f(x) \\ 
\textrm{subject to } & \>  x \geq 0.
\end{array} 
\end{equation}
For completeness, analogous results for problems on the form of (\ref{eq:NLP}) together with complementary remarks are given in Appendix~\ref{app:genCase}.

Bound-constrained optimization problems appear in many different applications and are frequently subproblems in augmented Lagrangian methods. For a general overview of solution methods, see \cite{GouOrbToi2005} and e.g., the introduction in \cite{HagZha2006} for a thorough review of previous work. Common solution techniques are: active-set methods, which aim to determine the active constraints and solve a reduced problem with the inactive variables, e.g., \cite{HagZha2006,FacSoa1998}; methods involving projections onto the feasible set such as projected-gradient methods, e.g., \cite{Ber1976,SchPol1997}, projected-Newton or trust-region methods, e.g., \cite{Ber1982,LinMor1999,ConGouToi1988,ConGouToi2000} and projected quasi-Newton methods, e.g., \cite{ByrLuNoc1995,ZhuByrLuNOc1997,KimSuvDhi2010}. We are not aware of any primal-dual interior-point methods specialized for bound-constrained optimization except for more general methods, e.g.,  \cite{WalMorNocOrb2006,WacBieLor2006,VanSha1999,ForGil1998,ForGilWri2002}. Other techniques that are related to trust-region and interior methods are affine-scaling interior-point methods, which are based upon a reformulation of the first-order necessary optimality conditions combined with a Newton-like method, e.g., \cite{ColLi1994,HeiUlbUlb1999,KanKlu2006}.

In contrast, we consider the classical primal-dual interior-point
framework. This means solving or approximately solving a sequence of
systems of nonlinear equations for which we consider Newton's methods
as the model method. As interior methods converge, the Newton systems typically become increasingly ill-conditioned due to large diagonal elements in the Schur complement. This is not harmful for direct solvers but it may deteriorate the performance of iterative solvers. We propose a strategy for generating approximate solutions to Newton systems, which in general involves solving smaller systems of linear equations. In the ideal case, these systems do not become increasingly ill-conditioned due to the barrier parameter approaching zero. The specific approximate solutions, and the size of the system that needs to be solved at each iteration, are determined by estimates of the active and inactive constraints at the solution. However, in general these sets are unknown and have to be estimated as the iterations proceed. In this work we use basic heuristics to determine the considered sets but other approaches may also be used, e.g., approaches similar to those in \cite{HagZha2006,FacSoa1998}. In addition, we motivate and suggest two Newton-like approaches which utilize an intermediate step in combination with the solution of a Newton-like system. The intermediate step partially consists of the proposed partial approximate solutions.

The work is meant to contribute to the theoretical and numerical understanding of approximate solutions to systems of linear equations arising in interior-point methods. The approach is mainly intended for, but not limited to, bound-constrained problems, e.g., the work may also be interpreted in the framework of linear complementarity problems, see e.g.,  \cite{Wri1995s}. We envisage the use of the approximate solution procedure as an
accelerator for a direct solver. In particular, when solving a sequence of Newton systems for a given value of the barrier parameter $\mu$. E.g., when the direct solver and the approximate solution procedure can be run in parallel. 
To give an indication of the potential of the approximate solutions, we show numerical simulations on randomly generated problems as well as problems from the CUTEst test collection \cite{GouOrbToi15}. 

The manuscript is organized as follows; Section~\ref{sec:Background} contains a brief background to primal-dual interior-point methods and an introduction to the theoretical framework; in Section~\ref{sec:Approximations} we propose partial and full approximate solutions to Newton systems arising in interior-point methods, as well as motivate two Newton-like approaches; Section~\ref{sec:numRes} contains numerical results on convex bound-constrained quadratic optimization problems, both randomly generated and problems from the CUTEst test collection; finally in Section~\ref{sec:Conc} we give some concluding remarks.
\section{Background} \label{sec:Background}
We are interested in the asymptotic behavior of primal-dual
interior-point methods in the vicinity of a local minimizer $x^*$ and
its corresponding multipliers $\lambda^*$. In particular, we assume
that the iterates of the method converge to a vector $\left( x^{*T}, \lambda^{*T} \right)^T \triangleq (x^*,\lambda^*)$ that satisfies
\begin{subequations} \label{eq:OptCond} 
\begin{align}
\nabla f(x^*) - \lambda^* & = 0, \quad \textrm{ (stationarity) } \label{eq:OptCond:stationarity} \\
x^* & \geq 0, \quad \textrm{ (feasibility)} \label{eq:OptCond:feasibility} \\
 \lambda^* & \geq 0, \quad \textrm{ (non-negativity of multipliers)} \label{eq:OptCond:nonnegofmultipliers}  \\
 x^* \cdot \lambda^* & = 0,  \quad \textrm{ (complementarity)} \label{eq:OptCond:complementarity}\\
 Z(x^*)^T \nabla^2 f(x^*) Z (x^*) & \succ 0 \label{eq:OptCond:posdef}, \\
 x^* + \lambda^* & > 0,  \quad \textrm{ (strict complementarity)} \label{eq:OptCond:strictcomplementarity}
\end{align} 
\end{subequations}%
where "$\cdot$" is defined as the component-wise operator and $Z(x^*)$ is a matrix whose columns span the nullspace of the Jacobian corresponding to the constraints with a strictly positive multiplier, $\lambda^*$. Equations (\ref{eq:OptCond:stationarity})-(\ref{eq:OptCond:complementarity}) constitute first-order necessary optimality conditions for a local minimizer of (\ref{eq:P}). These conditions together with (\ref{eq:OptCond:posdef}) form second-order sufficient conditions \cite{GriNasSof2009}.
For the theoretical framework we also assume that $(x^*, \lambda^*)$ satisfies  (\ref{eq:OptCond:strictcomplementarity}).
We are particularly interested in the function $F_{\mu}:\mathbb{R}^{2n} \rightarrow \mathbb{R}^{2n}$ defined by 
\begin{equation*}
F_{\mu}(x,\lambda) = \begin{bmatrix}
\nabla f(x) - \lambda \\
\Lambdait X e - \mu 
\end{bmatrix},
\end{equation*}
where $\mu \in \mathbb{R}$ is the barrier parameter, $X \in \mathbb{R}^{n \times n}, \Lambdait \in \mathbb{R}^{n \times n}$, $X=\textrm{diag}(x)$, \linebreak$\Lambdait = \textrm{diag}(\lambda)$ and $e$ is a vector of ones of appropriate size. A vector $(x,\lambda)$ with $x\ge0$, $\lambda\ge0$ and $F_{\mu}(x, \lambda) = 0$ for $\mu=0$ satisfies the first-order optimality conditions (\ref{eq:OptCond:stationarity})-(\ref{eq:OptCond:complementarity}).
Primal-dual interior-point methods aim to solve or approximately solve $F_{\mu}(x,\lambda) = 0$ for a decreasing sequence of $\mu>0$, while maintaining $x>0$ and $\lambda > 0$. This is typically done with Newton-like methods which means solving a sequence of systems of linear equations on the form
\begin{equation}\label{eq:PDsyst} 
F'(x, \lambda) \begin{bmatrix}
\Delta x^N \\ \Delta \lambda^N
\end{bmatrix} = -F_{\mu}(x,\lambda),
\end{equation}
where $F': \mathbb{R}^{2n} \rightarrow \mathbb{R}^{2n}$ is the Jacobian of $F_{\mu}$. The Jacobian is given by
\begin{equation}
F'(x, \lambda)  = \begin{bmatrix}
H & -I \\
\Lambdait & X
\end{bmatrix}, \label{eq:Fp}
\end{equation}
where $H=\nabla^2 f(x)$ and the subscript $\mu$ is omitted since $F'$ is independent of the barrier parameter. For each $\mu$, iterations are performed until a specified measure of improvement is achieved, thereupon $\mu$ is decreased and the process is repeated. A natural measure in our setting is $\| F_{\mu}(x,\lambda) \|_2 $ where $\| F_{\mu}(x,\lambda) \|_2 =0$ gives the exact solution. To improve efficiency many algorithms seek approximate solutions, a basic condition for the reduction of $\mu$ is $\| F_{\mu}(x,\lambda) \|_2  < \mu$ \cite[Ch.~17, p. 572]{NocWri06}. Herein, we consider a possibly weaker version, namely $\| F_{\mu}(x,\lambda) \|_2  < C_1 \mu$ for some constant $C_1 >0$. Moreover, it will throughout be assumed that all considered vectors $(x,\lambda)$ satisfy $x>0$ and $\lambda > 0$. The subscript in the norms will hereafter be omitted since all considered norms in this work are of type 2-norm.
\begin{definition}[Order-notation]
Let $\alpha$, $\gamma \in \mathbb{R}$ be two positive related quantities. If there exists a constant $C_2>0$ such that $\gamma \geq C_2 \alpha$ for sufficiently small $\alpha$, then $\gamma = \Omega(\alpha)$. Similarly, if there exists a constant $C_2>0$ such that $\gamma \leq C_2 \alpha$ for sufficiently small $\alpha$, then $\gamma = \mathcal{O}(\alpha)$. If there exist constants $C_2, C_3  > 0$ such that $C_2 \alpha \leq \gamma \leq C_3  \alpha$ for sufficiently small $\alpha$ then, $\gamma = \Theta(\alpha)$.
\end{definition}
\begin{definition}[Neighborhood] For a given $\delta >0$, let the neighborhood around $(x^*, \lambda^*)$ be defined by $\mathcal{B}( (x^*, \lambda^*), \delta) = \{ (x,\lambda) : \| (x,\lambda)-(x^*, \lambda^*) \|  <~\delta \}$. 
\end{definition}
\begin{assumption}[Strict local minimizer] The vector $(x^*, \lambda^*)$ satisfies (\ref{eq:OptCond}), i.e., second-order sufficient optimality conditions and strict complementarity. \label{ass1}
 \end{assumption}
The following two results provide the theoretical framework and additional definitions of various quantities. In particular, the existence of a neighborhood where the Jacobian is nonsingular and there exists a Lipschitz continuous barrier trajectory which is parameterized by the barrier parameter $\mu$. The results are well known and can be found in e.g., the work of Ortega and Rheinboldt \cite{OrtRhe00} and Byrd, Liu and Nocedal \cite{ByrLuiNoc98} whose setting is similar to the one in this work. 
\begin{lemma} \label{lemma:background:FpnonsingCont}
Under Assumption~\ref{ass1} there exists $\delta>0$ such that  $F'(x,\lambda)$ is continuous and nonsingular for $(x,\lambda) \in \mathcal{B}((x^*, \lambda^*), \delta)$ and 
\begin{equation*}
\| F'(x, \lambda) \inv \|  \leq M,
\end{equation*}
for some constant $M>0$.
\end{lemma}
\begin{proof}
See \cite[p. 46]{OrtRhe00}. 
\end{proof}
\begin{lemma} \label{lemma:LipcPath}
Let Assumption~\ref{ass1} hold and let $\mathcal{B}((x^*, \lambda^*), \delta)$ be defined by Lemma~\ref{lemma:background:FpnonsingCont}. Then there exists $\muhat>0$ such that for each $0<\mu \leq \muhat$ there is a Lipschitz continuous function $(x^{\mu},\lambda^{\mu}) \in \mathcal{B}((x^*, \lambda^*), \delta)$ that satisfies $F_{\mu}(x^{\mu},\lambda^{\mu}) = 0$ and
\begin{equation*}
\left\|  \left( x^{\mu}, \lambda^{\mu}
\right) - \left(
x^*, \lambda^* \right)
 \right\|   \leq C_4 \mu,
\end{equation*} 
where $C_4 = \inf_{(x,\lambda) \in \mathcal{B}((x^*, \lambda^*), \delta)} \| F'(x,\lambda) \inv \frac{ \partial F_{\mu} (x,\lambda)}{\partial \mu} \| $.
\end{lemma}
\begin{proof}
The result follows from the implicit function theorem, see e.g., \cite[p. 128]{OrtRhe00}.
\end{proof}
The next lemma relates the measure $\| F_\mu (x,\lambda) \| $ to the distance between the barrier trajectory and vectors $(x,\lambda)$ that are sufficiently close. An analogous result is given by Byrd, Liu and Nocedal \cite{ByrLuiNoc98}.
\begin{lemma} \label{lemma:FmuBound}
Under Assumption~\ref{ass1}, let $\mathcal{B}((x^*, \lambda^*), \delta)$ and $\muhat$ be defined by  Lemma~\ref{lemma:background:FpnonsingCont} and Lemma~\ref{lemma:LipcPath} respectively. For $0<\mu \le \muhat$ and $(x,\lambda)$ sufficiently close to $(x^{\mu},\lambda^{\mu}) \in \mathcal{B}((x^*, \lambda^*), \delta)$ there exist constants $C_5, C_6 > 0$ such that
\begin{equation*}
C_5  \left\|  \left(
x , \lambda
\right) - \left(
x^{\mu} , \lambda^{\mu}
\right) \right\|   \leq \| F_{\mu}(x,\lambda) \|   \leq C_6 \left\|  \left(
x ,\lambda
\right) - \left(
x^{\mu}, \lambda^{\mu}
\right) \right\| .
\end{equation*} 
\end{lemma}
\begin{proof}
See \cite[p. 43]{ByrLuiNoc98}. 
\end{proof}
Recall that the reduction of $\mu$ can be determined with the condition $\|F_\mu (x,\lambda)\|<C_1 \mu$, for some constant $C_1>0$. It can be shown that vectors $(x,\lambda)$, which satisfy this condition and are sufficiently close to the barrier trajectory, have their individual components bounded within certain intervals at sufficiently small $\mu$. The individual components can be partitioned into two sets of indices which depend on how close the iterate is to its feasibility bound, see Definition~\ref{def:activeinactive}. The order of magnitude of the individual components, which are given in Lemma~\ref{lemma:Order} below, will be of importance in the derivation of various approximate solutions to (\ref{eq:PDsyst}). 
\begin{definition} \label{def:activeinactive}
(Active/inactive constraint). For a given $x^* \geq 0$ constraint \linebreak $i\in \{1, \dots, n \}$ is defined as active if $x_i^* = 0$ and inactive if $x_i^* > 0$. The corresponding active and inactive set are defined as $\mathcal{A} = \{i\in \{1, \dots, n \} : x^*_i = 0 \}$, and $\mathcal{I}=\{i\in\{1,\dots,n\}:x^*_i>0\}$ respectively.
\end{definition}
\begin{lemma} \label{lemma:Order}
Under Assumption~\ref{ass1}, let $\mathcal{B}\left((x^*, \lambda^*),
  \delta\right)$ and $\muhat$ be defined by
Lemma~\ref{lemma:background:FpnonsingCont} and
Lemma~\ref{lemma:LipcPath} respectively. Then there exists $\mubar$, with
$0 <\mubar \le \muhat$, such that for $0 < \mu \leq \bar\mu$ and $(x,\lambda)$ sufficiently close to $(x^{\mu}, \lambda^{\mu}) \in \mathcal{B}((x^*, \lambda^*), \delta)$ so that $\| F_{\mu} (x,\lambda) \| = \mathcal{O}(\mu)$ it holds that
\begin{equation} \label{eq:lemma:Order}
x_i = \begin{cases} \mathcal{O}(\mu) & i \in \mathcal{A}, \\ 
                    \Theta(1) & i \in \mathcal{I}, \end{cases} 
                    \qquad \lambda_i = \begin{cases} \Theta(1) & i \in \mathcal{A}, \\ 
                    \mathcal{O}(\mu) & i \in \mathcal{I}. \end{cases}
\end{equation}
\end{lemma}
\begin{proof}
Under Assumption~\ref{ass1} it holds that
\begin{equation*}
x_i^* = \begin{cases} 0 & i \in \mathcal{A}, \\ 
                    c_i & i \in \mathcal{I}, \end{cases} 
                    \qquad \lambda_i^* = \begin{cases} c_i & i \in \mathcal{A}, \\ 
                    0 & i \in \mathcal{I}, \end{cases}
\end{equation*}
where $c_i = \Theta(1)$, $i = 1,\dots, n$. The function $(x^{\mu}, \lambda^{\mu})$ is Lipschitz continuous and hence for each $\mu \leq \muhat$ it holds that $(x^{\mu}, \lambda^{\mu}) \in \mathcal{B}\left( (x^*, \lambda^*), L_{F'} \mu \right)$, where $L_{F'}$ is the Lipschitz constant of $F'$ on $\mathcal{B}((x^*, \lambda^*), \delta)$. There exist $\mubar_1$, with $0 < \mubar_1 \leq \muhat$, such that for $0 < \mu \leq \mubar_1$ it holds that
\begin{equation*}
x_i^\mu = \begin{cases} \mathcal{O}(\mu) & i \in \mathcal{A}, \\ 
                    \Theta(1) & i \in \mathcal{I}, \end{cases} 
                    \qquad \lambda_i^\mu = \begin{cases} \Theta(1) & i \in \mathcal{A}, \\ 
                    \mathcal{O}(\mu) & i \in \mathcal{I}. \end{cases}
\end{equation*}
The condition $\| F_{\mu} (x,\lambda) \| = \mathcal{O}(\mu)$ implies that there exists a constant $C_1>0$ such that $\| F_{\mu} (x,\lambda) \| \leq C_1 \mu$. Lemma~\ref{lemma:FmuBound} and $\| F_{\mu} (x,\lambda) \| \leq C_1 \mu$ give
\[
 \left\|  \left(
x ,\lambda
\right) - \left(
x^{\mu} , \lambda^{\mu}
\right)  \right\|  \leq \frac{1}{ C_5 }\| F_{\mu}(x,\lambda) \|  \leq \frac{C_1}{C_5} \mu,
\]
which implies that $(x,\lambda) \in \mathcal{B}\left( (x^\mu, \lambda^\mu), \frac{C_1}{ C_5} \mu \right)$. Similarly here, there exists $\mubar_2$, with $0 < \mubar_2 \leq \muhat$, such that the result follows for $0<\mu \le \mubar$ with $\mubar = \min \{\mubar_1, \mubar_2 \}$. 
\end{proof}
The result of Lemma~\ref{lemma:Order} shows two regions which depend on $\mu$. The first region, $0<\mu \leq \muhat$, defines where  the barrier trajectory $(x^\mu, \lambda^\mu)$ exists and the second region, $0 < \mu \le \mubar \leq \muhat$, defines where asymptotic behavior occurs. 
\section{Approximate solutions} \label{sec:Approximations}
This section initially contains an introduction to the groundwork of the ideas which precede the results. It is followed by a subsection that contains approximate solutions for specific components of the solution of (\ref{eq:PDsyst}) together with related results. The last subsection contains procedures for approximating the full solution of (\ref{eq:PDsyst}), as well as related results. Under Assumption~\ref{ass1} it holds that
\[
\lim_{\mu \to 0}  x^\mu_i = 0, \quad i \in \mathcal{A}, \qquad \mbox{ and }\qquad \lim_{\mu \to 0} \lambda_i^\mu = 0, \quad i \in \mathcal{I},
\]
in consequence, the Schur complement of $X$ in (\ref{eq:PDsyst}) becomes increasingly ill-\linebreak conditioned as $\mu \to 0$. These properties have been utilized by several authors before, e.g., in the development of preconditioners \cite{GMPS92,FGG07}.
The idea in this work is to exploit them and the additional property that (\ref{eq:P}) only has bound constraints to obtain partial or full approximate solutions of (\ref{eq:PDsyst}). In particular, by utilization of structure and the asymptotic behavior of coefficients in the arising systems of linear equations. With the partition $(\Delta x^N, \Delta \lambda^N) = (\Delta x_\mathcal{A}^N, \Delta x_\mathcal{I}^N, \Delta \lambda_\mathcal{A}^N, \Delta \lambda_\mathcal{I}^N)$, (\ref{eq:PDsyst}) can be written as
\begin{equation}\label{eq:PDsyst_partitioned}
\begin{bmatrix}
H_{\mathcal{A} \mathcal{A}}  & H_{\mathcal{A} \mathcal{I}}  & -I_{\mathcal{A} \mathcal{A}} &  \\
H_{\mathcal{I} \mathcal{A}} & H_{\mathcal{I} \mathcal{I}}  &  & -I_{\mathcal{I}\mathcal{I}} \\
\Lambdait_{\mathcal{A} \mathcal{A}} &  & X_{\mathcal{A} \mathcal{A}} &  \\
 & \Lambdait_{\mathcal{I} \mathcal{I}} &  & X_{\mathcal{I} \mathcal{I}} 
\end{bmatrix} \begin{bmatrix} \Delta x_\mathcal{A}^N \\\Delta x_\mathcal{I}^N \\ \Delta \lambda_\mathcal{A}^N \\ \Delta \lambda_\mathcal{I}^N
\end{bmatrix} = -\begin{bmatrix}
\nabla f(x)_\mathcal{A} - \lambda_\mathcal{A} \\
\nabla f(x)_\mathcal{I} - \lambda_\mathcal{I} \\
\Lambdait_{\mathcal{A} \mathcal{A}} X_{\mathcal{A} \mathcal{A}} e - \mu e \\
\Lambdait_{\mathcal{I} \mathcal{I}} X_{\mathcal{I} \mathcal{I}} e - \mu e
\end{bmatrix},
\end{equation}
where the first and second set in the matrix subscripts give the indices of rows and columns respectively.  
The Schur complement of $X_{\mathcal{A} \mathcal{A}}$ and $X_{\mathcal{I} \mathcal{I}}$ in (\ref{eq:PDsyst_partitioned}) is 
\begin{equation} \label{eq:SchurC_partioned}
 \begin{bmatrix}
H_{\mathcal{A} \mathcal{A}} +X_{\mathcal{A} \mathcal{A}} \inv \Lambdait_{\mathcal{A} \mathcal{A}} & H_{\mathcal{A} \mathcal{I}} \\
H_{\mathcal{I} \mathcal{A}} & H_{\mathcal{I} \mathcal{I}} + X_{\mathcal{I} \mathcal{I}}\inv \Lambdait_{\mathcal{I} \mathcal{I}}
\end{bmatrix}  \begin{bmatrix} \Delta x_\mathcal{A}^N \\\Delta x_\mathcal{I}^N
\end{bmatrix} = - \begin{bmatrix}
\nabla f(x)_\mathcal{A} - \mu X_{\mathcal{A} \mathcal{A}} \inv e \\
\nabla f(x)_\mathcal{I} - \mu X_{\mathcal{I} \mathcal{I}} \inv e \\
\end{bmatrix}.
\end{equation}
By continuity of $(x^{\mu}, \lambda^{\mu})$ it follows that $x_i \to
0$, $i\in \mathcal{A}$, and $\lambda_i \to 0$, $i\in \mathcal{I}$, as
$\mu \to 0$. In consequence, $X_{\mathcal{I} \mathcal{I}}$ and
$\Lambdait_{\mathcal{A} \mathcal{A}}$ dominate the coefficients of the
third and fourth block of (\ref{eq:PDsyst_partitioned}) for
sufficiently small $\mu$ under strict complementarity.  Similarly $X_{\mathcal{A} \mathcal{A}} \inv \Lambdait_{\mathcal{A} \mathcal{A}}$ dominates the coefficients of the first block of (\ref{eq:SchurC_partioned}). Consequently, approximate solutions of $\Delta x_\mathcal{A}^N$ and $\Delta \lambda_\mathcal{I}^N$ can be obtained from the third and fourth block of (\ref{eq:PDsyst_partitioned}), and of $\Delta x_\mathcal{A}^N$ from the first block of (\ref{eq:SchurC_partioned}). These approximates can then be inserted into (\ref{eq:PDsyst_partitioned}), or (\ref{eq:SchurC_partioned}), to obtain a reduced system of size $| \mathcal{I} |$$ \times $$| \mathcal{I} |$ that involves $H_{\mathcal{I} \mathcal{I}}$. The solution of this system gives an approximation of $\Delta x_\mathcal{I}^N$. These observations together with Lemma~\ref{lemma:Order} and Lemma~\ref{lemma:dzBound} below provide the foundation for the results. The essence of Lemma~\ref{lemma:dzBound} is that the norm of the solution of (\ref{eq:PDsyst}) is bounded by a constant times $\mu$. 
\begin{lemma} \label{lemma:dzBound}
Under Assumption~\ref{ass1}, let $\mathcal{B}\left((x^*, \lambda^*),
  \delta\right)$ and $\muhat$ be defined by
Lemma~\ref{lemma:background:FpnonsingCont} and
Lemma~\ref{lemma:LipcPath} respectively.  For  $0< \mu \leq \muhat$ and $(x,\lambda)  \in \mathcal{B}((x^*, \lambda^*), \delta)$, let $(\Delta x^N, \Delta \lambda^N)$ be the solution of (\ref{eq:PDsyst}) with $\mu^+ = \sigma \mu$, where $0< \sigma < 1$. If $(x,\lambda)$ is sufficiently close to $(x^{\mu}, \lambda^{\mu}) \in \mathcal{B}((x^*, \lambda^*), \delta)$ such that $\| F_{\mu} (x,\lambda) \| = \mathcal{O}(\mu)$ then
\begin{equation*}
\left\| \left(
\Delta x^N ,\Delta \lambda^N
\right) \right\|  =  \mathcal{O}(\mu).
\end{equation*}
\end{lemma}
\begin{proof} By (\ref{eq:PDsyst}) it holds that 
\begin{align*}
\left\| \left(
\Delta x^N ,\Delta \lambda^N
\right) \right\|  & = \left\| F'(x,\lambda) \inv F_{\mu^+} (x, \lambda) \right\|  \\
& = \big\| F'(x,\lambda) \inv \left[ F_{\mu^+} (x, \lambda) - F_{\mu^+} (x^{\mu^+}, \lambda^{\mu^+}) \right] \big\| .
\end{align*}
Continuity of $F'$ on $\mathcal{B}((x^*, \lambda^*), \delta)$ implies that $F_{\mu^+}$ is Lipschitz continuous. Moreover, both $(x,\lambda)$ and $(x^{\mu^+}, \lambda^{\mu^+})$ belong to $\mathcal{B}((x^*, \lambda^*), \delta)$. Lipschitz continuity of $F_{\mu^+}$ and Lemma~\ref{lemma:background:FpnonsingCont} yield
\begin{equation*}
\left\| \left(
\Delta x^N ,\Delta \lambda^N
\right) \right\|  \leq M L_{F'}\big\| \left(
x , \lambda
\right) - (
x^{\mu^+} , \lambda^{\mu^+}
) \big\|.
\end{equation*}
Addition and subtraction of $(x^\mu, \lambda^\mu)$ in the norm of the right-hand side give
\begin{align*}
\left\| \left(
\Delta x^N ,\Delta \lambda^N
\right) \right\|& \leq M L_{F'} \big\| \left(
x , \lambda
\right) - \left(
x^{\mu} , \lambda^{\mu}
\right) + \left(
x^{\mu} , \lambda^{\mu}
\right) - (
x^{\mu^+} , \lambda^{\mu^+}
) \big\|  \\
& \leq M L_{F'} \left( \big\| \left(
x , \lambda
\right)  - \left(
x^{\mu} , \lambda^{\mu}
\right)  \big\|  +  \big\| \left(
x^{\mu}, \lambda^{\mu}
\right)  - (
x^{\mu^+} , \lambda^{\mu^+}
)  \big\|  \right) \\ 
& \leq M L_{F'} \left( \frac{1}{C_5} \big\| F_\mu(x,\lambda) \big\|  + C_4(1-\sigma)\mu \right) \\
& \leq M L_{F'}\left( \frac{C_1}{C_5} + C_4(1-\sigma) \right) \mu,
\end{align*}
where the second last inequality follows from Lemma~\ref{lemma:FmuBound} and Lipschitz continuity of $(x^{\mu}, \lambda^{\mu})$. The last inequality follows from $\| F_{\mu} (x,\lambda) \| = \mathcal{O}(\mu)$, i.e., there exists a constant $C_1>0$ such that $\| F_{\mu} (x,\lambda) \| \leq C_1 \mu$. 
\end{proof}
\subsection{Partial approximate solutions} \label{subsec:partialApp}
In this section we initially propose an approximate solution of $\Delta x_\mathcal{A}^N$ which originates from the Schur complement form (\ref{eq:SchurC_partioned}). This approximate solution will be labeled with superscript ``$S$'' due to its origin.  As $\mu \to 0$, the diagonal elements of the (1,1)-block become large and dominate the coefficients of the matrix under strict complementarity. In Proposition~\ref{prop:schurBased} we show that an approximate solution of $\Delta x_\mathcal{A}^N$ can be obtained by neglecting all off-diagonal coefficients in the the first block of  (\ref{eq:SchurC_partioned}).  Thereafter, we propose another approximate solution of $\Delta x_\mathcal{A}^N$, as well as one of $\Delta \lambda_\mathcal{I}^N$, which originate from the complementarity blocks of (\ref{eq:PDsyst_partitioned}). These approximate solutions will be labeled with superscript ``$C$'' due to their origin. The solutions are obtained by neglecting the coefficients in the complementarity blocks which approach zero as $\mu \to 0$, i.e., those in $ X_{\mathcal{A} \mathcal{A}}$ and $\Lambdait_{\mathcal{I} \mathcal{I}}$. The resulting partial approximate solutions are given below in Proposition~\ref{prop:compBased}. The essence of both results is that, under certain conditions, the asymptotic component error bounds are in the order of $\mu^2$.
Finally we motive and propose two Newton-like approaches which we later on investigate numerically.
\begin{proposition} \label{prop:schurBased}
Under Assumption~\ref{ass1}, let $\mathcal{B}\left((x^*, \lambda^*),
  \delta\right)$ and $\muhat$ be defined by
Lemma~\ref{lemma:background:FpnonsingCont} and
Lemma~\ref{lemma:LipcPath} respectively. For $(x,\lambda) \in \mathcal{B}((x^*, \lambda^*), \delta)$, let $(\Delta x^N, \Delta \lambda^N)$ be the solution of (\ref{eq:PDsyst}) with $\mu^+ = \sigma \mu$, where $0< \sigma < 1$. If the search direction components are defined as
\begin{equation} \label{eq:prop:schurBased:dx}
\Delta x_i^S = -\frac{ x_i [\nabla f(x)]_i - \mu^+}{ x_i \left[ \nabla^2 f(x)\right]_{ii} + \lambda_i},\qquad \qquad \qquad \qquad \quad \qquad \quad \> i=1,\dots,n,
\end{equation}
then
\begin{equation}\label{eq:prop:schurBased:dxErr}
\Delta x_i^S - \Delta x_i^N =  \frac{x_i}{ x_i \left[ \nabla^2 f(x)\right]_{ii} + \lambda_i}   \sum_{i \neq j} \left[ \nabla^2 f(x)\right]_{ij} \Delta x_j^N, \quad i=1,\dots,n.
\end{equation}
Assume in addition that $0 < \mu \le \muhat$ and $(x,\lambda)$ is sufficiently close to $(x^{\mu}, \lambda^{\mu})\in \mathcal{B}\left((x^*, \lambda^*), \delta\right)$ such that $\| F_{\mu} (x,\lambda) \| = \mathcal{O}(\mu)$. Then there exists $\mubar$, with
$0 <\mubar \le \muhat$, such that for $0 < \mu \leq \mubar$ it holds that
\begin{equation} \label{eq:prop:schurBased:denominator}
\frac{1}{x_i \left[ \nabla^2 f(x)\right]_{ii} + \lambda_i} = \Theta(1), \quad \qquad \quad \qquad \qquad \qquad \qquad \quad i=1,\dots,n,
\end{equation}
and
\begin{equation} \label{eq:prop:schurBasedOrdo}
| \Delta x_i^S - \Delta x^N_i | = \mathcal{O}(\mu^2), \qquad \qquad \qquad \qquad  \qquad \qquad \> i \in \mathcal{A}.
\end{equation}
\end{proposition}
\begin{proof}
The solution of ($\ref{eq:PDsyst}$) for $\Delta x^N$ is equivalent to the solution of (\ref{eq:SchurC_partioned}) where the $i$'th, $i=1,\dots,n$, row is
\begin{equation} \label{eq:prop:SchurBased:proof:eq0}
\sum_{j \neq i}^n \left[ \nabla^2 f(x) \right]_{ij} \Delta x_{j}^N + \left(\left[  \nabla^2 f(x) \right]_{ii} + \frac{\lambda_i}{x_i} \right) \Delta x_i^N = -\left( \left[\nabla f(x) \right]_i -  \frac{\mu^+}{x_i} \right).
\end{equation} 
If $ x_i \left[ \nabla^2 f(x)\right]_{ii} + \lambda_i \neq 0$ then (\ref{eq:prop:SchurBased:proof:eq0}) can be written as
\begin{align} \label{eq:prop:SchurBased:proof:eq1}
\Delta x_i^N & = \frac{x_i}{  x_i \left[ \nabla^2 f(x) \right]_{ii} + \lambda_i} \left(-\left( \left[\nabla f(x) \right]_i -  \frac{\mu^+}{x_i} \right) - \sum_{j \neq i}^n \left[ \nabla^2 f(x) \right]_{ij} \Delta x_{j}^N \right) \nonumber  \\ 
& = -\frac{ x_i [\nabla f(x)]_i - \mu^+}{ x_i \left[ \nabla^2 f(x)\right]_{ii} + \lambda_i} -\frac{x_i}{ x_i \left[ \nabla^2 f(x)\right]_{ii} + \lambda_i} \sum_{j \neq i}^n \left[ \nabla^2 f(x) \right]_{ij} \Delta x_{j}^N. 
\end{align}
Subtraction of (\ref{eq:prop:SchurBased:proof:eq1}) from (\ref{eq:prop:schurBased:dx}) gives (\ref{eq:prop:schurBased:dxErr}). By Lemma~\ref{lemma:Order} there exists $\mubar_3$, with \linebreak $0<\mubar_3 \leq \muhat$ such that the components of $(x,\lambda)$ satisfy (\ref{eq:lemma:Order}). Due to the boundedness of $f$ on $\mathcal{B}\left((x^*, \lambda^*), \delta \right)$ there exists $\mubar_4$, with $0<\mubar_4 \leq \muhat$, such that  (\ref{eq:prop:schurBased:denominator}) holds for $0<\mu \leq \mubar$ with $\mubar = \min\{\mubar_3, \mubar_4\}$. The result of (\ref{eq:prop:schurBasedOrdo}) follows from application of Lemma~\ref{lemma:Order} and Lemma~\ref{lemma:dzBound} to (\ref{eq:prop:schurBased:dxErr}) while taking (\ref{eq:prop:schurBased:denominator}) into account.
\end{proof}
The approximate solution $\Delta x^S$ in (\ref{eq:prop:schurBased:dx}) of Proposition~\ref{prop:schurBased} and its corresponding component error (\ref{eq:prop:schurBased:dxErr}) may be undefined for certain components. However, the essence is that the expressions are well-defined sufficiently close to the barrier trajectory for sufficiently small $\mu$, as shown by (\ref{eq:prop:schurBased:denominator}). In particular the component errors of (\ref{eq:prop:schurBasedOrdo}) are bounded by a constant times $\mu^2$ only for components $i \in \mathcal{A}$, although the expressions (\ref{eq:prop:schurBased:dx}) and associated errors (\ref{eq:prop:schurBased:dxErr}) hold for all components $i=1,\dots,n$. An approximate solution that is guaranteed to have all its components well-defined can be obtained from the complementarity blocks of (\ref{eq:PDsyst_partitioned}). This approximate solution, and in addition an approximate solution of $\Delta \lambda_\mathcal{I}^N$, are given in the proposition below. 
\begin{proposition} \label{prop:compBased}
Under Assumption~\ref{ass1}, let $\mathcal{B}\left((x^*, \lambda^*),
  \delta\right)$ and $\muhat$ be defined by
Lemma~\ref{lemma:background:FpnonsingCont} and
Lemma~\ref{lemma:LipcPath} respectively. For $(x,\lambda) \in \mathcal{B}((x^*, \lambda^*), \delta)$, let $(\Delta x^N, \Delta \lambda^N)$ be the solution of (\ref{eq:PDsyst}) with $\mu^+ = \sigma \mu$, where $0< \sigma < 1$. If the search direction components are defined as
\begin{subequations}\label{eq:prop:compBased:dxdlambda}
\begin{align} 
\Delta x_i^C &= - x_i + \frac{\mu^+}{\lambda_i}, \qquad \> & i=1,\dots,n, \label{eq:prop:compBased:dx} \\
\Delta \lambda_i^C &= - \lambda_i + \frac{\mu^+}{x_i}, \qquad \> & i=1,\dots,n,\label{eq:prop:compBased:dlambda}
\end{align}
\end{subequations}
then
\begin{subequations} \label{eq:prop:compBased:ErrBoth}
\begin{align} 
\Delta x_i^C - \Delta x_i^N & =  \frac{x_i}{\lambda_i} \Delta \lambda_i^N,  & i=1,\dots,n, \label{eq:prop:compBased:dxErr} \\
\Delta \lambda_i^C - \Delta \lambda_i^N & =   \frac{\lambda_i}{x_i} \Delta x_i^N,  & i=1,\dots,n.\label{eq:prop:compBased:dlambdaErr} 
\end{align}
\end{subequations}
Assume in addition that $0 < \mu \le \muhat$ and $(x,\lambda)$ is sufficiently close to $(x^{\mu}, \lambda^{\mu})\in \mathcal{B}\left((x^*, \lambda^*), \delta\right)$ such that $\| F_{\mu} (x,\lambda) \| = \mathcal{O}(\mu)$. Then there exists $\mubar$, with
$0 <\mubar \le \muhat$, such that for $0 < \mu \leq \mubar$ it holds that
\begin{subequations} \label{eq:prop:compBased:Ordo}
\begin{align} 
& | \Delta x_i^C - \Delta x_i^N | = \mathcal{O}(\mu^2), &  i \in \mathcal{A} , \label{eq:prop:compBased:dxOrdo} \\
& | \Delta \lambda_i^C - \Delta \lambda_i^N | =  \mathcal{O}(\mu^2), & i \in \mathcal{I}.\>\label{eq:prop:compBased:dlambdaOrdo} 
\end{align} 
\end{subequations}
\end{proposition} 
\begin{proof}
The $i$'th,  $i=1,\dots,n$, row in the second block of ($\ref{eq:PDsyst}$) is 
\begin{equation*}
\lambda_i \Delta x_i^N + x_i \Delta \lambda_i^N = - \lambda_i x_i + \mu^+,
\end{equation*}
For $x_i > 0$, $\lambda_i > 0$, $i,\dots, n$, it holds that
\begin{subequations} \label{eq:prop:compBased:eq3}
\begin{align} 
\Delta x_i^N  &= - x_i + \frac{\mu^+}{\lambda_i} - \frac{x_i}{\lambda_i} \Delta \lambda_i^N,
\label{eq:prop:compBased:proof:eq1} \\
\Delta \lambda_i^N &= - \lambda_i + \frac{\mu^+}{x_i} - \frac{\lambda_i}{x_i} \Delta x_i^N.  \label{eq:prop:compBased:proof:eq2}
\end{align} 
\end{subequations}
Subtraction of (\ref{eq:prop:compBased:proof:eq1}) from  (\ref{eq:prop:compBased:dx}) and subtraction of (\ref{eq:prop:compBased:proof:eq2}) from  (\ref{eq:prop:compBased:dlambda}) gives (\ref{eq:prop:compBased:dxErr}) and (\ref{eq:prop:compBased:dlambdaErr}) respectively. By Lemma~\ref{lemma:Order} there exists $\mubar$, with $0<\mubar \leq \muhat$ such that the components of $(x,\lambda)$ satisfy (\ref{eq:lemma:Order}) for $0<\mu \leq \mubar$. The result of (\ref{eq:prop:compBased:Ordo}) then follows from application of Lemma~\ref{lemma:dzBound} to (\ref{eq:prop:compBased:eq3}) while taking (\ref{eq:lemma:Order}) into account.
\end{proof}
The expressions for $\Delta x^C_i$ and $\Delta \lambda^C_i$,  (\ref{eq:prop:compBased:dx}) and (\ref{eq:prop:compBased:dlambda}) respectively, and their associated component errors (\ref{eq:prop:compBased:dxErr}) and (\ref{eq:prop:compBased:dlambdaErr}) respectively, hold for all components. The essence of the results in Proposition~\ref{prop:compBased} is that the component errors are bounded by a constant times $\mu^2$ only for certain components. Specifically, for $\Delta x^C_i$, $i \in \mathcal{A}$, and $\Delta \lambda_i^C$, $i \in \mathcal{I}$. Both $\Delta x^S_i$ given by (\ref{eq:prop:schurBased:dx}) and $\Delta x^C_i$ given by (\ref{eq:prop:compBased:dx}) provide approximate solutions of $\Delta x_i^N$, $i \in \mathcal{A}$, with similar asymptotic error bounds. Note that the order of the approximation error, $\| \Delta x_\mathcal{A} - \Delta x_\mathcal{A}^N \|$, is maintained even if some components $i \in \mathcal{A}$ are updated with (\ref{eq:prop:schurBased:dx}) and others with (\ref{eq:prop:compBased:dx}). Which expression to use can hence be chosen individually for each index $i \in \mathcal{A}$. The factors in front of $\Delta x_i^N$ and $\Delta \lambda_i^N$, $i =1,\dots, n$, in the component errors of (\ref{eq:prop:schurBased:dxErr}) and (\ref{eq:prop:compBased:ErrBoth}) respectively may be used as an indicator for which of the approximations to use, and also whether either expression is likely to provide an accurate approximation. 
Note also that the approximate solution $\Delta x^C$ given by (\ref{eq:prop:compBased:dx}) does not take into account any information from the first block equation of (\ref{eq:PDsyst}), whereas $\Delta x^S$ given by (\ref{eq:prop:schurBased:dx}) includes information from both blocks.

Provided that the norm of the combined steps $\Delta x_\mathcal{A}^N$ and $\Delta \lambda_\mathcal{I}^N$ is not smaller than the approximation error, then stepping in these components with (\ref{eq:prop:schurBased:dx}) or (\ref{eq:prop:compBased:dxdlambda}) give a vector which is not further from the Newton iterate. This is formalized in Proposition~\ref{prop:Astep} below. 
\begin{proposition} \label{prop:Astep}
Under Assumption~\ref{ass1}, let $\mathcal{B}\left((x^*, \lambda^*),
  \delta\right)$ and $\muhat$ be defined by
Lemma~\ref{lemma:background:FpnonsingCont} and
Lemma~\ref{lemma:LipcPath} respectively. For $(x,\lambda) \in \mathcal{B}((x^*, \lambda^*), \delta)$, define \linebreak $(x_+^{N},\lambda_+^{N})=(x,\lambda)+ (\Delta x^N,\Delta \lambda^N)$ where $(\Delta x^N, \Delta \lambda^N)$ is the solution of (\ref{eq:PDsyst}) with $\mu^+ = \sigma \mu$, where $0< \sigma < 1$. Moreover, let $(x_+, \lambda_+) = (x,\lambda) + (\Delta x, \Delta \lambda)$ where 
\begin{equation} \label{eq:prop:partialAstep}
\Delta x_i = \begin{cases} \Delta x_i^S \mbox{ or } \Delta x_i^C & i \in \mathcal{A}, \\
0 & i \in \mathcal{I},  \end{cases} \qquad \Delta \lambda_i = \begin{cases} 0 & i \in \mathcal{A}, \\
\Delta \lambda_i^C & i \in \mathcal{I}, \end{cases} 
\end{equation}
with  $\Delta x_i^C$, $\Delta \lambda_i^C$ and $\Delta x_i^S$ given by (\ref{eq:prop:compBased:dxdlambda}) and (\ref{eq:prop:schurBased:dx}) respectively. Assume that \linebreak $0 < \mu \le \muhat$, $\| (\Delta x_\mathcal{A}^N, \Delta \lambda_\mathcal{I}^N) \|
=\Omega ( \mu^\gamma )$ for $0\le \gamma < 2$, and $(x,\lambda)$ is sufficiently close to $(x^{\mu}, \lambda^{\mu})\in \mathcal{B}\left((x^*, \lambda^*), \delta\right)$ such that $\| F_{\mu} (x,\lambda) \| = \mathcal{O}(\mu)$. Then there exists $\mubar$, with
$0 <\mubar \le \muhat$, such that for $0 < \mu \leq \mubar$ it holds that
\begin{equation} \label{eq:prop:partialAstepCloser}
\| (x_+^N,\lambda_+^N) - (x_+,\lambda_+)  \|  \le \| (x_+^N,\lambda_+^N) - (x,\lambda)  \|.
\end{equation}
\end{proposition} 
\begin{proof}
With $(\Delta x, \Delta \lambda)$ defined as in (\ref{eq:prop:partialAstep}) of the proposition it holds that
\begin{align}
&\| (x_+^N,\lambda_+^N) - (x_+,\lambda_+)  \| ^2 - \| (x_+^N,\lambda_+^N) - (x,\lambda)  \| ^2 = \nonumber \\
&  \| (\Delta x^N - \Delta x, \Delta \lambda^N - \Delta \lambda) \| ^2 - \| (\Delta x^N, \Delta \lambda^N) \| ^2 = \nonumber \\
&  \| (\Delta x_\mathcal{A}^N - \Delta x_\mathcal{A}, \Delta \lambda_\mathcal{I}^N - \Delta \lambda_\mathcal{I}) \| ^2 - \| (\Delta x_\mathcal{A}^N, \Delta \lambda_\mathcal{I}^N) \|^2. \label{eq:lemma:stepEq1}
\end{align}
By Proposition~\ref{prop:schurBased} and Proposition~\ref{prop:compBased} there exists $\mubar_5$ and $\mubar_6$ respectively, with $0 < \mubar_i \le \muhat$, $i=5,6$ such that for  $\Delta x_i$ equal to $\Delta x_i^S \mbox{ or } \Delta x_i^C$ it holds that \linebreak $| \Delta x_i - \Delta x_i^N | =\mathcal{O}(\mu^2)$, $i \in \mathcal{A}$, for $0 < \mu \leq \min\{ \mubar_5, \mubar_6 \}$. By Proposition~\ref{prop:compBased} it also holds that $| \Delta \lambda_i^C - \Delta \lambda_i^N | = \mathcal{O}(\mu^2)$, $i \in \mathcal{I}$, for $0 < \mu \leq \mubar_6$. Hence, for $0<\mu \leq \min \{ \mubar_5, \mubar_6 \}$, there exist constants $C_7>0$ and $C_8>0$, where $C_8$ comes from the condition $\| (\Delta x_\mathcal{A}^N, \Delta \lambda_\mathcal{I}^N) \|=\Omega ( \mu^\gamma )$, $0\leq \gamma < 2$, such that
\begin{equation} \label{eq:lemma:stepEq2}
\| (\Delta x_\mathcal{A}^N - \Delta x_\mathcal{A}, \Delta \lambda_\mathcal{I}^N - \Delta \lambda_\mathcal{I}) \| ^2 - \| (\Delta x_\mathcal{A}^N, \Delta \lambda_\mathcal{I}^N) \| ^2    \leq C_7^2 \mu^4 - C_8^2 \mu^{2\gamma}.
\end{equation}
The right-hand side of (\ref{eq:lemma:stepEq2}) is non-positive for $0<\mu \leq (C_8/C_7)^{\frac{1}{2-\gamma}}$, $0\leq \gamma < 2$. Combining (\ref{eq:lemma:stepEq1})-(\ref{eq:lemma:stepEq2}) with $\mubar = \min\{\mubar_5, \mubar_6, (C_8/C_7)^{\frac{1}{2-\gamma}} \}$ gives the result.
\end{proof}
The partial approximate solution (\ref{eq:prop:partialAstep}) of Proposition~\ref{prop:Astep} is computationally inexpensive compared to solving (\ref{eq:PDsyst}). In consequence, (\ref{eq:prop:partialAstepCloser}) motivates the study of Newton-like approaches which make use of (\ref{eq:prop:partialAstep}). We will construct two such approaches where the idea is to utilize the intermediate iterate 
\begin{equation} \label{eq:zE}
(x^{E}, \lambda^{E}) = (x+\Delta x^E, \lambda +\Delta \lambda^E),
\end{equation} 
with $(\Delta x^E, \Delta \lambda^E)$ as in (\ref{eq:prop:partialAstep}). It is thus only the active components of $x$ and inactive components of $\lambda$ that is updated in the step to $(x^E, \lambda^E)$. For simplicity we describe the ideas for unit step length, in practice the iterates would be required to be strictly feasible.

The first approach is based on the fact that solving a Newton system from the iterate $(x^E,\lambda^E)$ provides potential improvement, provided that $(x^E, \lambda^E)$ is strictly feasible and lies in $\mathcal{B}((x^*, \lambda^*), \delta)$. A full iteration in the approach consists of the \emph{approximate intermediate step} (\ref{eq:zE}) together with the solution of 
\begin{equation}
F'(x^E,\lambda^E) \begin{bmatrix}
\Delta x \\  \Delta \lambda
\end{bmatrix} = - F_\mu (x^E, \lambda^E),   \label{eq:modNewtonAstep}\\
\end{equation}
and the step $(x^E + \Delta x, \lambda^E + \Delta \lambda)$. 

The idea of the second approach is to update the coefficients in the complementarity blocks of the matrix in (\ref{eq:PDsyst}). The approach may hence under strict complementarity be interpreted as an \emph{approximate higher-order} method.  A full iteration consists of the step (\ref{eq:zE}), the solution of 
\begin{equation}
\begin{bmatrix}
H & -I \\ \Lambdait^E & X^E
\end{bmatrix} \begin{bmatrix}
\Delta x \\  \Delta \lambda
\end{bmatrix}  = - F_\mu (x, \lambda), \label{eq:modNewtonHO}
\end{equation}
where $\Lambdait^E = \mbox{diag}(\lambda^E)$ and $X^E = \mbox{diag}(x^E)$, together with the step $(x+\Delta x, \lambda +\Delta \lambda)$. The approach may hence also be interpreted as a modified Newton method where the Jacobian of each Newton system is altered. 

Numerical results for the \emph{approximate intermediate step} and the \emph{approximate higher-order} approach are shown in Section~\ref{sec:numRes}. The results are for bound-constrained quadratic optimization problems where strict complementarity typically does not hold. 
The complexity of each iteration in both approaches is the same as with Newton's method. The hope is thus to reduce the total number of iteration necessary for convergence.  See the work by  Gondzio and Sobral \cite{GonSob2019} for quasi-Newton approaches for quadratic problems where each iteration is inexpensive in comparison to the approaches above. 
\subsection{Full approximate solution} \label{subsec:fullApp}
In this section we propose approximate solutions of (\ref{eq:PDsyst}) that, in the considered framework, have an asymptotic error bound in the order of $\mu^2$. The full approximate solutions are obtained by utilizing either of the partial approximate solutions of $\Delta x_\mathcal{A}^N$ in Proposition~\ref{prop:schurBased} or Proposition~\ref{prop:compBased} while exploiting structure in the systems that arise. Specifically, suppose that an approximate $\Delta x_{\mathcal{A}}$ is given, e.g., $\Delta x^S_\mathcal{A}$ given by (\ref{eq:prop:schurBased:dx}) or $\Delta x^C_\mathcal{A}$ given by (\ref{eq:prop:compBased:dx}). Insertion of the approximate $\Delta x_{\mathcal{A}}$ into (\ref{eq:PDsyst_partitioned}) yields
\begin{equation}\label{eq:redPDsyst_partitioned}
\begin{bmatrix}
H_{\mathcal{A} \mathcal{I}}  & -I_{\mathcal{A} \mathcal{A}} &  \\
H_{\mathcal{I} \mathcal{I}}  &  & -I_{\mathcal{I}\mathcal{I}} \\
 & X_{\mathcal{A} \mathcal{A}} &  \\
\Lambdait_{\mathcal{I} \mathcal{I}} &  & X_{\mathcal{I} \mathcal{I}} 
\end{bmatrix} \begin{bmatrix} \Delta x_\mathcal{I}^{ls} \\ \Delta \lambda_\mathcal{A}^{ls} \\ \Delta \lambda_\mathcal{I}^{ls}
\end{bmatrix} = -\begin{bmatrix}
\nabla f(x)_\mathcal{A} - \lambda_\mathcal{A}+ H_{\mathcal{A} \mathcal{A}} \Delta x_{\mathcal{A}} \\
\nabla f(x)_\mathcal{I} - \lambda_\mathcal{I}+ H_{\mathcal{I} \mathcal{A}} \Delta x_{\mathcal{A}} \\
\Lambdait_{\mathcal{A} \mathcal{A}} X_{\mathcal{A} \mathcal{A}} e - \mu e + \Lambdait_{\mathcal{A} \mathcal{A}} \Delta x_{\mathcal{A}} \\
\Lambdait_{\mathcal{I} \mathcal{I}} X_{\mathcal{I} \mathcal{I}} e - \mu e
\end{bmatrix},
\end{equation}
where the solution is given the superscript ``$ls$'' since it will lead to a least squares system. The second and fourth block of (\ref{eq:redPDsyst_partitioned}) provide unique solutions of $\Delta x^{ls}_\mathcal{I}$ and $\Delta \lambda^{ls}_\mathcal{I}$ which satisfy
\begin{equation} \label{eq:Redx2syst}
\begin{bmatrix}
H_{\mathcal{I} \mathcal{I}}  & -I_{\mathcal{I}\mathcal{I}} \\
\Lambdait_{\mathcal{I} \mathcal{I}} & X_{\mathcal{I} \mathcal{I}} 
\end{bmatrix} \begin{bmatrix} \Delta x_\mathcal{I}^{ls} \\ \Delta \lambda_\mathcal{I}^{ls}
\end{bmatrix} = -\begin{bmatrix}
\nabla f(x)_\mathcal{I} - \lambda_\mathcal{I} + H_{\mathcal{I} \mathcal{A}} \Delta x_{\mathcal{A}}  \\
\Lambdait_{\mathcal{I} \mathcal{I}} X_{\mathcal{I} \mathcal{I}} e - \mu e
\end{bmatrix}.
\end{equation}
The solution of (\ref{eq:Redx2syst}) can be obtained by first solving with the Schur complement of $X_{\mathcal{I} \mathcal{I}}$
\begin{equation}  \label{eq:Redx2syst_schur}
\left(H_{\mathcal{I} \mathcal{I}} + X_{\mathcal{I} \mathcal{I}} \inv \Lambdait_{\mathcal{I} \mathcal{I}}  \right) \Delta x_\mathcal{I}^{ls}  = -\left( \nabla f(x)_\mathcal{I}  + H_{\mathcal{I} \mathcal{A}} \Delta x_{\mathcal{A}} \right)  + \mu X_{\mathcal{I} \mathcal{I}} \inv  e,
\end{equation}
and then 
\begin{equation} \label{eq:fullApprox:dlambdaI}
\Delta \lambda_{\mathcal{I}}^{ls} = - X_{\mathcal{I} \mathcal{I}} \inv \left( \Lambdait_{\mathcal{I} \mathcal{I}} X_{\mathcal{I} \mathcal{I}} e - \mu e \right)  - X_{\mathcal{I} \mathcal{I}} \inv \Lambdait_{\mathcal{I} \mathcal{I}} \Delta x_{\mathcal{I}}^{ls}. 
\end{equation}
Note that (\ref{eq:Redx2syst_schur}) can also be obtained by insertion of the given $\Delta x_\mathcal{A}$ into the second block of (\ref{eq:SchurC_partioned}). The matrix of (\ref{eq:Redx2syst_schur}) is by Assumption~\ref{ass1} a symmetric positive definite $( |\mathcal{I}|$$  \times $$|\mathcal{I}| )$-matrix. Moreover, the matrix does not become increasingly ill-conditioned due to large elements in $X\inv \Lambdait$, under strict complementarity as $\mu \to 0$, in contrast to the matrix of (\ref{eq:SchurC_partioned}). The remanding part of the solution of (\ref{eq:redPDsyst_partitioned}), that is $\Delta \lambda_\mathcal{A}^{ls}$, is then given by
\begin{equation} \label{eq:LS:overdet}
\begin{bmatrix}
-I_{\mathcal{A} \mathcal{A}} \\
X_{\mathcal{A} \mathcal{A}} 
\end{bmatrix} \Delta \lambda_{\mathcal{A}}^{ls} = -\begin{bmatrix}
\nabla f(x)_\mathcal{A} + \lambda_\mathcal{A} + H_{\mathcal{A}   \mathcal{A}}  \Delta x_{\mathcal{A}}+ H_{\mathcal{A} \mathcal{I}} \Delta x_{\mathcal{I}}^{ls}\\
\Lambdait_{\mathcal{A} \mathcal{A}} X_{\mathcal{A} \mathcal{A}} e - \mu e +\Lambdait_{\mathcal{A} \mathcal{A}} \Delta x_{\mathcal{A}}
\end{bmatrix}.
\end{equation}
If the approximate $\Delta x_\mathcal{A}$ is exact, i.e., if $\Delta x_\mathcal{A} = \Delta x_\mathcal{A}^N$, then $\Delta x_\mathcal{I}^{ls}=  \Delta x_\mathcal{I}^N$ by (\ref{eq:Redx2syst_schur}). In consequence, the over-determined system (\ref{eq:LS:overdet}) has a unique solution that satisfies all equations, i.e., $\Delta \lambda_\mathcal{A}^{ls}$ is the corresponding part of the solution to (\ref{eq:PDsyst}). The solutions corresponding to the first and second block equation of (\ref{eq:LS:overdet}) will be assigned superscripts ``$b$'' and ``$-$'' respectively. These are given by
\begin{subequations}
\begin{equation} \label{eq:fullApprox:dlambdaAsimple1}
\Delta \lambda_{\mathcal{A}}^b = \nabla f(x)_\mathcal{A} + \lambda_\mathcal{A} + H_{\mathcal{A}   \mathcal{A}}  \Delta x_{\mathcal{A}}+ H_{\mathcal{A} \mathcal{I}} \Delta x_{\mathcal{I}}^{ls},
\end{equation}
and
\begin{equation}  \label{eq:fullApprox:dlambdaAsimple2}
\Delta \lambda_{\mathcal{A}}^{-} = - \lambda_\mathcal{A} + \mu X_{\mathcal{A} \mathcal{A}} \inv e + X_{\mathcal{A} \mathcal{A}} \inv \Lambdait_{\mathcal{A} \mathcal{A}} \Delta x_{\mathcal{A}}.
\end{equation}
\end{subequations} 
Alternatively, $\Delta \lambda_{A}^{ls}$ can be obtained as the least squares solution of (\ref{eq:LS:overdet}) that is
\begin{align} 
\Delta \lambda_{\mathcal{A}}^{ls} =  & \left( I_{\mathcal{A} \mathcal{A}} +
X_{\mathcal{A} \mathcal{A}}^2 \right) \inv \Big[
\nabla f(x)_\mathcal{A} + \lambda_\mathcal{A} + H_{\mathcal{A}  \mathcal{A}} \Delta x_{\mathcal{A}} + H_{\mathcal{A} \mathcal{I}} \Delta x_{\mathcal{I}}^{ls} \nonumber \\
& - X_{\mathcal{A} \mathcal{A}} \left( \Lambdait_{\mathcal{A} \mathcal{A}} X_{\mathcal{A} \mathcal{A}} e - \mu e +\Lambdait_{\mathcal{A} \mathcal{A}} \Delta x_{\mathcal{A}} \right) \Big]. \label{eq:fullApprox:dlambdaAls}
\end{align}
In Theorem~\ref{thm:simplifiedCase} it is shown that, under certain conditions, both $\Delta \lambda_{\mathcal{A}}^b$ given by (\ref{eq:fullApprox:dlambdaAsimple1}) and $\Delta \lambda_{\mathcal{A}}^{ls}$ given by (\ref{eq:fullApprox:dlambdaAls}) can be used to approximate $\Delta \lambda_\mathcal{A}^N$ without affecting the order of the asymptotic error. Note however that this is not true for $\Delta \lambda_{\mathcal{A}}^{-}$ given by (\ref{eq:fullApprox:dlambdaAsimple2}) due to the last term that contains $X_{\mathcal{A} \mathcal{A}} \inv$ in combination with approximation error.
\begin{theorem} \label{thm:simplifiedCase}
Under Assumption~\ref{ass1}, let $\mathcal{B}\left((x^*, \lambda^*),
  \delta\right)$ and $\muhat$ be defined by \linebreak
Lemma~\ref{lemma:background:FpnonsingCont} and
Lemma~\ref{lemma:LipcPath} respectively. For  $0< \mu \leq \muhat$ and $(x,\lambda) \in \mathcal{B}((x^*, \lambda^*), \delta)$, let $(\Delta x^N, \Delta \lambda^N)$ be the solution of (\ref{eq:PDsyst}) with $\mu^+ = \sigma \mu$, where $0< \sigma < 1$. Moreover, let the search direction components be defined as
\begin{equation*}
\Delta x_i = \begin{cases} \Delta x_i^S \textrm{ or } \Delta x_i^C &   i \in \mathcal{A}, \\  \Delta x_i^{ls} & i \in \mathcal{I},
\end{cases} \quad \Delta \lambda_i = \begin{cases}
\Delta \lambda_i^{ls} \textrm{ or } \Delta \lambda_i^b &  i \in \mathcal{A},\\ 
\Delta \lambda_i^{ls} \textrm{ or } \Delta \lambda_i^{C}  & i \in \mathcal{I}, 
\end{cases}
\end{equation*}
where $\Delta x_i^S$ is given by (\ref{eq:prop:schurBased:dx}), $\Delta x_i^C$ by (\ref{eq:prop:compBased:dx}),  $\Delta x_i^{ls}$ by (\ref{eq:Redx2syst_schur}), $\Delta \lambda_i^{ls}$ by  (\ref{eq:fullApprox:dlambdaAls}),  $\Delta \lambda_i^b$ by (\ref{eq:fullApprox:dlambdaAsimple1}),  $\Delta \lambda_i^{ls}$ by (\ref{eq:fullApprox:dlambdaI}) and $\Delta \lambda_i^C$ by  (\ref{eq:prop:compBased:dlambda}).
Assume that $0 < \mu \le \muhat$ and $(x,\lambda)$ is sufficiently close to $(x^{\mu}, \lambda^{\mu})\in \mathcal{B}\left((x^*, \lambda^*), \delta\right)$ such that $\| F_{\mu} (x,\lambda) \| = \mathcal{O}(\mu)$. Then there exists $\mubar$, with $0 <\mubar \le \muhat$, such that for $0 < \mu \leq \mubar$ it holds that
\begin{equation*}
\left\| (\Delta x, \Delta \lambda) - (\Delta x^N , \Delta \lambda^N) \right\|   = \mathcal{O}(\mu^2).
\end{equation*}
\end{theorem}
\begin{proof}
Similarly as in the proof of Proposition~\ref{prop:Astep}. By Proposition~\ref{prop:schurBased} and Proposition~\ref{prop:compBased} there exists $\mubar_5$ and $\mubar_6$ respectively, with $0 < \mubar_i \le \muhat$, $i=5,6$ such that for  $\Delta x_i$ equal to $\Delta x_i^S \mbox{ or } \Delta x_i^C$ it holds that $| \Delta x_i - \Delta x_i^N | =\mathcal{O}(\mu^2)$, $i \in \mathcal{A}$, for $0 < \mu \leq \min\{ \mubar_5, \mubar_6 \}$. In consequence it follows that $\| \Delta x_\mathcal{A} - \Delta x_\mathcal{A}^N \| = \mathcal{O}(\mu^2)$, $0 < \mu \leq \min\{ \mubar_5, \mubar_6 \}$. By Proposition~\ref{prop:compBased} it also holds that $| \Delta \lambda_i^C - \Delta \lambda_i^N | = \mathcal{O}(\mu^2)$, $i \in \mathcal{I}$, $0 < \mu \leq \mubar_6$. The backward error with $\Delta x^{ls}_\mathcal{I}$ as given in (\ref{eq:Redx2syst_schur}) is
\[\Delta x^{ls}_\mathcal{I} - \Delta x_\mathcal{I}^N = - \left(H_{\mathcal{I} \mathcal{I}} + X_{\mathcal{I} \mathcal{I}} \inv \Lambdait_{\mathcal{I} \mathcal{I}}  \right) \inv H_{\mathcal{I} \mathcal{A}} \left( \Delta x_{\mathcal{A}} - \Delta  x_{\mathcal{A}}^N \right), \]
which gives
\begin{align*}
\left\| \Delta x^{ls}_\mathcal{I} - \Delta x_\mathcal{I}^N \right\| &  \leq \| \left(H_{\mathcal{I} \mathcal{I}} + X_{\mathcal{I} \mathcal{I}} \inv \Lambdait_{\mathcal{I} \mathcal{I}}  \right)\inv \|  \|  H_{\mathcal{I} \mathcal{A}}  \| \|  \Delta x_{\mathcal{A}} - \Delta  x_{\mathcal{A}}^N \| \\ & 
\leq \frac{1}{\sigma_{min}\left(H_{\mathcal{I} \mathcal{I}} + X_{\mathcal{I} \mathcal{I}} \inv \Lambdait_{\mathcal{I} \mathcal{I}}  \right) }  \|  H_{\mathcal{I} \mathcal{A}}  \| \|  \Delta x_{\mathcal{A}} - \Delta  x_{\mathcal{A}}^N \| .
\end{align*}
Due to the assumption on $f$ the elements of $ H_{\mathcal{I} \mathcal{A}}$ are bounded. Moreover, the smallest singular value of $H_{\mathcal{I} \mathcal{I}} + X_{\mathcal{I} \mathcal{I}} \inv \Lambdait_{\mathcal{I} \mathcal{I}}$ is bounded away from zero since the matrix is positive definite by Assumption~\ref{ass1}. Hence it follows that $\left\| \Delta x^{ls}_\mathcal{I} - \Delta x_\mathcal{I}^N \right\|  = \mathcal{O}(\mu^2)$, $0 < \mu \leq \min\{ \mubar_5, \mubar_6 \}$. Note that $\Delta \lambda_\mathcal{I}^N$ is the solution of (\ref{eq:fullApprox:dlambdaI}) with $\Delta x_\mathcal{I}^N$. Subtraction of (\ref{eq:fullApprox:dlambdaI}), with $\Delta x_\mathcal{I}^N$, from (\ref{eq:fullApprox:dlambdaI}) with the approximated solution $\Delta x_\mathcal{I}^{ls}$ gives $
\Delta \lambda^{ls}_\mathcal{I} - \Delta \lambda^{N}_\mathcal{I} = -X_{\mathcal{I} \mathcal{I}} \inv \Lambdait_{\mathcal{I} \mathcal{I}} \left( \Delta x_{\mathcal{I}}^{ls} -  \Delta x^{N}_\mathcal{I} \right)$, and hence
\begin{equation*}
\| \Delta \lambda^{ls}_\mathcal{I} - \Delta \lambda^{N}_\mathcal{I} \| \leq \| X_{\mathcal{I} \mathcal{I}} \inv  \Lambdait_{\mathcal{I} \mathcal{I}} \|  \| \Delta x_{\mathcal{I}}^{ls} -  \Delta x^{N}_\mathcal{I} \|.
\end{equation*}
By Lemma~\ref{lemma:Order} it holds that $\| X_{\mathcal{I} \mathcal{I}} \inv  \Lambdait_{\mathcal{I} \mathcal{I}} \| =  \mathcal{O}(\mu)$, $0 < \mu \le \max\{ \mubar_5, \mubar_6 \}$. With $\left\| \Delta x^{ls}_\mathcal{I} - \Delta x_\mathcal{I}^N \right\|  = \mathcal{O}(\mu^2)$, $0 < \mu \leq \min\{ \mubar_5, \mubar_6 \}$ it then follows that $\| \Delta \lambda^{ls}_{\mathcal{I}} -  \Delta \lambda_{\mathcal{I}}^N \|  = \mathcal{O}(\mu^3)$, and also $| \Delta \lambda^{ls}_i - \Delta \lambda_i^N| = \mathcal{O}(\mu^3)$, $i \in \mathcal{I}$, $0 < \mu \leq \min\{ \mubar_5, \mubar_6 \}$. Similarly, $\Delta \lambda^N_\mathcal{A}$ is the solution to (\ref{eq:fullApprox:dlambdaAls}) with $\Delta x^N_\mathcal{A}$ and $\Delta x^N_\mathcal{I}$. Subtraction of (\ref{eq:fullApprox:dlambdaAls}), with $\Delta x^N_\mathcal{A}$ and $\Delta x^N_\mathcal{I}$, from (\ref{eq:fullApprox:dlambdaAls}) with the approximated solutions gives   
\begin{align*}
\left( I_{\mathcal{A} \mathcal{A}} +
X_{\mathcal{A} \mathcal{A}}^2 \right) \left( \Delta \lambda^{ls}_{\mathcal{A}} -  \Delta \lambda_{\mathcal{A}}^N \right) 
& =  \left(  H_{\mathcal{A} \mathcal{A}} - X_{\mathcal{A} \mathcal{A}} \Lambdait_{\mathcal{A} \mathcal{A}}  \right)  \left(\Delta x_{\mathcal{A}} - \Delta x_{\mathcal{A}}^N \right) \\
& \quad +  H_{\mathcal{A} \mathcal{I}} \left( \Delta x^{ls}_{\mathcal{I}}-  \Delta x_{\mathcal{I}}^N \right).
\end{align*}
The the largest singular value of $\left( I_{\mathcal{A} \mathcal{A}} +
X_{\mathcal{A} \mathcal{A}}^2 \right)\inv$ is bounded by $1$ and hence
\begin{equation*}
\|\Delta \lambda^{ls}_{\mathcal{A}} -  \Delta \lambda_{\mathcal{A}}^N \| \leq \left( \| H_{\mathcal{A} \mathcal{A}}  \| + \| X_{\mathcal{A} \mathcal{A}} \Lambdait_{\mathcal{A} \mathcal{A}} \| \right) \| \Delta x_{\mathcal{A}} - \Delta x_{\mathcal{A}}^N \| + \|  H_{\mathcal{A} \mathcal{I}} \| \|\Delta x^{ls}_{\mathcal{I}}-  \Delta x_{\mathcal{I}}^N \|.
\end{equation*}
The elements of $ H_{\mathcal{A} \mathcal{A}}$ and $H_{\mathcal{A} \mathcal{I}}$ are bounded and by Lemma~\ref{lemma:Order} it holds that $\| X_{\mathcal{A} \mathcal{A}} \Lambdait_{\mathcal{A} \mathcal{A}} \| =  \mathcal{O}(\mu)$, $0 < \mu \le \max\{ \mubar_5, \mubar_6 \}$. Thus it follows that $\| \Delta \lambda^{ls}_{\mathcal{A}} -  \Delta \lambda_{\mathcal{A}}^N \|  = \mathcal{O}(\mu^2)$, and also $| \Delta \lambda^{ls}_i - \Delta \lambda_i^N| = \mathcal{O}(\mu^2)$, $i \in \mathcal{A}$, $0 < \mu \leq \min\{ \mubar_5, \mubar_6 \}$. Similarly, (\ref{eq:fullApprox:dlambdaAsimple1}) gives the backward error
\[
 \Delta \lambda^b_{\mathcal{A}} -  \Delta \lambda_{\mathcal{A}}^N  = H_{\mathcal{A}   \mathcal{A}}  \left( \Delta x_{\mathcal{A}} - \Delta x_{\mathcal{A}}^N \right)+ H_{\mathcal{A} \mathcal{I}} \left( \Delta x^{ls}_{\mathcal{I}} -  \Delta x_{\mathcal{I}}^N \right). 
\]
Hence
\[
\|  \Delta \lambda^b_{\mathcal{A}} -  \Delta \lambda_{\mathcal{A}}^N \|  \leq \| H_{\mathcal{A}   \mathcal{A}} \| \|  \Delta x_{\mathcal{A}} - \Delta x_{\mathcal{A}}^N \|+  \|  H_{\mathcal{A} \mathcal{I}} \| \|  \Delta x^{ls}_{\mathcal{I}} -  \Delta x_{\mathcal{I}}^N \|, 
\]
from which it follows that $\|  \Delta \lambda^b_{\mathcal{A}} -  \Delta \lambda_{\mathcal{A}}^N \| = \mathcal{O}(\mu^2)$, and also $|  \Delta \lambda^b_i -  \Delta \lambda_i^N | = \mathcal{O}(\mu^2)$, $i \in \mathcal{A}$, $0 < \mu \leq \min\{ \mubar_5, \mubar_6 \}$. Thus the result holds for $\mubar = \min\{ \mubar_5, \mubar_6 \}$. 
\end{proof}
Information is discarded in the calculation of the components $\Delta x^S_i$, $\Delta x_i^C$, $i\in \mathcal{A}$, and $\Delta \lambda_i^C$, $i \in\mathcal{I}$, with (\ref{eq:prop:schurBased:dx}) and (\ref{eq:prop:compBased:dxdlambda}) respectively. The equations for the approximate solution in Theorem~\ref{thm:simplifiedCase} show that it is essential to obtain a good approximate solution of $\Delta x_\mathcal{A}^N$. It is the error in the approximate solution of $\Delta x_\mathcal{A}^N$ that propagates through the suggested solutions labeled with $ls$ and $b$. In contrast to all other components of the proposed full approximate solution, $\Delta \lambda_i^{ls}$, $i\in \mathcal{I}$, actually have asymptotic component error bounds in the order of magnitude $\mu^3$, as can be seen in the proof of  Theorem~\ref{thm:simplifiedCase}.

In general the active and inactive sets at the optimal solution are
unknown and have to be estimated as the iterations proceed. The
quality of the approximate solution of $\Delta x_\mathcal{A}^N$ will
hence also depend on these estimates. There is a trade-off when
estimating the set of active constraints. A restrictive strategy may
lead to a more accurate approximate $\Delta x_\mathcal{A}$.
However, it increases the cardinality of the inactive set and in
consequence the size of the system (\ref{eq:Redx2syst_schur}) that
needs to be solved at each iteration. In theory, the cardinality of the
inactive set is determined by the number of inactive constraints at the solution of the specific problem, whereas in practice it is determined by the estimate. The size of the system that needs to be solved at each iteration may thus range from $0$ to $n$. A restrictive strategy may also increase the size of some coefficients in the diagonal of the matrix of (\ref{eq:Redx2syst_schur}), or (\ref{eq:genCase:Redx2syst_schur}) in the general case, which may increase the condition number. A generous strategy on the other hand, decreases the size of the system that has to be solved but may increase the error in the approximate $\Delta x_\mathcal{A}$, which then propagates to other components of the approximate solution.
In the ideal case with the true inactive set, then (\ref{eq:Redx2syst})
and (\ref{eq:Redx2syst_schur}) are composed of the inactive parts of
(\ref{eq:PDsyst}), or equivalently (\ref{eq:PDsyst_partitioned}), and
(\ref{eq:SchurC_partioned}) respectively. Consequently, the inactive part of the Schur complement in (\ref{eq:Redx2syst_schur}) does not become increasingly ill-conditioned due to $\mu$ approaching zero, in contrast to the complete Schur complement in (\ref{eq:SchurC_partioned}). 
However, in practice the behavior will be dependent on an estimate of the inactive set. 

Note also that the system that needs to be solved for the full approximate solution has the same structure as the original one. In consequence, our analysis may be interpreted
in the framework of previous work on stability and effects of
finite-precision arithmetic for interior-point methods,
e.g., \cite{Wri2001,Wri1995s,ForGilShi1996,Wri1999}. In the case of quadratic problems, see also \cite{MorSim2017}.

To increase the comprehensibility of the work we have described the theoretical foundation for problems on the form (\ref{eq:P}). Analogous results for problems on the more general form (\ref{eq:NLP}) together with complementary remarks are given in Appendix~\ref{app:genCase}. 
\section{Numerical results} \label{sec:numRes}

As an initial numerical study we consider convex quadratic optimization problems with lower and upper bounds. In particular, randomly generated problems and a selection from the corresponding class in the CUTEst test collection \cite{GouOrbToi15}. The minimizers of the randomly generated problems satisfy strict complementarity, whereas the minimizers of the CUTEst problems typically do not. 
The simulations were done in \texttt{Julia} and all systems of linear equations were solved by its built-in solver. Moreover, the benchmark problems were initially processed using the packages \texttt{CUTEst.jl} and \texttt{NLPmodels.jl} by Orban and Siqueira \cite{OrbSigSmoothOpt2019}. 

The purpose of the first part of this section is to compare the proposed approximate solutions in Theorem~\ref{thm:genCase}. The intent is also to give a rough indication of how the approximation errors develop for practical values of $\mu$. A setting is considered where the vector $(x,\lambda)$, that satisfies $ \| F_{\mu}(x,\lambda) \| < \mu$, is found by an interior-point method. Thereafter, $\mu$ is decreased by a factor $\sigma = 0.1$ to $\mu^+=\sigma \mu$ and the approximate solution of (\ref{eq:PDsyst}) is calculated. This procedure was then repeated for different values of $\mu$. 
Mean errors with one standard deviation error bars for the proposed approximate solutions are shown in Figure~\ref{fig:randProbs}. 
As mentioned, the results are for the approximate solutions given in Theorem~\ref{thm:genCase} of Appendix~\ref{app:genCase} since the problems in general include lower and upper bounds. In order to avoid double subscripts in the approximates, we have throughout this section omitted the second subscript. Furthermore, $\Delta x^{S}_\mathcal{A}$ was used in the equations which require an initial approximation of $\Delta x_\mathcal{A}^N$. 
Figure~\ref{fig:randProbs} also shows the mean improvement in terms of the measure $\| F_{\mu^+} \|$ for two new iterates $(x_+^S, \lambda^S_+)$ and $(x_+^C, \lambda^C_+)$ defined by
\begin{equation*}
(x_+^{ S,C }, \lambda_+^{ S,C }) =  ( x+\alpha^P \Delta x, \lambda+\alpha^D \Delta \lambda ), \> \> \> (\Delta x, \Delta \lambda)  = \left( \begin{pmatrix} \Delta x_\mathcal{A}^{ S,C } \\ \Delta x_\mathcal{I}^{ls} \end{pmatrix}, \begin{pmatrix} \Delta \lambda_\mathcal{A}^{ls} \\ \Delta \lambda_\mathcal{I}^{ls} \end{pmatrix}  \right),
\end{equation*}
with step lengths $\alpha^P$ and $\alpha^D$ as in Algorithm~\ref{alg:IPMnewt}. Specifically, the search direction is composed of (\ref{eq:prop:genCase:schurBased:dx}) or (\ref{eq:prop:genCase:compBased:dx}) combined with (\ref{eq:genCase:Redx2syst_schur}), (\ref{eq:genCase:fullApprox:dlambdaAls}) and (\ref{eq:genCase:fullApprox:dlambdaI}).
The figure also contains the mean improvement of the Newton iterate $(x^{N}_+, \lambda^{N}_+)$, which is defined analogously. 
The results are for $10^2$ randomly generated
problems, with $10^3$ variables, whose minimizers satisfy (\ref{eq:genCase:OptCond}). For each problem, both the specific bounds as well as the specific active and inactive constraints were chosen by random. Moreover, the elements of the Hessian were uniformly distributed around zero with a sparsity level corresponding to approximately 40 percent non-zero elements. The condition numbers were in the order of magnitude $10^7$-$10^{10}$ and the largest singular values in the order of magnitude of $10^3$.
\begin{figure}[H]
\centering
\parbox{6.0cm}{
\includegraphics[width=6.0cm]{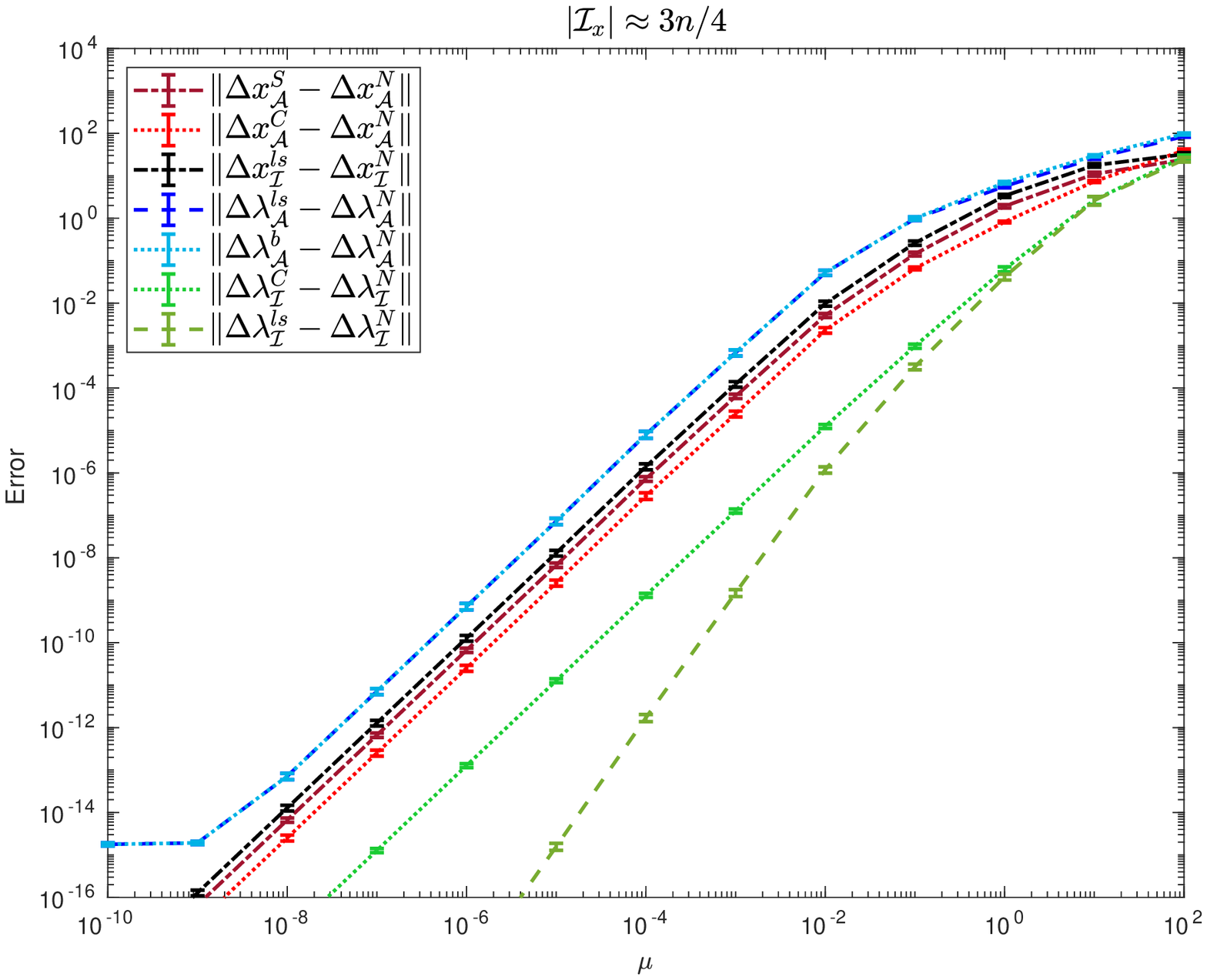}}
\hspace{-5mm}
\begin{minipage}{6.0cm}
\includegraphics[width=6.0cm]{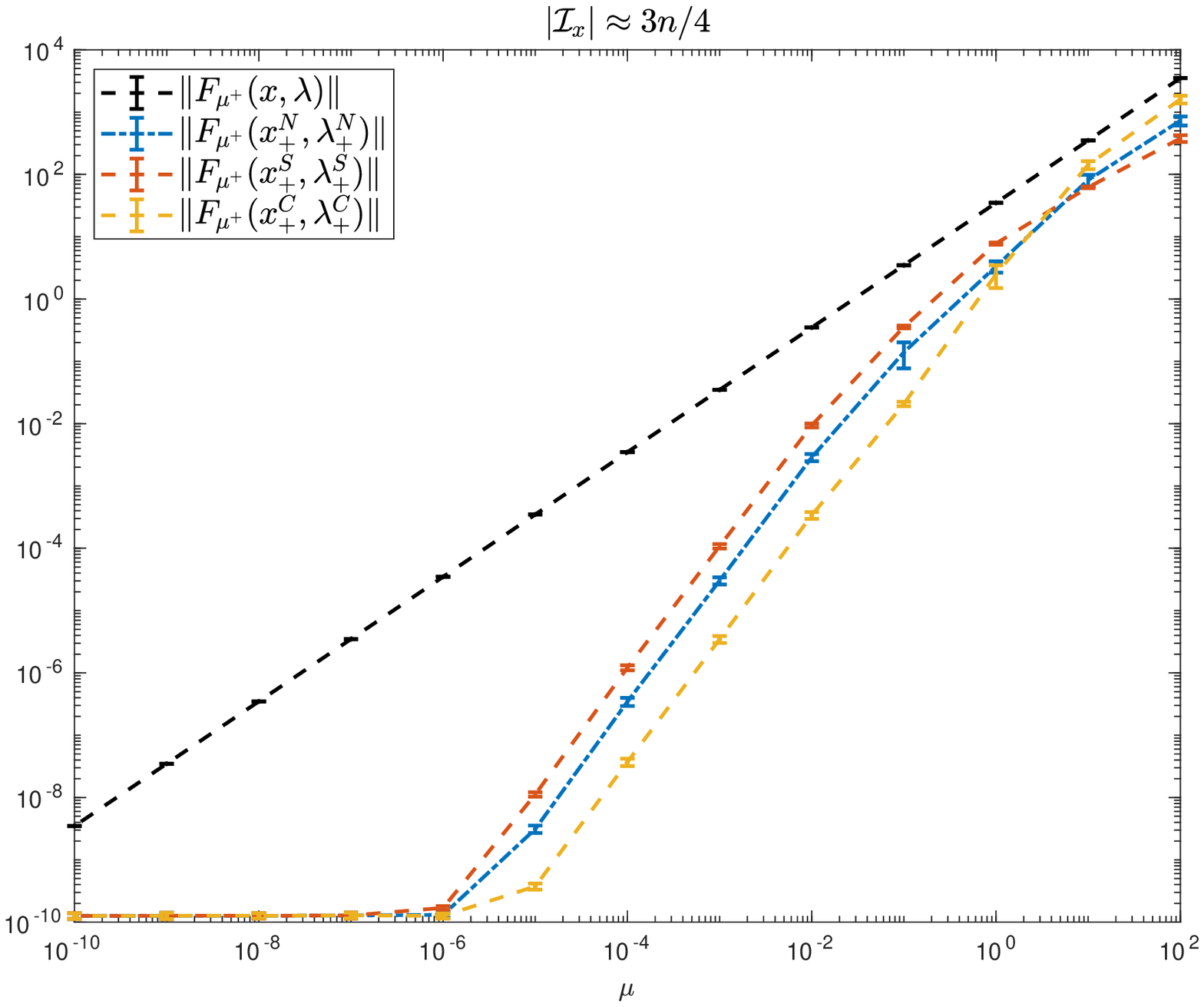} 
\end{minipage}
\parbox{6.0cm}{
\includegraphics[width=6.0cm]{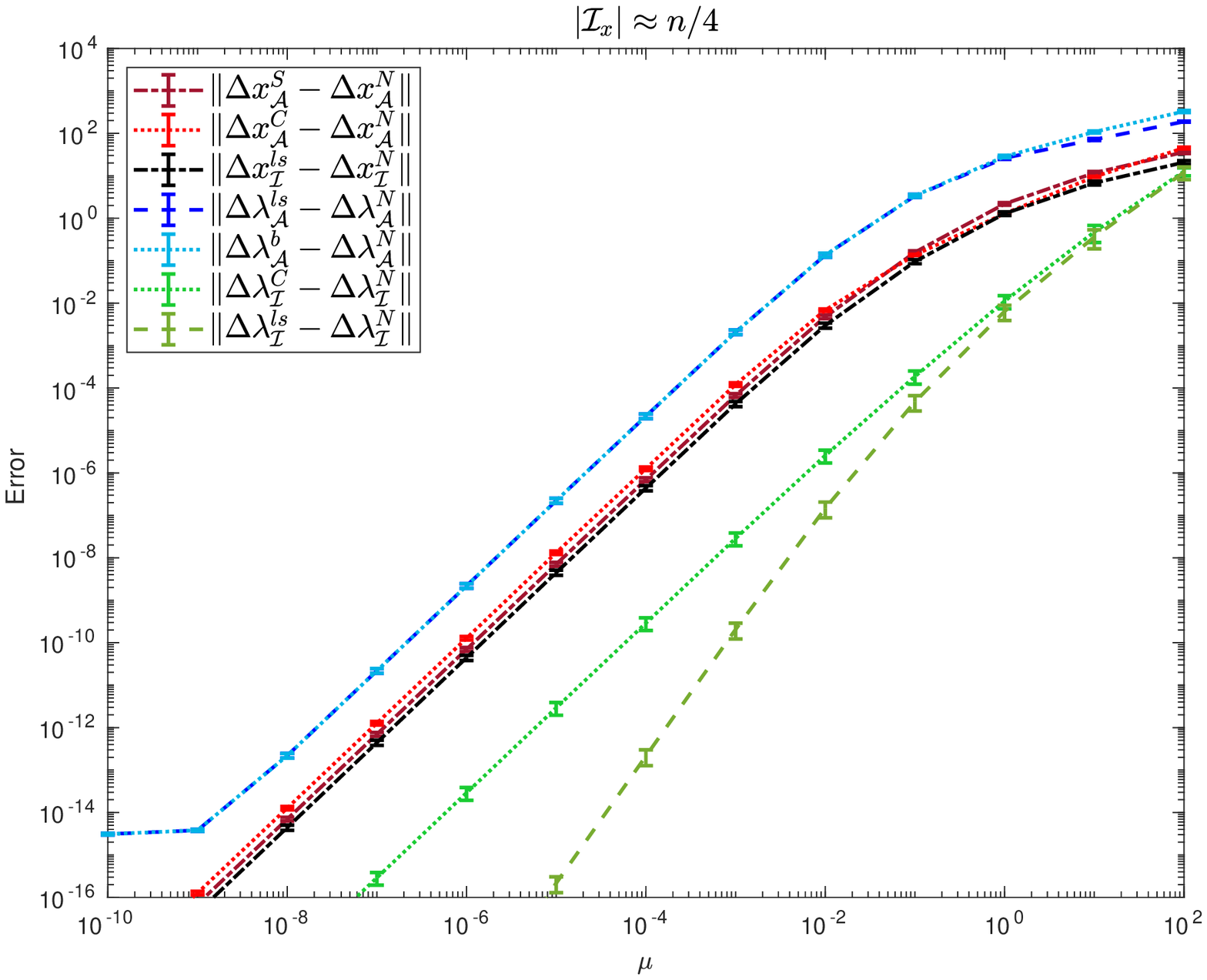}}
\hspace{-5mm}
\begin{minipage}{6.0cm}
\includegraphics[width=6.0cm]{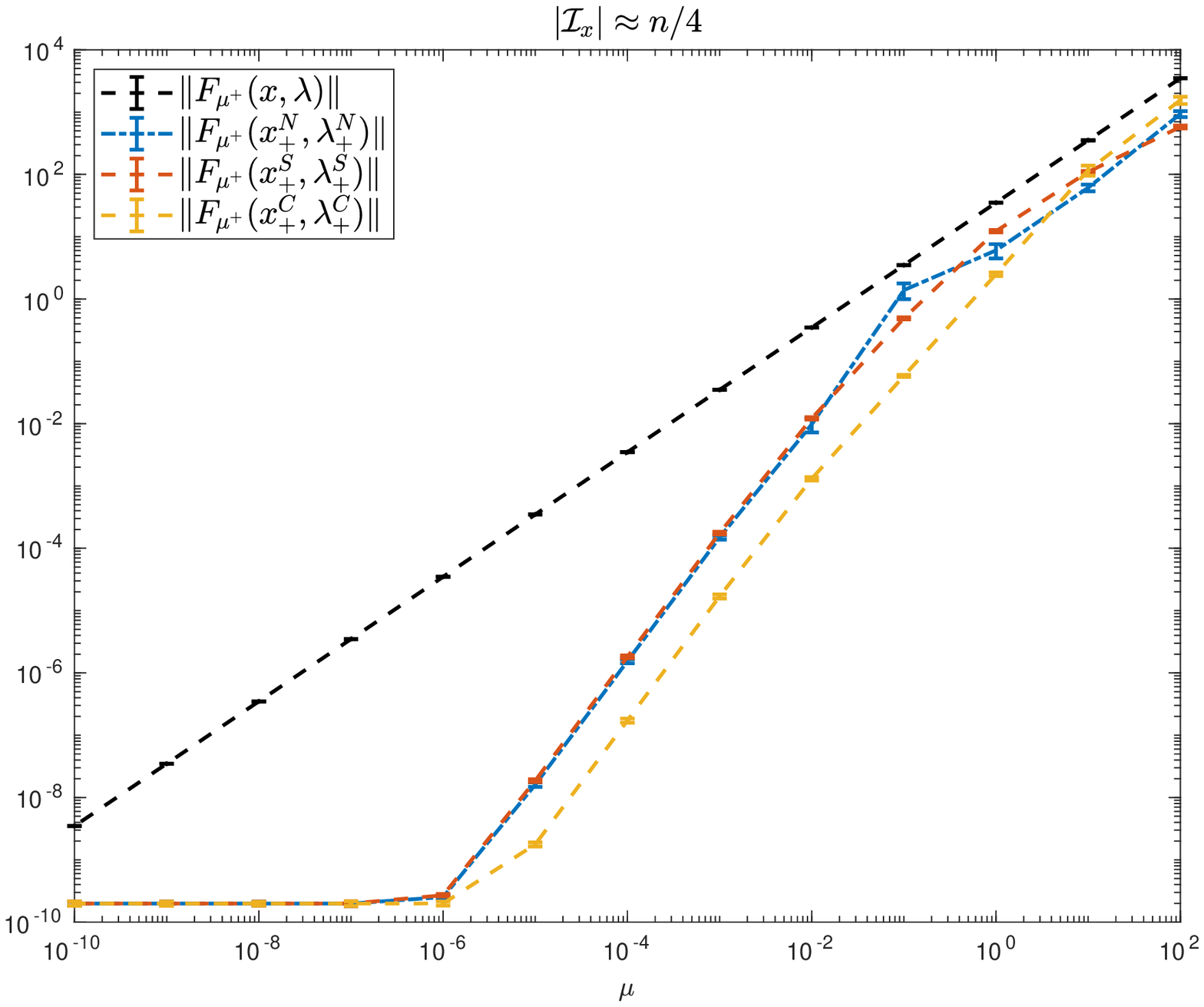} 
\end{minipage}
\caption{Mean approximation error and mean progress with measure $\| F_{\mu^+}\|$ with one standard deviation error bars for randomly generated quadratic problems. The top and the bottom correspond to problems where approximately $3/4$ and $1/4$ of the variables respectively are inactive at the solution.}
\label{fig:randProbs}
\end{figure}
The least accurate approximate solutions in
Figure~\ref{fig:randProbs} are those corresponding to active $\lambda$
and inactive $x$. This is anticipated as their error bounds rely more
heavily on the size of the elements of $H$. Moreover, it can be seen
that $\Delta \lambda_\mathcal{I}^{ls}$ is favorable over $\Delta
\lambda_\mathcal{I}^C$ for the problems considered. This is anticipated as $\Delta \lambda_\mathcal{I}^{ls}$ has asymptotic error bounds in the order of magnitude $\mu^3$, in contrast to the bounds corresponding to $\Delta \lambda_\mathcal{I}^C$ which is in the order of magnitude $\mu^2$, as mentioned in Section~\ref{subsec:fullApp}. In general,
Figure~\ref{fig:randProbs} gives an indication of what equation that
is favorable for each partial approximate solution if one is to be
chosen. However, as mentioned, more sophisticated choices can be made
by carefully considering the known quantities in the individual error
terms for specific components. The right side of
Figure~\ref{fig:randProbs} shows that the iterates $(x_+^S, \lambda^S_+)$ and $(x_+^C, \lambda^C_+)$ perform similar to
$(x_+^N, \lambda^N_+)$ in terms of the measure $\|F_{\mu^+} \|$ for a wide range of $\mu$. The error bars show that the results are not sensitive to changes in specific bounds, which of the constraints are active/inactive or different initial solutions. 
Numerical simulations have shown, as the theory also predicts, that the results can be improved (or dis-improved) by increasing (or decreasing) the size of the coefficients of the matrix $H$ as well as its sparsity level.

Next we show results for a selection of problems in the CUTEst test
collection in the analogous setting. In the problems with variable options, the number of primal variables, $n_x$, was typically chosen to approximately $5000$, resulting in a total number of primal-dual variables in the order of $10^4$. The number of primal variables of each specific problem is shown in Table~\ref{table:CUTEstData1}. Each problem was initially solved by an
interior-point method with stopping criterion $\| F_{0} (x,\lambda )\|
< 10^{-14}$, i.e., the first-order optimality conditions given by
(\ref{eq:genCase:Fmu}) for $\mu= 0$. This was to determine the selection of problems as well as estimates of the active and inactive sets. Problems with an unconstrained optimal solution or an optimal solution with only degenerate active constraints were not considered. In the first case the proposed approximate solutions are equivalent to the true solution. In the second case it is not clear how to deduce active/inactive sets. A constraint was considered as active if the corresponding variable was closer than $10^{-10}$ to its bound. An active constraint was deemed degenerate if the corresponding multiplier value was below $10^{-6}$. An exception was made for problem \texttt{ODNAMUR}, due to its larger size, for which the tolerances above were increased by a factor of $10^1$ and $10^2$. 
Figure~\ref{fig:CUTEstDiffMu} shows mean errors with the approximate solutions of Theorem~\ref{thm:genCase} on each CUTEst problem. The results are for three different values of $\mu$ with 10 different random initial solutions. The figure also shows the measure $\| F_{\mu^+}  \| $ for $(x,\lambda)$, $(x_{+}^S, \lambda_{+}^S)$, $(x_{+}^C, \lambda_{+}^C)$ and $ (x^{N}_+, \lambda^{N}_+)$. Simulations with the set estimation heuristic above have shown that the behavior of the approximate solution varies in three different regions depending on $\mu$. These regions are approximately, $[10^2,10^{-2})$, $[10^{-2}, 10^{-6}]$ and $(10^{-6},0)$. The $\mu$-values in Figure~\ref{fig:CUTEstDiffMu} correspond to representative behavior in their respective region. The problems are ordered such that the fraction of estimated active constraints at the solution decreases from  left to right.
\begin{figure}[H]
\centering
\parbox{6.1cm}{
\includegraphics[width=6.3cm]{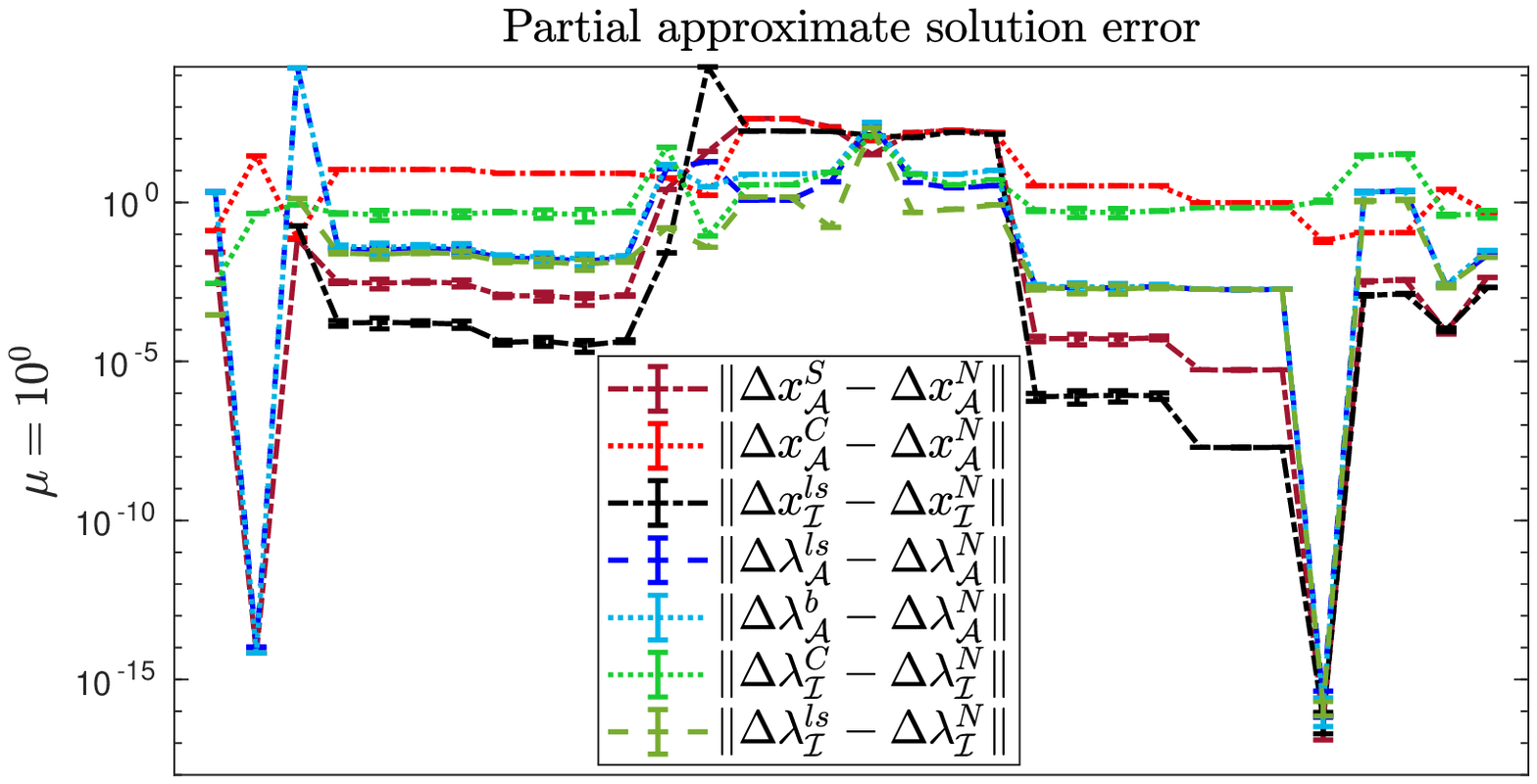}}
\hspace{-5mm}
\begin{minipage}{6.1cm}
\includegraphics[width=6.3cm]{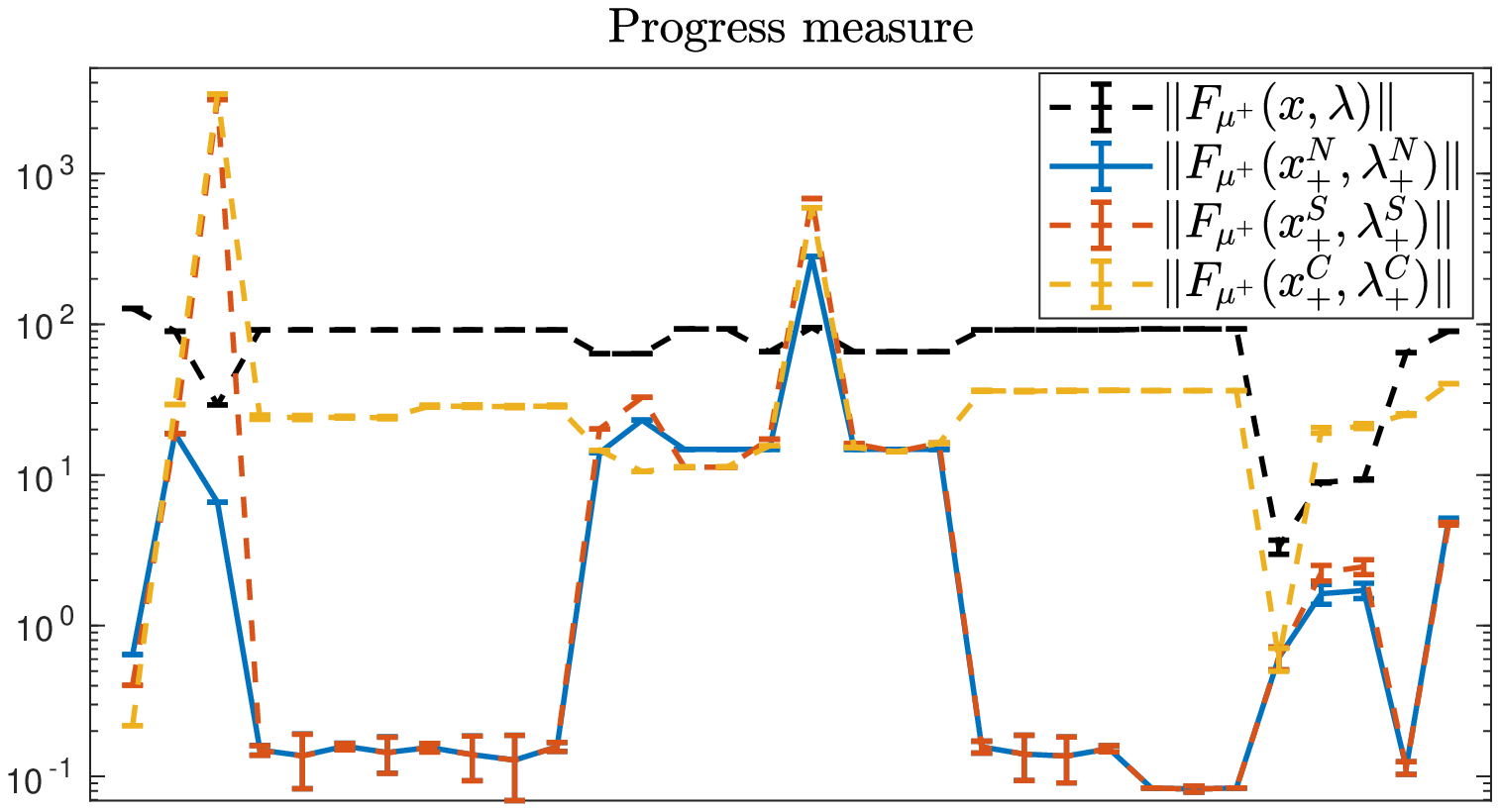} 
\end{minipage}
\parbox{6.1cm}{
\includegraphics[width=6.3cm]{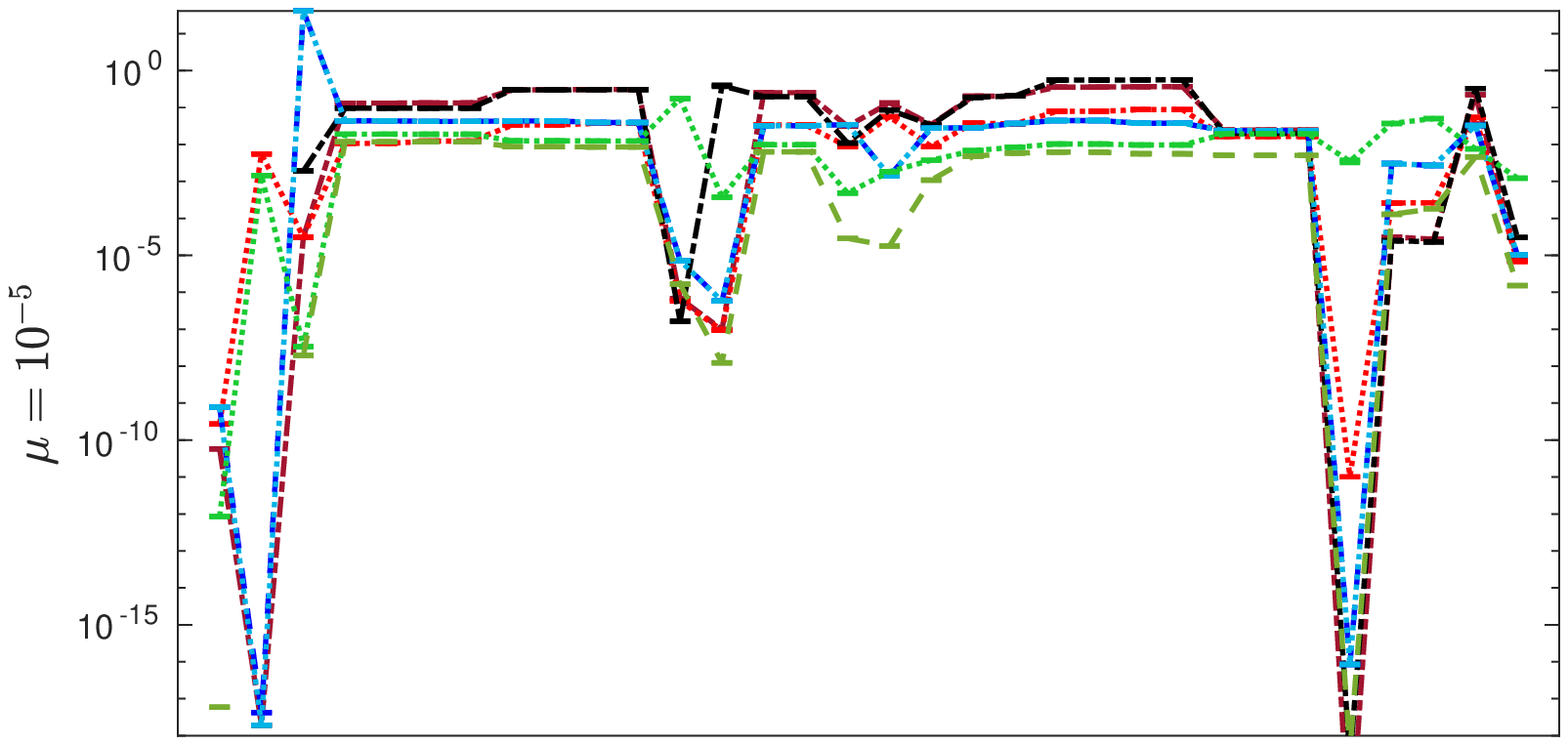}}
\hspace{-5mm}
\begin{minipage}{6.1cm}
\includegraphics[width=6.3cm]{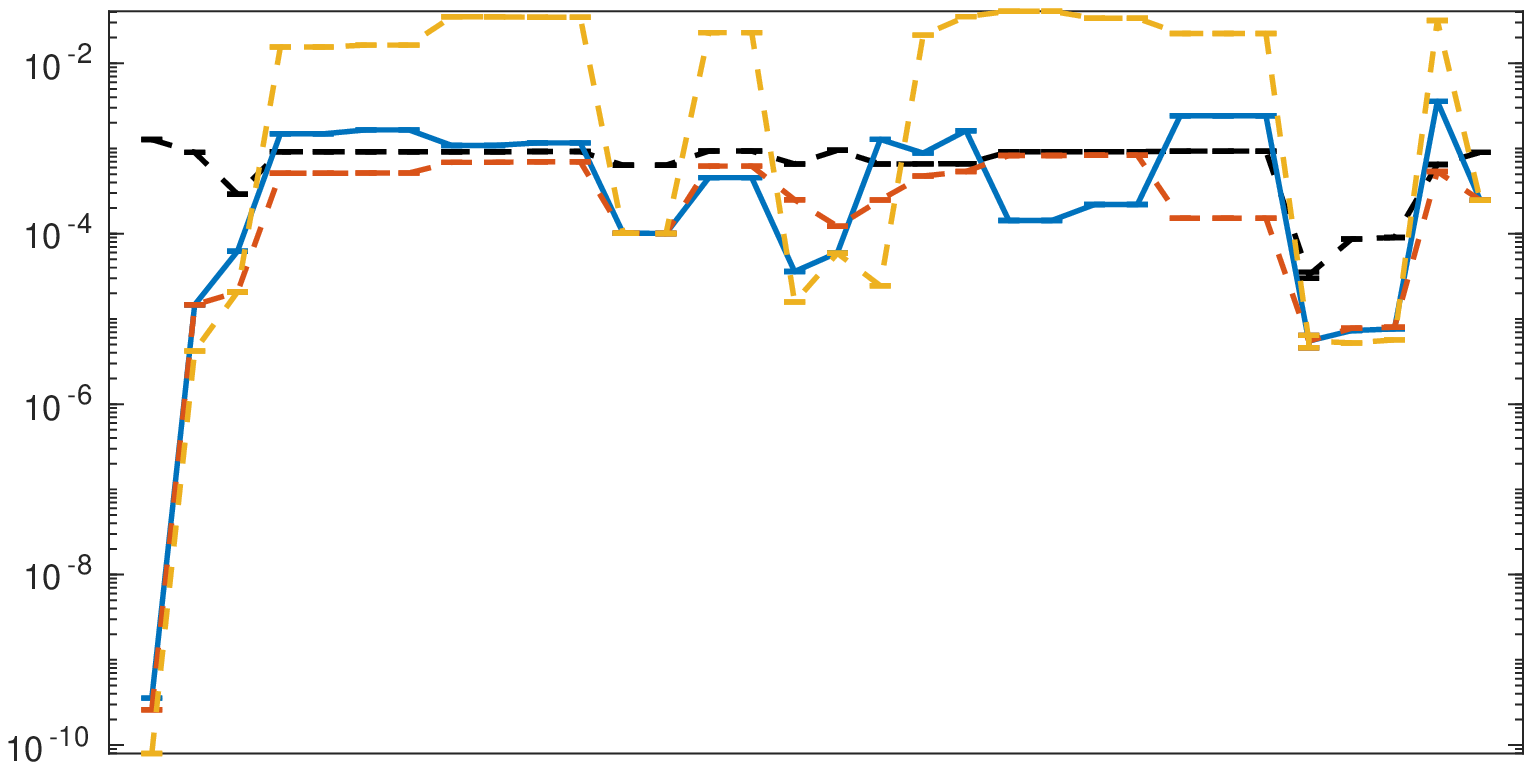} 
\end{minipage}
\parbox{6.1cm}{
\includegraphics[width=6.3cm]{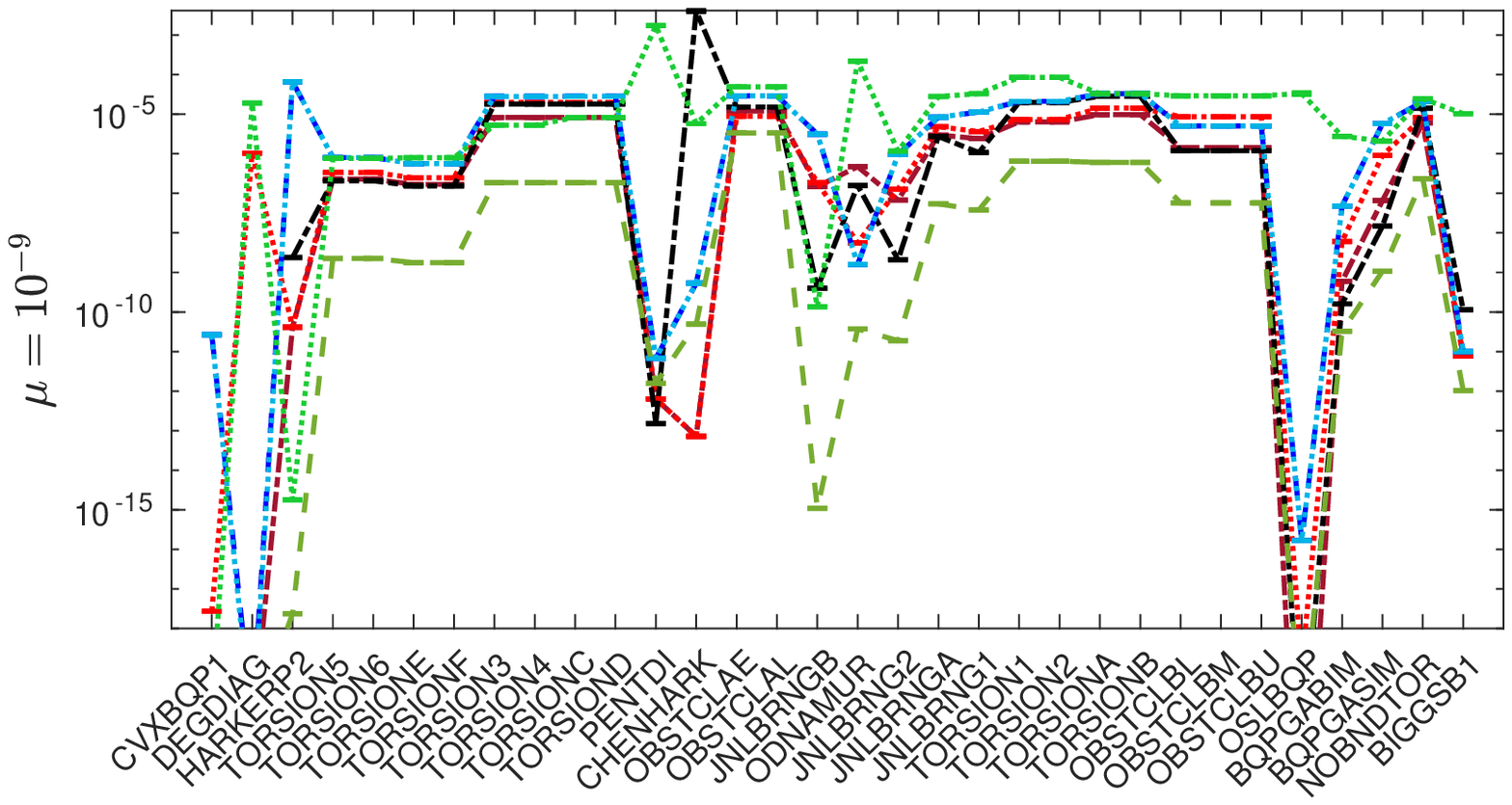}}
\hspace{-5mm} 
\begin{minipage}{6.1cm}
\includegraphics[width=6.3cm]{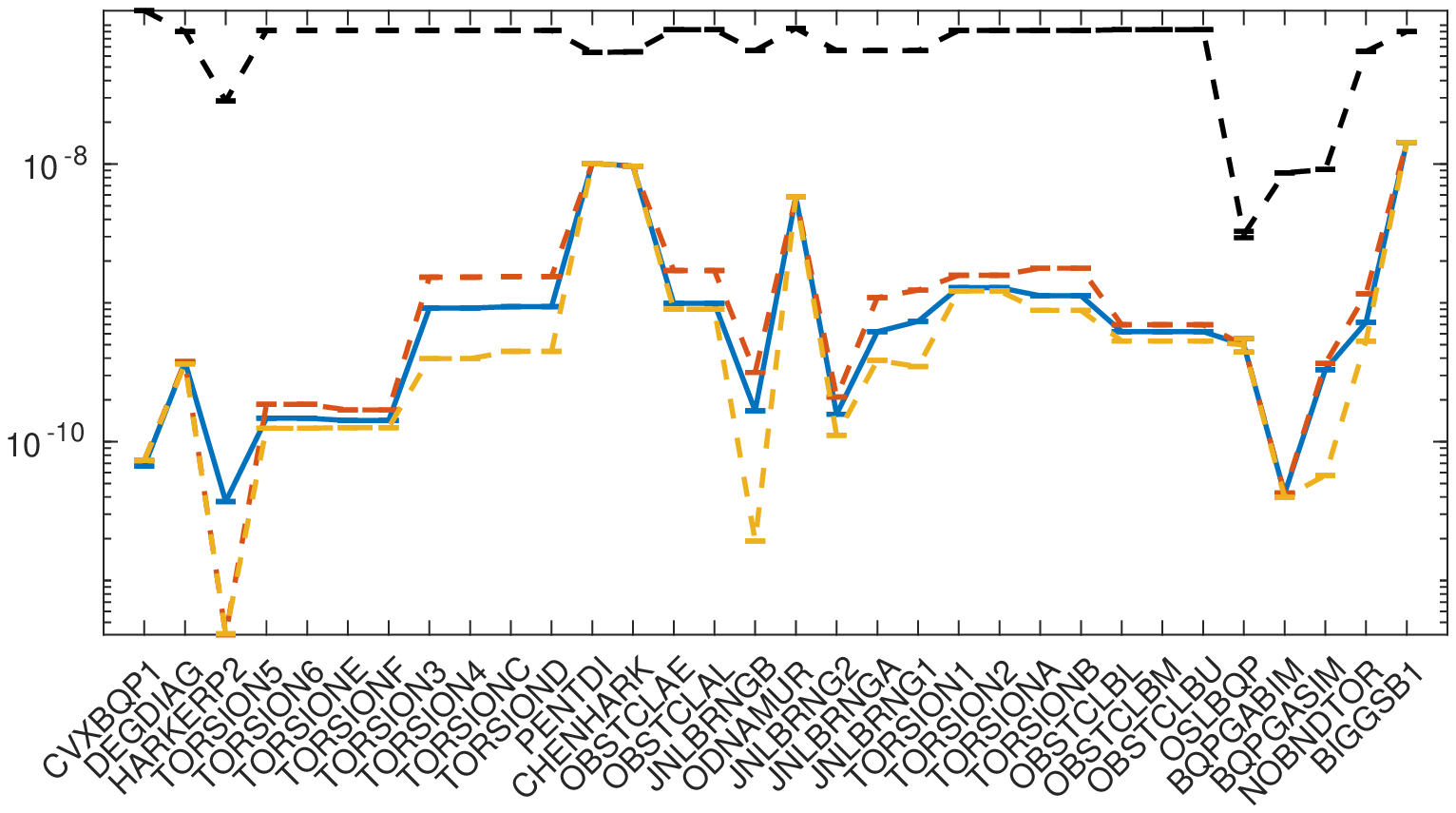} 
\end{minipage}
\caption{Mean approximation error and mean progress with measure $\| F_{\mu^+}\| $ with one standard deviation error bars for a collection of CUTEst test problems.}
\label{fig:CUTEstDiffMu}
\end{figure}
The partial approximate solution errors in
Figure~\ref{fig:CUTEstDiffMu} are significantly larger compared to
those of Figure~\ref{fig:randProbs}. This is expected since the
optimal solutions of the CUTEst test problems typically do not satisfy strict complementarity. Moreover, with the above strategy for determining the active and inactive sets, the smallest active multipliers may be in the order of $10^{-5}$. Small active multipliers may cause inaccurate components in the approximate solution of $\Delta x_{\mathcal{A}}^N$.
Nevertheless, the approximate solutions perform asymptotically similar to the Newton solution in terms of the measure $\| F_{\mu^+} \| $, as shown in Figure~\ref{fig:CUTEstDiffMu}. The figure also shows that the approximation error and the progress measure are not particularly sensitive to different initial solutions for smaller $\mu$, whereas some effects can be seen for larger $\mu$. The results may be improved and dis-improved depending on how the estimation of the active constraints at the solution is made.  We chose to give the results for the strategy described above which gives a potentially significant reduction in the computational iteration cost.

In practice the active constraints at the optimal solution are unknown and have to be estimated as the iterations proceed. The purpose of the following simulations is to give an initial indication of the performance of the proposed approximate solutions within a primal-dual interior-point framework. In particular, we focus on the behavior on problems that do not satisfy the assumptions for which the theoretical results are valid, but also on the robustness in regards to how the set of active constraints is estimated. Algorithm~\ref{alg:IPMnewt} and \ref{alg:simpleIPM} were considered with the aim of not drowning, or combining, approximation effects with other effects from more advanced features in more sophisticated methods. Algorithm~\ref{alg:IPMnewt} should here be seen as the reference method as it only contains Newton steps.
\begin{algorithm}[H]
\caption{Reference interior-point method for convex (\ref{eq:NLP}).}
\begin{footnotesize}
\begin{algorithmic}[1]
\Statex $k \gets 0$,\quad $\mu \gets 10^2$,  \quad $(x_k, \lambda_k) \gets$  Feasible point such that $\| F_{\mu}(x_k,\lambda_k) \|  < \mu,$ \quad $\mu \gets \sigma \mu$.
\Statex \textbf{While} $\| F_0(x_k,\lambda_k) \| > \epsilon$ \textbf{do} 
\Statex\hspace{\algorithmicindent} $(\Delta x_k, \Delta \lambda_k)  \qquad \gets$  (\ref{eq:PDsyst})
\Statex\hspace{\algorithmicindent} $(\alpha^P_k, \alpha^D_k)  \qquad  \quad   \gets \left( \min \{ 1,0.98\alpha^P_{max,k}\}, \min \{ 1,0.98\alpha^D_{max,k}\} \right)$ 
\Statex\hspace{\algorithmicindent} $(x_{k+1}, \lambda_{k+1})  \> \> \quad \gets (x_k + \alpha^P_k \Delta x_k, \lambda_k + \alpha^D_k \Delta \lambda_k )$ 
\Statex\hspace{\algorithmicindent} \textbf{If} $\| F_{\mu}(x_{k+1}, \lambda_{k+1})\| < \mu$
\Statex\hspace{\algorithmicindent}\hspace{\algorithmicindent}    $\mu \gets \sigma \mu$
\Statex\hspace{\algorithmicindent} \textbf{End}
\Statex\hspace{\algorithmicindent}  $k \gets k+1$
\Statex \textbf{End}
\end{algorithmic}
\end{footnotesize}
\label{alg:IPMnewt}
\end{algorithm} 
\vspace{-4mm}
\begin{algorithm}[H]
\caption{Simple interior-point method for convex (\ref{eq:NLP}).}
\begin{footnotesize}
\begin{algorithmic}[1]
\Statex $k \gets 0$,\quad $\mu \gets 10^2$,  \quad $(x_k, \lambda_k) \gets$  Feasible point such that $\| F_{\mu}(x_k,\lambda_k) \|  < \mu,$ \quad $\mu \gets \sigma \mu$.
\Statex \textbf{While} $\| F_0(x_k,\lambda_k) \| > \epsilon$ \textbf{do} 
\Statex\hspace{\algorithmicindent} Estimate active constraints to obtain active/inactive-sets
\Statex\hspace{\algorithmicindent} $(\Delta x_k, \Delta \lambda_k)  \qquad \gets$ (\ref{eq:prop:genCase:schurBased:dx}) or  (\ref{eq:prop:genCase:compBased:dx}) combined with (\ref{eq:genCase:Redx2syst_schur}), (\ref{eq:genCase:fullApprox:dlambdaI}) and (\ref{eq:genCase:fullApprox:dlambdaAls})
\Statex\hspace{\algorithmicindent} $(\alpha^P_k, \alpha^D_k)  \qquad  \quad   \gets \left( \min \{ 1,0.98\alpha^P_{max,k}\}, \min \{ 1,0.98\alpha^D_{max,k}\} \right)$ 
\Statex\hspace{\algorithmicindent} $(x_{k+1}, \lambda_{k+1})  \> \> \quad \gets (x_k + \alpha^P_k \Delta x_k, \lambda_k + \alpha^D_k \Delta \lambda_k )$ 
\Statex\hspace{\algorithmicindent} \textbf{If} $\| F_{\mu}(x_{k+1}, \lambda_{k+1})\| < \mu$
\Statex\hspace{\algorithmicindent}\hspace{\algorithmicindent}    $\mu \gets \sigma \mu$
\Statex\hspace{\algorithmicindent} \textbf{End}
\Statex\hspace{\algorithmicindent}  $k \gets k+1$
\Statex \textbf{End}
\end{algorithmic}
\label{alg:simpleIPM}
\end{footnotesize}
\end{algorithm} 
\vspace{-2mm}
\noindent At iteration $k$ of Algorithm~\ref{alg:IPMnewt} and
Algorithm~\ref{alg:simpleIPM}, $\alpha^P_{max,k}$ and
$\alpha^D_{max,k}$ are the maximum feasible step lengths for $x_k$
along $\Delta x_k$ and $\lambda_k$ along $\Delta \lambda_k$
respectively. Table~\ref{table:CUTEstData1} contains a comparison of
Algorithm~\ref{alg:IPMnewt} and two versions of
Algorithm~\ref{alg:simpleIPM} which differ in how $\Delta
x_\mathcal{A}$ is computed. The versions are denoted by
$\texttt{aN}^\texttt{S}$ and $\texttt{aN}^\texttt{C}$ as they use the
approximates $\Delta x_\mathcal{A}^S$ and $\Delta x_\mathcal{A}^C$
respectively. In Algorithm~\ref{alg:simpleIPM}, a constraint was
considered active if the distance to its bound was smaller than the
value of its multiplier and a threshold
$\tau_\mathcal{A}$. The procedure is thus a basic heuristic aimed at
determining the non-degenerate active constraints. In essence, the heuristic gives an estimate of set $\mathcal{A}_x$, compare to Definition~\ref{def:genCase:sets} in the
theoretical setting. The thresholds of the two versions
$\texttt{aN}^\texttt{S}$ and $\texttt{aN}^\texttt{C}$ were chosen to
$\tau_\mathcal{A} = \mu^{2/3}$ and the more restrictive
$\tau_\mathcal{A} = \mu^{3/4}$ respectively. This was done to show the effects of
two different thresholds $\tau_\mathcal{A}$, but also because numerical
experiments have shown that steps with Schur-based
  approximation are more robust at larger $\mu$, see
  Figure~\ref{fig:CUTEstDiffMu}.  Table~\ref{table:CUTEstData1} gives
a comparison of the number of iterations for different values of $\mu$
as well as the average cardinality of $\mathcal{I}_x$, the set of
indices corresponding to the estimated inactive components of $x$, i.e., the size
of the systems that has to be solved in every iteration. The symbol
"-" denotes the situation when the method failed to converge within 50
iterations for the corresponding $\mu$. If the method failed at a specific $\mu$ then Newton steps were performed instead until $\| F_\mu (x,\lambda) \| < \mu$. The order of the
problems is the same as in Figure~\ref{fig:CUTEstDiffMu}.
\begin{table}[H]
      \centering
\begin{footnotesize}
      \caption{Comparison of Algorithm~\ref{alg:IPMnewt}, ($\texttt{N}$), and two versions of Algorithm~\ref{alg:simpleIPM}, ($\texttt{aN}^\texttt{S}$ and $\texttt{aN}^\texttt{C}$) on a selection of CUTEst test problems.}
\begin{tabular}{|c|l|c|c|c|c|c|c|c|c|c|c|c|} \label{table:CUTEstData1}
& $\mu$& $10^{1}$ & $10^{0}$  & $10^{-2}$ & $10^{-3}$ & $10^{-5}$ & $10^{-6}$  & $10^{-8}$ & $10^{-9}$ & $10^{-10}$  \\      \hline \hline 
   \multirow{5}{*}{\rotatebox[origin=r]{90}{\parbox[c]{1cm}{\centering \texttt{CVXBQP1}\\ \tiny{$n_x$=10000}}}}
 &$\texttt{N}$&    3  & 3  &  2&   1 &  1  & 1 &    1  & 1  & 1\\
&$\texttt{aN}^\texttt{S}$  &  3  & 2 &    2 &  1 &   1 &  1 &   1 &  1  & 1\\
&$\texttt{aN}^\texttt{C}$  & 4 &  2&  1 &  1 &   1 &  1 &    1 &  1 &  1\\
&$|\bar{\mathcal{I}}_x^S |$  &  0 &  0  &  0 &  0 &    0 &  0  & 0  & 0  & 0\\
&$|\bar{\mathcal{I}}_x^C |$  &  0  & 0  &   0 &  0 &   0 &  0  & 0 &  0&   0  \\
 \hline \hline 
   \multirow{5}{*}{\rotatebox[origin=r]{90}{\parbox[c]{1cm}{\centering \texttt{DEGDIAG}\\ \tiny{$n_x$=10001}}}}
&$\texttt{N}$ &      4   &  4    & 3  &   3    &  3  &   3  &   2      &  2   &  2\\
&$\texttt{aN}^\texttt{S}$  &    4  &   4   &    3  &   3    &   3   &  3  &   2 &    2   &  2\\
&$\texttt{aN}^\texttt{C}$  &  12   &  7  &  42 &    5    &   3   &  3   &  2   &  2   &  2\\
&$|\bar{\mathcal{I}}_x^S |$  &  1   &  1   &   830 &  635 &   195  &  93  &   28   & 13   &  6\\
&$|\bar{\mathcal{I}}_x^C |$  &  1 &  142   & 93 &  709  & 428  & 237  & 100  &  56  &  32     \\   
 \hline \hline
   \multirow{5}{*}{\rotatebox[origin=r]{90}{\parbox[c]{1cm}{\centering \texttt{HARKERP2}\\ \tiny{$n_x$=1000}}}}
 &$\texttt{N}$ &  3  &  3    & 4  &  3   &  3  &  3   &  2  &  2  &  1 \\
 &$\texttt{aN}^\texttt{S}$ &  - & -   & - &  -  &   15  &  6    &  2 &   1 &   1 \\
 &$\texttt{aN}^\texttt{C}$ &  - &  -  & -  &  2   & 15  &  6   &  2 &   1  &  1 \\
 &$|\bar{\mathcal{I}}_x^S |$ &  -  &  -    & -  &  -  &   1  &  1  &    1   & 1  &  1\\
  &$|\bar{\mathcal{I}}_x^C |$  & -  &  -    &  - &  1   & 1  &  1 &   1   & 1  &  1\\
 \hline \hline
      \multirow{5}{*}{\rotatebox[origin=r]{90}{\parbox[c]{1cm}{\centering \texttt{TORSION5$^*$}\\ \tiny{$n_x$=5184}}}}
&$\texttt{N}$ &       1  &    1       & 2   &   3     &    3  &    3  &   3    &  2   &   2\\ 
&$\texttt{aN}^\texttt{S}$ &  1  &    1      &  2   &   3     &   3   &   4    &   3    &  2    &  2\\ 
&$\texttt{aN}^\texttt{C}$ &  29   &  -     & -  &   -    &   3    &  3    &   3   &   2   &   2\\ 
&$|\bar{\mathcal{I}}_x^S |$ &   0  &    0    & 2564 &  4535   & 5083  & 2802  &  2277   & 968  &  960\\ 
&$|\bar{\mathcal{I}}_x^C |$ &  0   &   -   & -  & -   & 5101 &  5064  &  2376 &  2944  & 2936     \\  
 \hline \hline
     \multirow{5}{*}{\rotatebox[origin=r]{90}{\parbox[c]{1cm}{\centering \texttt{TORSIONE$^*$}\\ \tiny{$n_x$=5184}}}}
&$\texttt{N}$        &   1   &   1  &    2     & 3    &    3   &   3   &       3    &  2   &   2\\
&$\texttt{aN}^\texttt{S}$   &  1   &   1       &  2    &  3     &  3   &   4     &   3   &   2    &  2\\
&$\texttt{aN}^\texttt{C}$   &  29  &   - &      -   &  -   &    3  &    3  &        3   &   2  &    2\\
&$|\bar{\mathcal{I}}_x^S |$  &   0  &    0  &    2564 &  4535  & 5171  & 2872  &  2379   & 984 &   976\\
&$|\bar{\mathcal{I}}_x^C |$  &     0   &   -  &  -  & -  & 5184 &  5171  &  2387  & 3080  & 3080     \\  
 \hline  \hline 
     \multirow{5}{*}{\rotatebox[origin=r]{90}{\parbox[c]{1cm}{\centering \texttt{TORSION3$^*$} \\ \tiny{$n_x$=5184}}}}
&$\texttt{N}$   &   1    &    1         &     2    &    2     &     4   &     3      &   3    &    3  &      3 \\
&$\texttt{aN}^\texttt{S}$   &  1    &      1    &    2    &    3       &     4    &    3     &    3     &   3  &      3 \\
&$\texttt{aN}^\texttt{C}$   &  29   &    -   &    -   &    -      &   4   &     3     &       3  &     3   &     3 \\
&$|\bar{\mathcal{I}}_x^S |$    &  0   &     0   &  2564  &   4535    &  5062   &  4933   &   2931   &  2867   &  1872 \\
&$|\bar{\mathcal{I}}_x^C |$  &    0   &  -   &  -  &   -    &  5184   &  5069  &   4008   &  3035   &  2931     \\  
 \hline  \hline 
      \multirow{5}{*}{\rotatebox[origin=r]{90}{\parbox[c]{1cm}{\centering \texttt{TORSIONC$^*$}\\ \tiny{$n_x$=5184}}}}
&$\texttt{N}$        &     1  &    1    &    2    &  2    &    3    &  3     &   3  &    3    &  3\\
&$\texttt{aN}^\texttt{S}$   &   1  &    1   &  2   &   3     &    3    &  3     &   3    &  3    &  3\\
&$\texttt{aN}^\texttt{C}$   &     29  &   -  &    -  &   -    &   3    &  3    &  3   &   3  &    3\\
&$|\bar{\mathcal{I}}_x^S |$  &  0   &   0   & 2564   &4535 &    5184 &  5104  &  3043 &  2976  & 1907\\
&$|\bar{\mathcal{I}}_x^C |$  &   0   &   -  & - &  -  & 5184  & 5179  & 4109  & 3059  & 3040     \\  
 \hline
\end{tabular}
\end{footnotesize}
\end{table}
\begin{table}[H]
      \centering
\begin{footnotesize}
     \begin{flushleft}
{\textbf{Table~\ref{table:CUTEstData1} continued:}}
\end{flushleft}
\begin{tabular}{|c|l|c|c|c|c|c|c|c|c|c|c|c|} 
& $\mu$& $10^{1}$ & $10^{0}$  & $10^{-2}$ & $10^{-3}$ & $10^{-5}$ & $10^{-6}$  & $10^{-8}$ & $10^{-9}$ & $10^{-10}$  \\      \hline \hline 
        \multirow{5}{*}{\rotatebox[origin=r]{90}{\parbox[c]{1cm}{\centering \texttt{PENTDI}\\ \tiny{$n_x$=5000}}}}
&$\texttt{N}$   &  4    &  4     & 4   &   4       &  4  &    4     &  4 &     4    &  4\\
&$\texttt{aN}^\texttt{S}$  & 6   &   7    &   7   &   7    &   5  &    4     &    4   &   4    &  4\\
&$\texttt{aN}^\texttt{C}$ &     -  &   -  &    - &   -   &   4  &    4      &  4 &     4  &    4\\
&$|\bar{\mathcal{I}}_x^S |$   &     0   &   0    &  2  &    2 &   1000  & 1873    & 2498 &  2498  & 2498\\
&$|\bar{\mathcal{I}}_x^C |$ &   - &    -  &  -   &  -    &2498   &2498   &   2498  & 2498  & 2498\\
  \hline  \hline 
            \multirow{5}{*}{\rotatebox[origin=r]{90}{\parbox[c]{1cm}{\centering \texttt{CHENHARK}\\ \tiny{$n_x$=5000}}}}
&$\texttt{N}$   &    4  &    4    &   4  &    4    &    4  &    4     &   4   &   3   &   3\\
&$\texttt{aN}^\texttt{S}$  &  4   &   4  & 4  &    4   &    4   &   4      &  4   &   3     & 3\\
&$\texttt{aN}^\texttt{C}$ &     4  &   -  &    4   &   4     &  4    &  4     &   4 &     3 &     3\\
&$|\bar{\mathcal{I}}_x^S |$   & 4999 &  4567  &  2502 &  2502  &  2502 &  2502 &  2502 &  2502  & 2502\\
&$|\bar{\mathcal{I}}_x^C |$ &  4999  & -   &  2502  & 2502 & 2502 &  2502  & 2502  & 2502 &  2502\\
  \hline \hline 
              \multirow{5}{*}{\rotatebox[origin=r]{90}{\parbox[c]{1cm}{\centering \texttt{JNLBRNGB}\\ \tiny{$n_x$=5329}}}}
&$\texttt{N}$   &     4   &   4     &   4  &    3    &   3   &   3      &   3   &   2  &    2\\
&$\texttt{aN}^\texttt{S}$  &  6  &   15   &    16  &   19    &  20   &   9     &   3    &  2   &   2\\
&$\texttt{aN}^\texttt{C}$ &   - &   -   &  - &   -   &    3     & 3      &   3  &    2   &   2\\
&$|\bar{\mathcal{I}}_x^S |$   &  5196  & 5171  &  5157 &  5190  &  3220  & 3291  & 3111  & 3063  & 3026\\
&$|\bar{\mathcal{I}}_x^C |$ &  - &  -  & - &  -  & 4899  & 4182   &  3843 &  3758  & 3283\\
  \hline \hline
                \multirow{5}{*}{\rotatebox[origin=r]{90}{\parbox[c]{1cm}{\centering \texttt{OBSTCLAE$^*$}\\ \tiny{$n_x$=5329}}}}
&$\texttt{N}$   &      4   &    4    &    4    &   4     &    3  &     3    &   3  &     3   &    2\\
&$\texttt{aN}^\texttt{S}$  & 4  &     4    &    4    &   4     &   5 &      4     &   3  &     3    &   2\\
&$\texttt{aN}^\texttt{C}$ &   4   &    4   &    4   &    4     &     3  &     3   &     3  &     3   &    2\\
&$|\bar{\mathcal{I}}_x^S |$   & 5329 &   5329 &   5329 &   5329 &    5063&    4153   &  3766 &   3268  &  2978\\
&$|\bar{\mathcal{I}}_x^C |$ & 5329 &   5329 &  5329  &  5329   &  5290 &   5313  &    4539  &  3786 &   4158\\
  \hline \hline
     \multirow{5}{*}{\rotatebox[origin=r]{90}{\parbox[c]{1cm}{\centering \texttt{JNLBRNG2}\\ \tiny{$n_x$=5329}}}}
&$\texttt{N}$   &       4  &    4   &    4&      3     &    3   &   3   &       3   &   2  &    2\\
&$\texttt{aN}^\texttt{S}$   &     5  &    9    &  11   &   9  &   11  &    6      &  3    &  2   &   2\\
&$\texttt{aN}^\texttt{C}$   &   -   &  -  &    -  &  -  &    3  &    3       & 3  &    2  &    2\\
&$|\bar{\mathcal{I}}_x^S |$   &    5206  & 5177  & 5187 &  5217  &  3561 &  3622  &   3315  & 3270  & 3232\\
&$|\bar{\mathcal{I}}_x^C |$   &  -  &  - &  - &  - &  4736  & 4297  & 3981 &  3861 &  3450\\
  \hline \hline
                    \multirow{5}{*}{\rotatebox[origin=r]{90}{\parbox[c]{1cm}{\centering \texttt{OBSTCLBL$^*$}\\ \tiny{$n_x$=5329}}}}
&$\texttt{N}$   &       1    &  1   &    2  &    3    &   3   &   3    &   3    &  3    &  2\\
&$\texttt{aN}^\texttt{S}$  &   1   &   1   &  3    &  6   &     4   &   3      &  3    &  3    &  2\\
&$\texttt{aN}^\texttt{C}$ &    24  &   -  &   -  &   -   &    3    &  3   &    3     & 3  &    2\\
&$|\bar{\mathcal{I}}_x^S |$   &   0    &  0    & 1121 &  3190  &4469  & 4318 &  3950 &  3903  & 3880\\
&$|\bar{\mathcal{I}}_x^C |$ &   0  &   - & - &  -  &  4862  & 4515  &  4212  & 4077  & 4002\\
     \hline \hline 
               \multirow{5}{*}{\rotatebox[origin=r]{90}{\parbox[c]{1cm}{\centering \texttt{JNLBRNGA}\\ \tiny{$n_x$=5329}}}}
&$\texttt{N}$   &        4  &    4   &    4   &   4    &   3  &    3  &   3   &   3  &    2\\
&$\texttt{aN}^\texttt{S}$   &     4   &   4      &   4   &   4      &  3  &    3    &    3    &  3   &   2\\
&$\texttt{aN}^\texttt{C}$   &    4    &  4     &  4    &  4      &  3   &   3    &     3    &  3   &   2\\
&$|\bar{\mathcal{I}}_x^S |$   &  5329  & 5329  & 5329 &  5329   & 5329 &  4991 &  4257 &  3985 &  3718\\
&$|\bar{\mathcal{I}}_x^C |$ & 5329 &  5329 &  5329 &  5329 &  5329  & 5279  &  4728 &  4381  & 4506\\
  \hline \hline
         \multirow{5}{*}{\rotatebox[origin=r]{90}{\parbox[c]{1cm}{\centering \texttt{TORSION1$^*$}\\ \tiny{$n_x$=5184}}}}
&$\texttt{N}$ &        1    &  1    &   1   &   2      &   4   &   3 &      3  &    3   &   3\\ 
&$\texttt{aN}^\texttt{S}$  &   1   &   1   &    1   &   2   &    4    &  3  &   3  &    3   &   3\\ 
&$\texttt{aN}^\texttt{C}$ &   29   & -  &    -    & -   &   4  &    3   &    3  &    3   &   3\\ 
&$|\bar{\mathcal{I}}_x^S |$ &  0    &  0  & 1764  & 4490   & 5098   &5032  & 4133  & 4040  & 4024\\ 
&$|\bar{\mathcal{I}}_x^C |$ &  0    &  - & - &  -  &  5184  & 5061  &  5011 &  4968  & 4080 \\  
\hline \hline
       \multirow{5}{*}{\rotatebox[origin=r]{90}{\parbox[c]{1cm}{\centering \texttt{JNLBRNG1}\\ \tiny{$n_x$=5329}}}}
&$\texttt{N}$   &    4  &    4   &    4    &  4    &   3 &     3     &   3     & 3  &    2\\
&$\texttt{aN}^\texttt{S}$   &    4   &   5   &   6  &    5    &     3   &   4    &   3      & 3    &  2\\
&$\texttt{aN}^\texttt{C}$   &   10  &  -   &   39  &   -   &    3   &   3   &    3     & 3   &   2\\
&$|\bar{\mathcal{I}}_x^S |$   & 5322  & 5319   & 5319 &  5320  & 5329  & 4972 &  4331 &  4028 &  3800\\
&$|\bar{\mathcal{I}}_x^C |$   & 5312  & -  &  5319 &  -&  5329  & 5283  &  4993 &  4446  & 4536\\
  \hline       
\end{tabular}
\end{footnotesize}
\end{table}
\begin{table}[H]
      \centering
\begin{footnotesize}
 \begin{flushleft}
{\textbf{Table~\ref{table:CUTEstData1} continued:}}
\end{flushleft}
\begin{tabular}{|c|l|c|c|c|c|c|c|c|c|c|c|c|} 
& $\mu$& $10^{1}$ & $10^{0}$ & $10^{-2}$ & $10^{-3}$ & $10^{-5}$ & $10^{-6}$ & $10^{-8}$ & $10^{-9}$ & $10^{-10}$  \\ 
\hline \hline       
   \multirow{5}{*}{\rotatebox[origin=r]{90}{\parbox[c]{1cm}{\centering \texttt{TORSIONA$^*$} \\ \tiny{$n_x$=5184}}}}
&$\texttt{N}$        &   1   &   1     &  1    &  2      & 4  &    3  &       3  &    3  &    2\\  
&$\texttt{aN}^\texttt{S}$   &   1   &   1     &   1   &   2      &  4   &   3    &   3    &  3     & 2\\  
&$\texttt{aN}^\texttt{C}$   &   29  &   - &     -   &  -     &  4 &     3  &     3    &  3  &    2\\  
&$|\bar{\mathcal{I}}_x^S |$  &   0    &  0   &  1764  & 4490   & 5184  & 5184  & 4261 &  4173  & 4416\\  
&$|\bar{\mathcal{I}}_x^C |$  &   0   &   -  & -  & -   & 5184 &  5184  & 5168 &  5125  & 4444    \\  
 \hline \hline
        \multirow{5}{*}{\rotatebox[origin=r]{90}{\parbox[c]{1cm}{\centering \texttt{OSLBQP}\\ \tiny{$n_x$=8}}}}
&$\texttt{N}$ &    2   &3  &  3 &  2 &  2 &  3  & 3 &  2 &  3\\
&$\texttt{aN}^\texttt{S}$  &  2  & 3&  3  & 2 &   2 &  3  &  3 &  2 &  3\\
&$\texttt{aN}^\texttt{C}$ &  4 &  5&   2 &  3  & 3  & 2 &   2 &  3 &  2\\
&$|\bar{\mathcal{I}}_x^S |$ & 4  & 2 &   4 &  6 &   6 &  6 &  6 &  6 &  6\\
&$|\bar{\mathcal{I}}_x^C |$ &  1 &  2 &  6 &  6  &  6  & 6 & 6 &   6  & 6 \\  
  \hline \hline
         \multirow{5}{*}{\rotatebox[origin=r]{90}{\parbox[c]{1cm}{\centering \texttt{BQPGABIM}\\ \tiny{$n_x$=46}}}}
&$\texttt{N}$ &   1 &   2   &  3  &  3  &  3 &   3  &  2  &  2   & 1\\
&$\texttt{aN}^\texttt{S}$  & 1  &  3 &    9 &   7    &  3  & 3 &  2  &  2  &  1\\
&$\texttt{aN}^\texttt{C}$ & - &  -  & -  & -   & - &  - &   2  &  2 &   1\\
&$|\bar{\mathcal{I}}_x^S |$ &  0  &  0  &  4  & 24   & 29 &  31 &   37 &  36 &  36\\
&$|\bar{\mathcal{I}}_x^C |$ &  -   & -  & -  & - &  - &  - &  39  & 38 & 38 \\  
 \hline \hline
        \multirow{5}{*}{\rotatebox[origin=r]{90}{\parbox[c]{1cm}{\centering \texttt{BQPGASIM}\\ \tiny{$n_x$=50}}}}
&$\texttt{N}$ &     1  &  2  &  3  &  3 &  3  &  3  &  2 &   2  &  2\\
&$\texttt{aN}^\texttt{S}$  &  1 &   3  &   10  &  7   & 3  &  3 &   2  &  2   & 2\\
&$\texttt{aN}^\texttt{C}$ & - &  -  & - &  -& -  & - &  2  &  2   & 2\\
&$|\bar{\mathcal{I}}_x^S |$ & 0  &  0   &  4  & 27  &  32 &  34  &  41 &  42  & 42\\
&$|\bar{\mathcal{I}}_x^C |$ &  -  &  -  & - &  - &  - &  - & 43  & 42  & 42 \\  
\hline  \hline 
      \multirow{5}{*}{\rotatebox[origin=r]{90}{\parbox[c]{1cm}{\centering \texttt{NOBNDTOR}\\ \tiny{$n_x$=5184}}}}
&$\texttt{N}$   &       1   &   1     &  2   &   2    &   4   &   4   &   3   &   3   &   2\\
&$\texttt{aN}^\texttt{S}$  &  1    &  1  &    2  &    2   &    4  &    4   &    3    &  3   &   2\\
&$\texttt{aN}^\texttt{C}$ &    21 &  -   &  - &  -      &   4    &  4    &   3   &   3  &    2\\
&$|\bar{\mathcal{I}}_x^S |$   &   2592 &  2592  &  3874&   4837  &  5132 &  5107 &  4621  & 4570  & 4681\\
&$|\bar{\mathcal{I}}_x^C |$ &  2592  & -  &  - &  -   & 5183 &  5148 & 5077 &  5050  & 4704\\
  \hline  \hline 
   \multirow{5}{*}{\rotatebox[origin=r]{90}{\parbox[c]{1cm}{\centering \texttt{BIGGSB}\\ \tiny{$n_x$=5000}}}}
&$\texttt{N}$ &       1   &   1    &   2   &   3    &    3  &    3   &   4     & 4   &   4\\
&$\texttt{aN}^\texttt{S}$   &  1   &   1   &   4  &    3    &   3  &    3  &     4     & 4   &   4\\
&$\texttt{aN}^\texttt{C}$   &  11  &   13    & 2  &    3   &     3   &   3   &  4     & 4    &  4\\
&$|\bar{\mathcal{I}}_x^S |$  &  1    &  1   &  4999 &  4999  & 4998  & 4998  &  4998 &  4998 &  4998\\
&$|\bar{\mathcal{I}}_x^C |$  &   1   &   1   & 5000  & 4999 &  4998 &  4998  & 4998 &  4998  & 4998  \\
 \hline
\end{tabular}
\end{footnotesize}
\end{table}
\footnotetext[1]{The tables would be identical for other versions of the same problem and are therefore omitted}
The results in Table~\ref{table:CUTEstData1} display similar characteristics as the results in Figure~\ref{fig:CUTEstDiffMu}. The version associated with the Schur-based approximate solution, $\texttt{aN}^\texttt{S}$ of Algorithm~\ref{alg:simpleIPM}, makes sufficient progress at $\mu \in [10^2,10^{-2})$, often at a relatively low computational cost. Version $\texttt{aN}^\texttt{S}$ converges at $\mu \in [10^{-2}, 10^{-6}]$, however, often while solving relatively large systems due to the difficulty of estimating $\mathcal{A}_x$. At $\mu \in (10^{-6},0)$ the asymptotic behavior becomes more pronounced. Consequently, $\texttt{aN}^\texttt{S}$ does similar in terms of iteration count to Algorithm~\ref{alg:IPMnewt} while solving systems of reduced size. Version $\texttt{aN}^\texttt{S}$ converges at all considered $\mu$ in all problems of Table~\ref{table:CUTEstData1}, except on \texttt{HARKERP2} for larger $\mu$. The version associated with the complementarity-based approximate solution, $\texttt{aN}^\texttt{C}$ of Algorithm~\ref{alg:simpleIPM}, tend to perform poorly overall for $\mu \in [10^2,10^{-2})$ and parts of $[10^{-2}, 10^{-6}]$. Although $\texttt{aN}^\texttt{C}$ converges for large $\mu$, this is often at the expense of either solving relatively large systems or performing many iterations. In general, $\texttt{aN}^\texttt{C}$ performs similar to Algorithm~\ref{alg:IPMnewt} for $\mu$ in the approximate region $[10^{-5},0)$ while solving systems of reduced size. The versions $\texttt{aN}^\texttt{S}$ and $\texttt{aN}^\texttt{C}$ have similar asymptotic performance, however in general, $\texttt{aN}^\texttt{S}$ performs better for larger values of $\mu$, as also indicated by previous results in Figure~\ref{fig:CUTEstDiffMu}.

Finally we show results for the two Newton-like approaches, mentioned in Section~\ref{subsec:partialApp}, in a simple primal-dual interior-point setting. The approximate intermediate step method and the approximate higher-order method are described in Algorithm~\ref{alg:newtLikeST} and Algorithm~\ref{alg:newtLikeHO} respectively. In contrast to Section~\ref{subsec:partialApp}, here the intermediate iterate is required to be strictly feasible. 
The total number of iterations required at different intervals of $\mu$ with the two Newton-like approaches is shown in Figure~\ref{fig:modNewtonIPMs}. 
The figure shows results for three different choices of $(\Delta x^E, \Delta \lambda^E)$. Moreover, the selection of which components to update was done as the iterations proceeded similarly as above. Note however that it is not necessary to label each constraint and each component of $\lambda$ as active or inactive in this case, some may be defined as neither. The set of indices corresponding to active constraints, $\mathcal{A}_x$, was estimated as above and the sets of indices corresponding to inactive $\lambda$, $\mathcal{I}_l$ and $\mathcal{I}_u$, see Definition~\ref{def:genCase:sets}, were estimated analogously. I.e., a multiplier was considered inactive if its value was smaller than the distance of the corresponding $x$ to its feasibility bound and a threshold $\tau_\mathcal{I}$. Table~\ref{table:modNewton} shows how the nonzero components of $(\Delta x^E, \Delta \lambda^E)$ were chosen in the different versions of the approaches as well as the different thresholds $\tau_\mathcal{A}$ and $\tau_\mathcal{I}$.
\begin{table}[H] \renewcommand{\arraystretch}{1.3}
      \caption{Thresholds and nonzero components of the steps to $(x^E, \lambda^E) $ for the three versions compared in Figure~\ref{fig:modNewtonIPMs}.}
      \begin{footnotesize}
\begin{tabular}{l|c|c|} \label{table:modNewton}
Nonzero components in  $(\Delta x^E, \Delta \lambda^E)$  & $\tau_\mathcal{A}$  &  $\tau_\mathcal{I}$ \\
\hline 
$\Delta x^S_\mathcal{A}$                                                                      &     $\mu^{1/2}$ &   \\
 $\Delta x^S_\mathcal{A} $, $  \Delta \lambda^C_\mathcal{I}$ & $\mu^{1/2}$    & $\mu^{3/4}$ \\
 $\Delta x^C_\mathcal{A} $, $ \Delta \lambda^C_\mathcal{I}$  & $\mu^{3/4}$   & $\mu^{3/4}$ \\
     \end{tabular} \renewcommand{\arraystretch}{1.0}
     \end{footnotesize}
\end{table} 
\begin{algorithm}[H]
\caption{Simple interior-point method with an approximate intermediate step for convex (\ref{eq:NLP}).}
\begin{footnotesize}
\begin{algorithmic}[1]
\Statex $k \gets 0$,\quad $\mu \gets 10^2$,  \quad $(x_k, \lambda_k) \gets$  Feasible point such that $\| F_{\mu}(x_k,\lambda_k) \|  < \mu,$ \quad $\mu \gets \sigma \mu$.
\Statex \textbf{While} $\| F_0(x_k,\lambda_k) \| > \epsilon$ \textbf{do} 
\Statex\hspace{\algorithmicindent} Estimate active constraints to obtain active/inactive-sets
\Statex\hspace{\algorithmicindent} $(\Delta x^E_k, \Delta \lambda^E_k)  \quad  \> \> \> \gets$ (\ref{eq:prop:genCase:partialAstep})
\Statex\hspace{\algorithmicindent} $(\alpha^{E,P}_k, \alpha^{E,D}_k)  \quad \>  \gets \left( \min \{ 1,0.98\alpha^{E,P}_{max,k}\}, \min \{ 1,0.98\alpha^{E,D}_{max,k}\} \right)$ 
\Statex\hspace{\algorithmicindent} $(x^E_{k}, \lambda^E_{k})  \qquad \quad \> \gets (x_k + \alpha^{E,P}_k \Delta x_k^E, \lambda_k + \alpha^{E,D}_k \Delta \lambda_k^E)$ 
\Statex\hspace{\algorithmicindent} $(\Delta x_k, \Delta \lambda_k)  \qquad  \gets$   (\ref{eq:modNewtonAstep})
\Statex\hspace{\algorithmicindent} $(\alpha^P_k, \alpha^D_k)  \qquad  \quad   \gets \left( \min \{ 1,0.98\alpha^P_{max,k}\}, \min \{ 1,0.98\alpha^D_{max,k}\} \right)$ 
\Statex\hspace{\algorithmicindent} $(x_{k+1}, \lambda_{k+1})  \quad \> \> \gets (x_k^E + \alpha^P_k \Delta x_k, \lambda_k^E + \alpha^D_k \Delta \lambda_k )$ 
\Statex\hspace{\algorithmicindent} \textbf{If} $\| F_{\mu}(x_{k+1}, \lambda_{k+1})\| < \mu$
\Statex\hspace{\algorithmicindent}\hspace{\algorithmicindent}    $\mu \gets \sigma \mu$
\Statex\hspace{\algorithmicindent} \textbf{End}
\Statex\hspace{\algorithmicindent}  $k \gets k+1$
\Statex \textbf{End}
\end{algorithmic}
\label{alg:newtLikeST}
\end{footnotesize}
\end{algorithm} 
\begin{algorithm}[H]
\begin{footnotesize}
\caption{Simple interior-point method with approximate higher-order solve for convex (\ref{eq:NLP}).}
\begin{algorithmic}[1]
\Statex $k \gets 0$,\quad $\mu \gets 10^2$,  \quad $(x_k, \lambda_k) \gets$  Feasible point such that $\| F_{\mu}(x_k,\lambda_k) \|  < \mu,$ \quad $\mu \gets \sigma \mu$.
\Statex \textbf{While} $\| F_0(x_k,\lambda_k) \| > \epsilon$ \textbf{do} 
\Statex\hspace{\algorithmicindent} Estimate active constraints to obtain active/inactive-sets
\Statex\hspace{\algorithmicindent} $(\Delta x^E_k, \Delta \lambda^E_k)  \quad  \> \> \> \gets$ (\ref{eq:prop:genCase:partialAstep})
\Statex\hspace{\algorithmicindent} $(\alpha^{E,P}_k, \alpha^{E,D}_k)  \quad \>  \gets \left( \min \{ 1,0.98\alpha^{E,P}_{max,k}\}, \min \{ 1,0.98\alpha^{E,D}_{max,k}\} \right)$ 
\Statex\hspace{\algorithmicindent} $(x^E_{k}, \lambda^E_{k})  \qquad \quad \> \gets (x_k + \alpha^{E,P}_k \Delta x_k^E, \lambda_k + \alpha^{E,D}_k \Delta \lambda_k^E)$ 
\Statex\hspace{\algorithmicindent} $(\Delta x_k, \Delta \lambda_k)  \qquad  \gets$ (\ref{eq:modNewtonHO})
\Statex\hspace{\algorithmicindent} $(\alpha^P_k, \alpha^D_k)  \qquad  \quad   \gets \left( \min \{ 1,0.98\alpha^P_{max,k}\}, \min \{ 1,0.98\alpha^D_{max,k}\} \right)$ 
\Statex\hspace{\algorithmicindent} $(x_{k+1}, \lambda_{k+1})  \quad \> \> \gets (x_k + \alpha^P_k \Delta x_k, \lambda_k + \alpha^D_k \Delta \lambda_k )$ 
\Statex\hspace{\algorithmicindent} \textbf{If} $\| F_{\mu}(x_{k+1}, \lambda_{k+1})\| < \mu$
\Statex\hspace{\algorithmicindent}\hspace{\algorithmicindent}    $\mu \gets \sigma \mu$
\Statex\hspace{\algorithmicindent} \textbf{End}
\Statex\hspace{\algorithmicindent}  $k \gets k+1$
\Statex \textbf{End}
\end{algorithmic}
\label{alg:newtLikeHO}
\end{footnotesize}
\end{algorithm} 
\noindent In Algorithm~\ref{alg:newtLikeST} and Algorithm~\ref{alg:newtLikeHO} at iteration $k$, $\alpha^P_{max,k}$, $\alpha^D_{max,k}$, $\alpha^{E,P}_{max,k}$ and $\alpha^{E,D}_{max,k}$ are for the prescribed steps defined analogously as in Algorithm~\ref{alg:IPMnewt}.
\begin{figure}[H]
\centering
\parbox{6.1cm}{
\includegraphics[width=6.3cm]{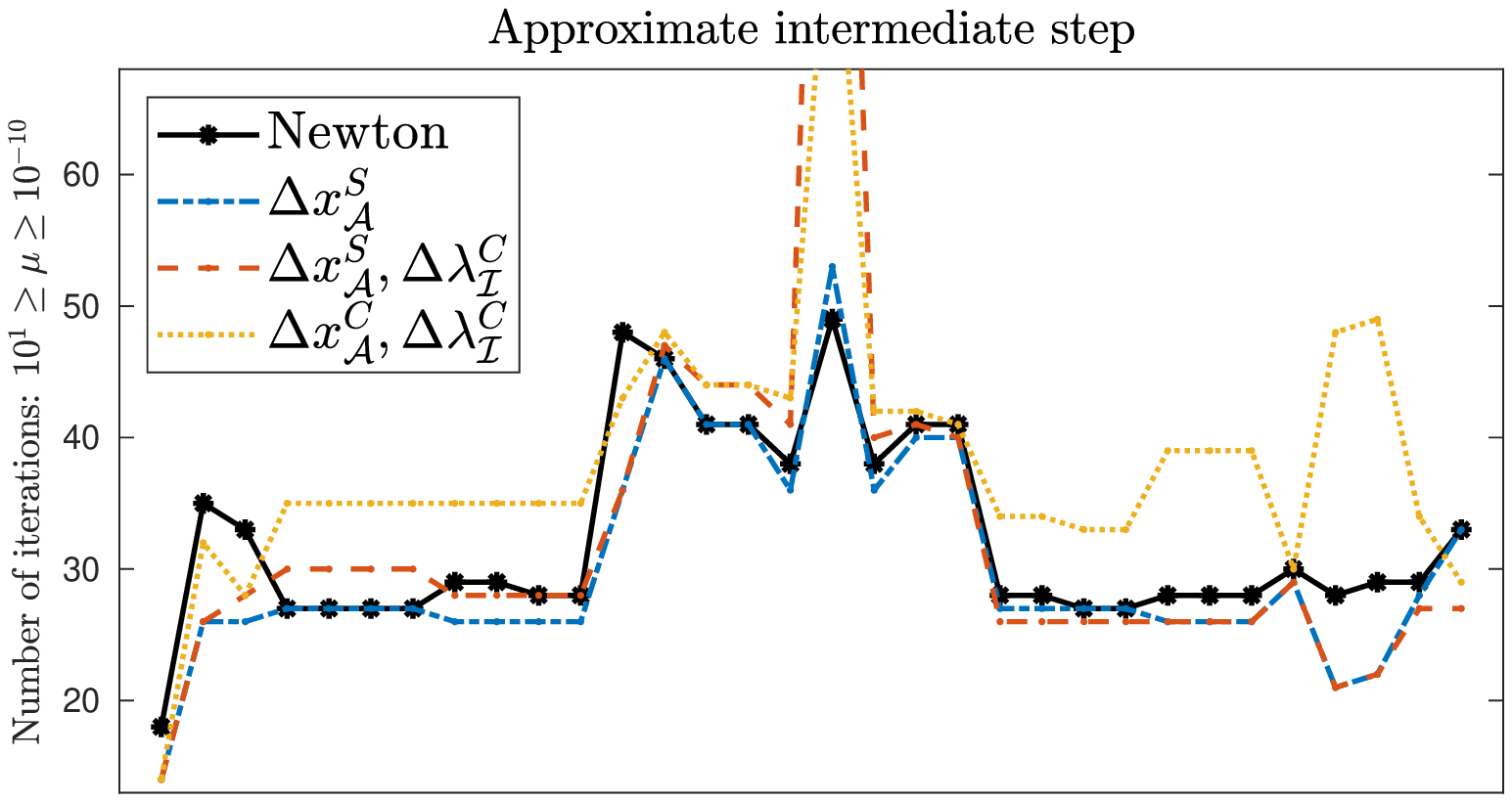}}
\hspace{-5mm}
\begin{minipage}{6.1cm}
\includegraphics[width=6.3cm]{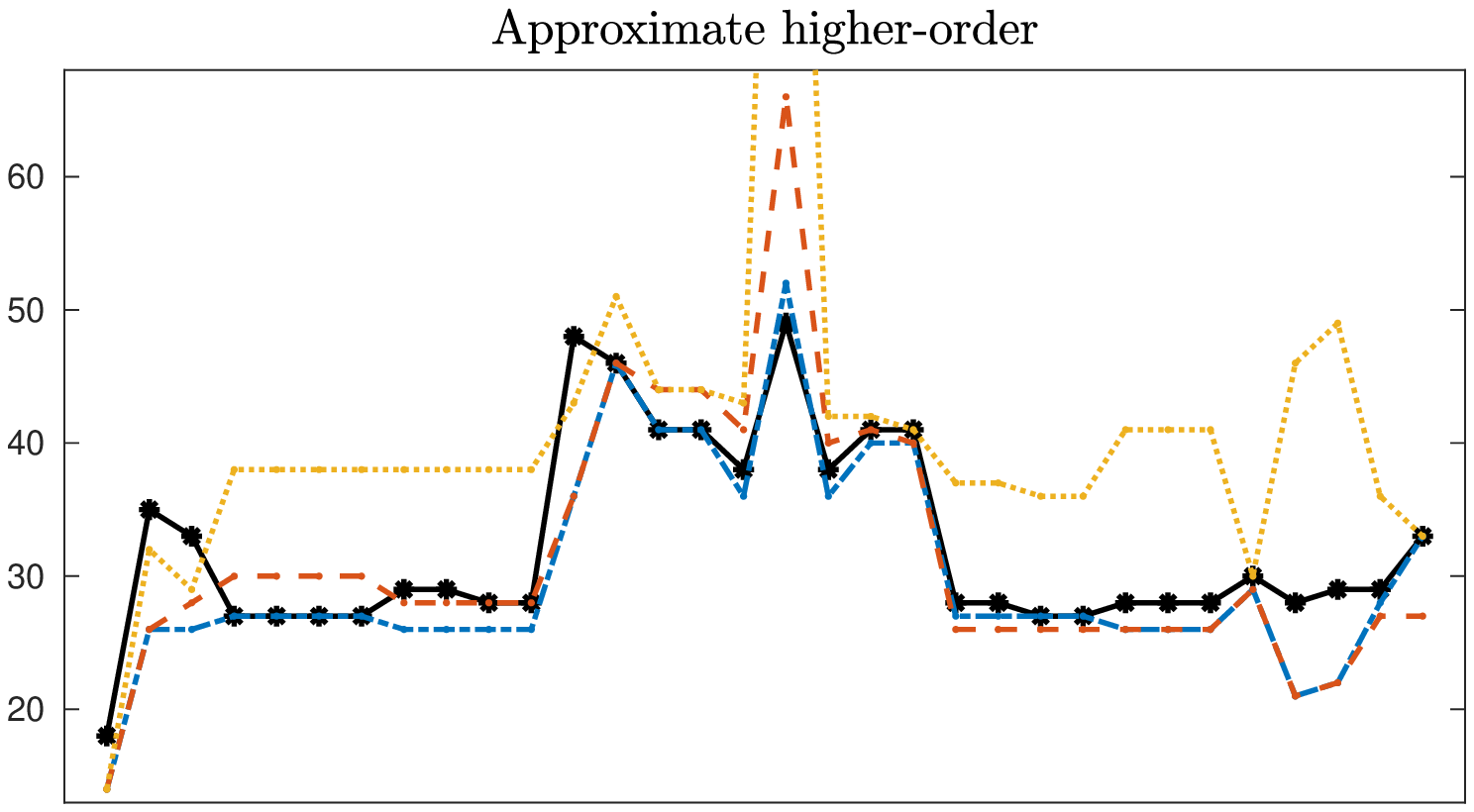} 
\end{minipage}
\parbox{6.1cm}{
\includegraphics[width=6.3cm]{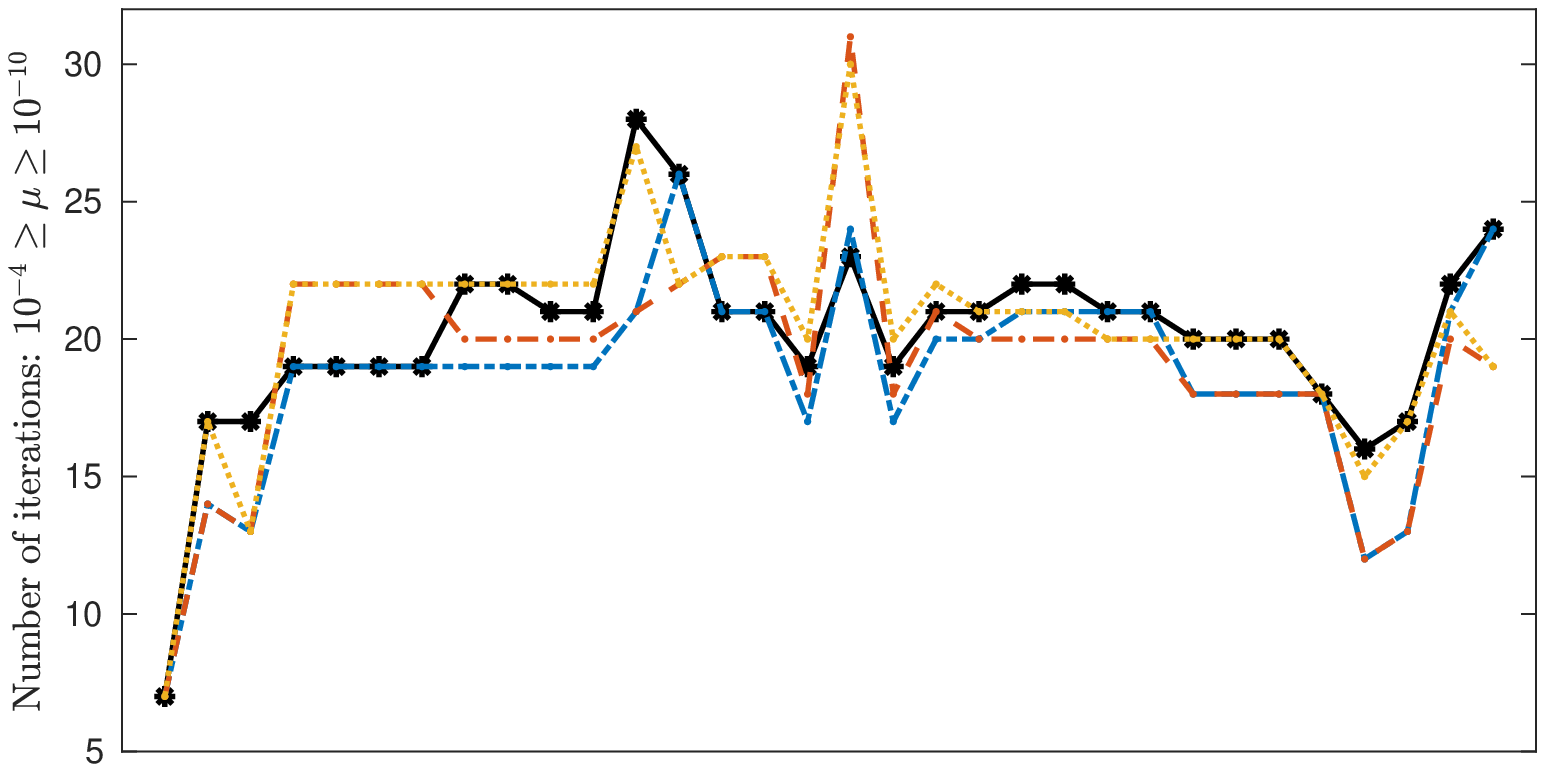}}
\hspace{-5mm}
\begin{minipage}{6.1cm}
\includegraphics[width=6.3cm]{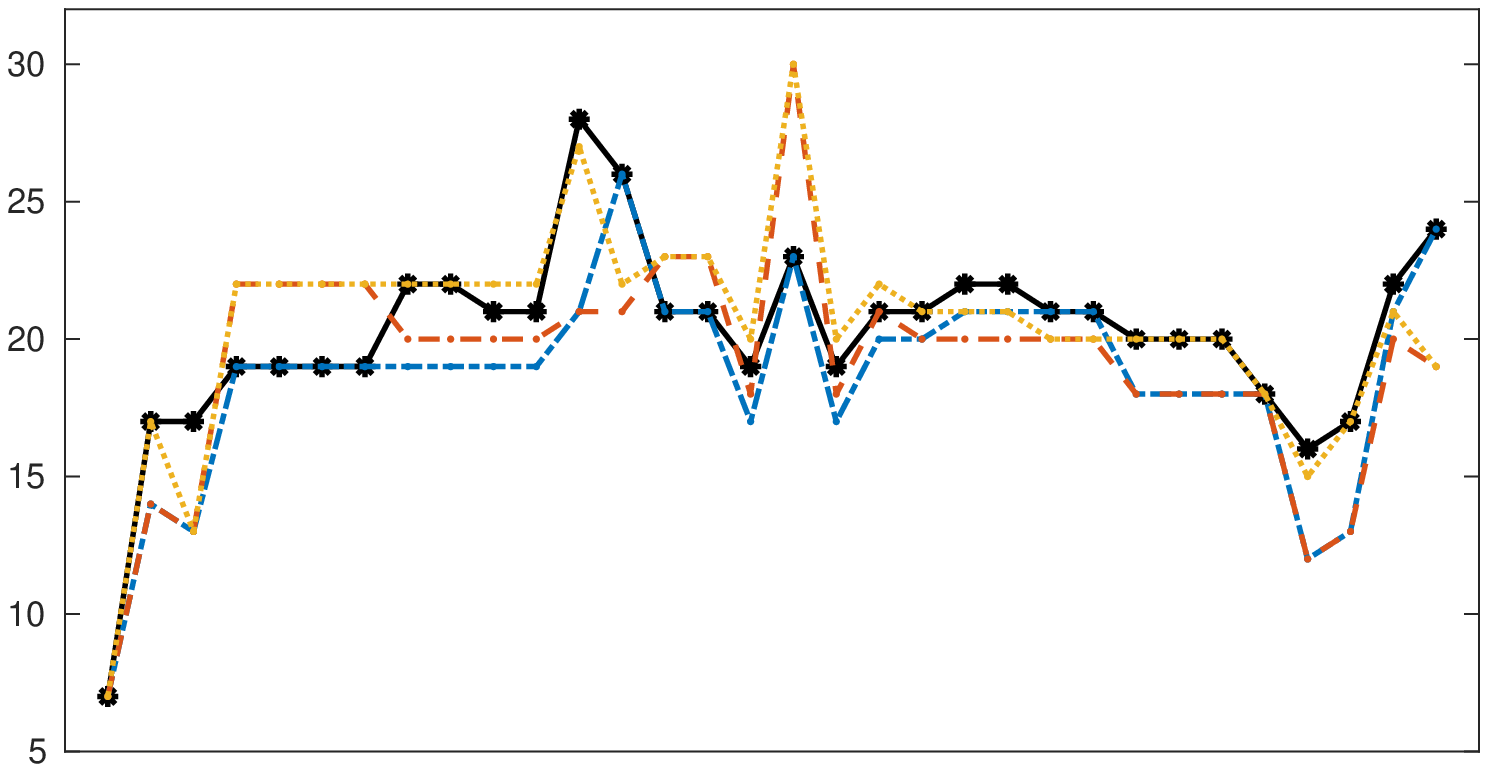} 
\end{minipage}
\parbox{6.1cm}{
\includegraphics[width=6.3cm]{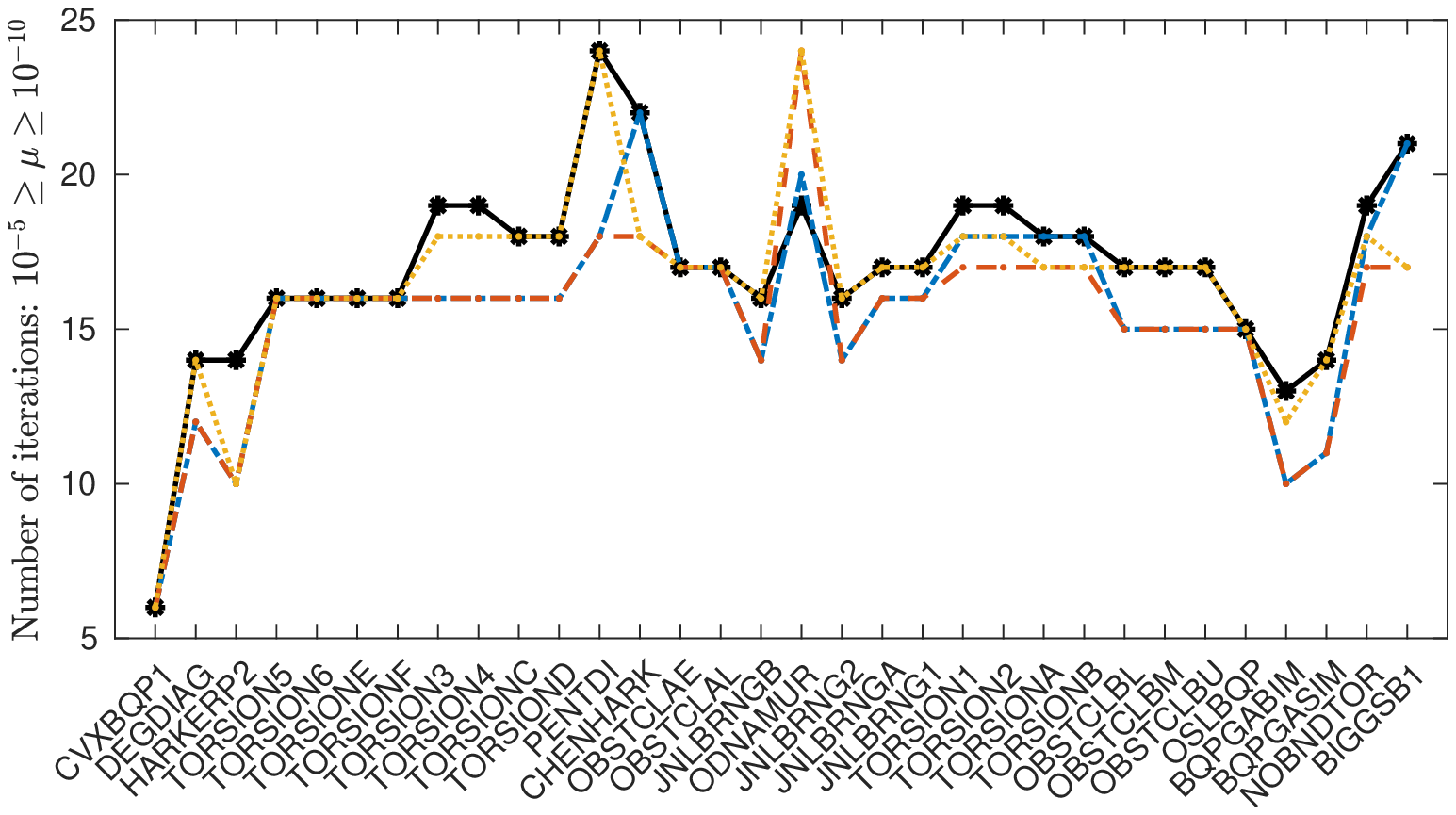}}
\hspace{-5mm} 
\begin{minipage}{6.1cm}
\includegraphics[width=6.3cm]{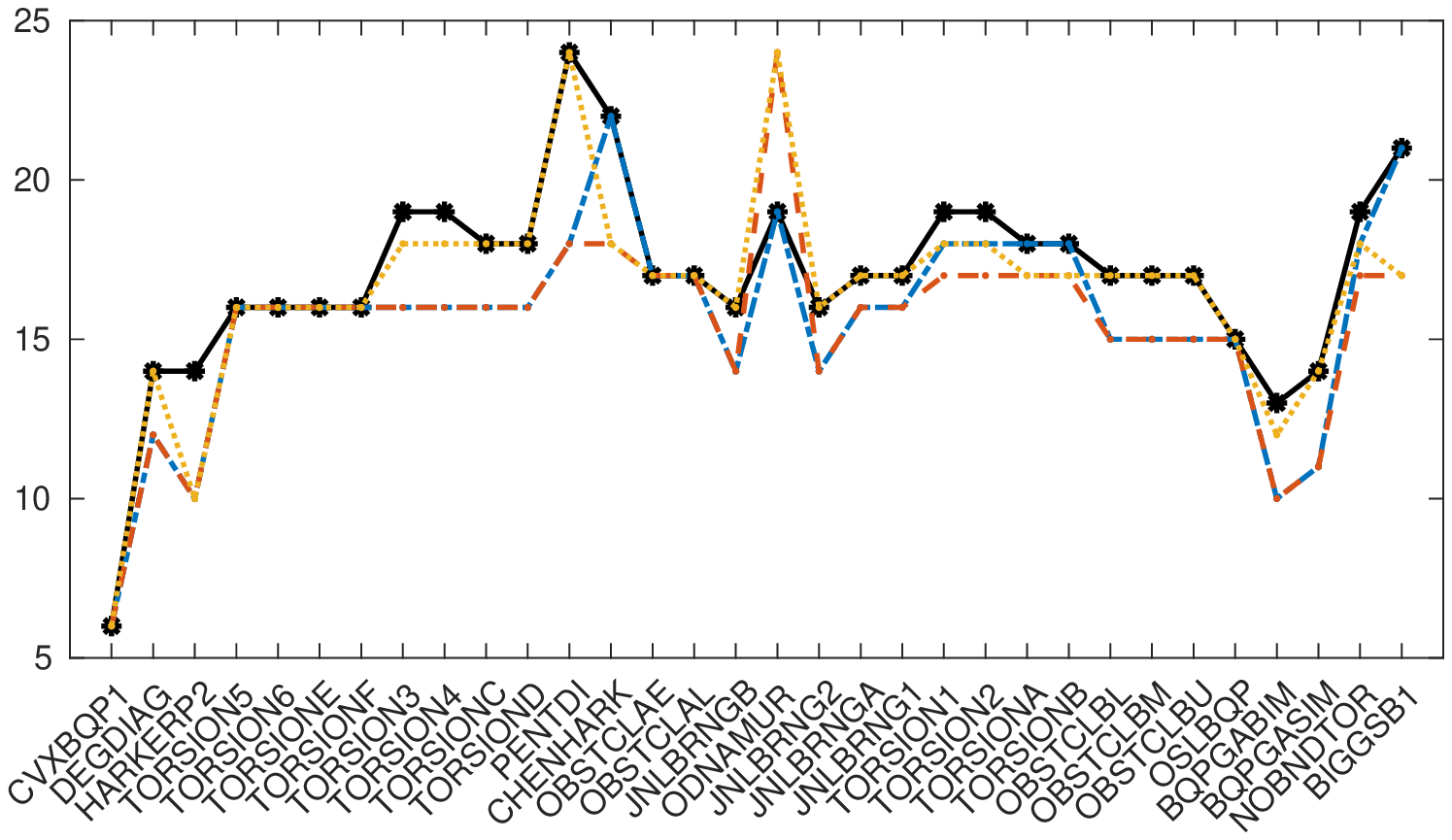} 
\end{minipage}
\caption{Number of iterations required at different intervals of $\mu$ for three versions of the Newton-like approaches and Algorithm~\ref{alg:IPMnewt}, (Newton).}
\label{fig:modNewtonIPMs}
\end{figure}
The total iteration count for $\mu \in [10^1,\> 10^{-10}]$ in
Figure~\ref{fig:modNewtonIPMs} shows that the approximate higher-order approach requires the same, or fewer iterations, compared to the approach with the approximate intermediate step. 
The iteration count for the approaches with the Schur-based approximate is similar to that of Algorithm~\ref{alg:IPMnewt} for this range of $\mu$. Also here, numerical
experiments show indications of three regions. For $\mu$ in the
approximate region $[10^1,10^{-2})$, the versions with the Schur-based
approximate yield a potentially reduced number of iterations. Their performance varies in the region of intermediate sized $\mu$. However,  it can not be discarded that this is an effect of the relatively simple set estimation heuristics. 
On all problems, with the exception of \texttt{ODNAMUR}, in Figure~\ref{fig:modNewtonIPMs} for $\mu \in [10^{-5},\> 10^{-10}]$ all versions of both approaches give an iteration
count less or equal to Algorithm~\ref{alg:IPMnewt}, hence providing potential savings in computation cost. The results may be improved with a flexible set estimation heuristics, e.g., more restrictive
thresholds for intermediate sized $\mu$. However, we chose not to include another layer of detail and instead give the
results for a relatively simple setting to obtain an initial evaluation of the potential performance.

\section{Conclusions} \label{sec:Conc}
In this work we have given approximate solutions to systems of linear equations that arise in interior point methods for bound-constrained optimization; in particular, partial approximate solutions, where the asymptotic component error bounds are in the order of $\mu^2$, and full approximate solutions with asymptotic error bounds in the order of $\mu^2$. Numerical simulations on randomly generated bound-constrained convex quadratic optimization problems, whose minimizers satisfy strict complementarity, have shown that the approximate solutions perform similarly to Newton solutions for sufficiently small $\mu$. Simulations on convex bound-constrained quadratic problems from the CUTEst test collection, whose minimizers typically do not satisfy strict complementarity, has shown that the predicted asymptotic behavior still occurs, however at significantly smaller values of $\mu$.

We have performed numerical simulations in a simple yet more realistic setting. Specifically, in a primal-dual interior-point framework where the active and inactive sets were estimated with basic heuristics as the iterations proceeded. These simulations were done on a selection of CUTEst benchmark problems. The results showed that the behavior roughly varied with three regions determined by the size of $\mu$. The Schur-based approximate solutions showed potential in the region for larger $\mu$, in the region of intermediate sized $\mu$ the performance varied, partly due to difficulties in determining the active and inactive sets. For sufficiently small $\mu$ the approximate solutions showed performance similar to our reference method while solving systems of reduced size. 

Finally we showed numerical results for two Newton-like approaches, which include an approximate intermediate step consisting of partial approximate solutions, on the considered CUTEst benchmark problems. The simulations showed similar characteristics as the previous results and also a potential for reducing the overall iteration count of interior-point methods. 

The results of this work are meant to contribute to the theoretical and numerical understanding for approximate solutions to systems of linear equations that arise in interior-point methods. We hope that the work can lead to further research on approximate solutions and approximate higher-order methods for optimization problems with linear inequality constraints. 

\subsubsection*{Acknowledgments}
  We thank the anonymous referees for many helpful suggestions which
  significantly improved the presentation.
\newpage
\begin{footnotesize}
\section*{Appendix}
\setcounter{section}{1}
\setcounter{table}{0}
\setcounter{equation}{0}
\counterwithin{table}{section}
\counterwithin{equation}{section}
\renewcommand{\thesection}{\Alph{section}}%
\subsection{The general case} \label{app:genCase}
Consider problems on the form of (\ref{eq:NLP}). In this situation the Lagrange multiplier vector at a local minimizer $x^*$ takes the form $\lambda^* = \begin{pmatrix} \lambda^{l*} \\ \lambda^{u*}
\end{pmatrix}$, where $\lambda^{l*}$ and $\lambda^{u*}$ are the multiplier vectors corresponding to lower and upper bounds respectively. The second-order conditions sufficient optimality conditions together with strict complementarity take the form
\begin{subequations}
\begin{align}
\nabla f(x^*) - \lambda^{l*} + \lambda^{u*} & = 0, \label{eq:genCase:OptCond:stationarity} \\
 x^*-l  \geq 0, \quad u-x^* & \geq 0,\label{eq:genCase:OptCond:feasibility} \\
\lambda^{l*} \geq 0, \quad \lambda^{u*} & \geq 0,  \label{eq:genCase:OptCond:nonnegofmultipliers}  \\
 (x^*-l) \cdot \lambda^{l*}  = 0,  \quad (u-x^*) \cdot \lambda^{u*} & = 0,  \label{eq:genCase:OptCond:complementarity}\\
 Z(x^*)^T \nabla^2 f(x^*) Z (x^*) & \succ 0. \\
  x + \lambda^{l*}  > 0, \quad x + \lambda^{u*} & > 0. \label{eq:genCase:OptCond:strictcomplementarity}
\end{align} \label{eq:genCase:OptCond}
\end{subequations}
Similarly as in Section~\ref{sec:Background}, define the function $F_{\mu}:\mathbb{R}^{3n} \rightarrow \mathbb{R}^{3n}$ by 
\begin{equation} \label{eq:genCase:Fmu}
F_{\mu}(x,\lambda) = \begin{bmatrix}
\nabla f(x) - \lambda^l + \lambda^u \\
\Lambdait^l (X-L) e - \mu e \\
\Lambdait^u (U-X) e - \mu e
\end{bmatrix},
\end{equation}
where $L= \textrm{diag} (l)$, $U = \textrm{diag}(u)$, $\Lambdait^l=\textrm{diag}(\lambda^l)$ and $\Lambdait^u=\textrm{diag}(\lambda^u)$. The corresponding Jacobian $F' : \mathbb{R}^{3n} \rightarrow \mathbb{R}^{3n}$ is
\begin{equation} \label{eq:genCase:Jac}
F'(x, \lambda) = \begin{bmatrix}
H & -I & I\\ 
\Lambdait^l & (X-L) & \\ 
-\Lambdait^u & & (U-X)
\end{bmatrix}.
\end{equation}
For the case with upper and lower bounds it is useful to distinguish whether a specific component of $x^*$ is active with respect to an upper or a lower bound. 
\begin{definition} \label{def:genCase:sets}
(Active/inactive sets). For a given $x^*$ such that $l \le x^* \le u $, define the sets
\begin{align*}
& \mathcal{A}_{l}  = \{i \in \{1, \dots, n \} : x^*_i - l_i = 0 \}, 
& \mathcal{I}_{l} = \{1,\dots, n\} \setminus \mathcal{A}_{l}, \> \> \\
& \mathcal{A}_{u} = \{i \in \{1, \dots, n \} : u_i - x^*_i = 0 \},
& \mathcal{I}_{u} = \{1,\dots, n\} \setminus \mathcal{A}_{u} , \\
& \mathcal{A}_x  = \mathcal{A}_{l}  \cup \mathcal{A}_{u}, 
& \mathcal{I}_x  = \{1,\dots, n\} \setminus \mathcal{A}_x. 
\end{align*}
\end{definition}
Throughout the remaining part of the manuscript, Assumption~\ref{ass1} means that the vector $(x^*, \lambda^*)$ satisfies (\ref{eq:genCase:OptCond}), i.e., second-order sufficient optimality conditions and strict complementarity.
Bounds on individual components of the solution $(x, \lambda)$ in the region of asymptotic behavior is given the lemma below. 
\begin{lemma} \label{lemma:genCase:Order}
Under Assumption~\ref{ass1}, let $\mathcal{B}\left((x^*, \lambda^*), \delta \right)$ and $\muhat$ be defined by Lemma~\ref{lemma:background:FpnonsingCont} and Lemma~\ref{lemma:LipcPath} respectively. Then there exists $\mubar$, with
$0 <\mubar \le \muhat$, such that for $0 < \mu \leq \bar\mu$ and $(x,\lambda)$ sufficiently close to $(x^{\mu}, \lambda^{\mu}) \in \mathcal{B}((x^*, \lambda^*), \delta)$ so that $\| F_{\mu} (x,\lambda) \| = \mathcal{O}(\mu)$ it holds that
\begin{equation*}
x_i-l_i = \begin{cases} \mathcal{O}(\mu) & i \in \mathcal{A}_l, \\ 
                    \Theta(1) & i \in \mathcal{I}_l, \end{cases} 
                    \qquad 
\lambda_i^l = \begin{cases} \Theta(1) & i \in \mathcal{A}_{l}, \\ 
                    \mathcal{O}(\mu) & i \in \mathcal{I}_{l}, \end{cases}
\end{equation*}
\begin{equation*}
u_i-x_i = \begin{cases} \mathcal{O}(\mu) & i \in \mathcal{A}_u, \\ 
                    \Theta(1) & i \in \mathcal{I}_u, \end{cases}
\qquad \lambda_i^u = \begin{cases} \Theta(1) & i \in \mathcal{A}_{u}, \\ 
                    \mathcal{O}(\mu) & i \in \mathcal{I}_{u}. \end{cases}
\end{equation*}
\end{lemma}
\subsubsection*{Partial approximate solutions }
In this section we give results analogous to those given in Section~\ref{subsec:partialApp} together with some complementary remarks.  With $F'(x,\lambda)$ and $F_{\mu}(x,\lambda)$ defined as in (\ref{eq:genCase:Jac}) and (\ref{eq:genCase:Fmu}) respectively the Schur complement of $(X-L)$ and $(U-X)$ in (\ref{eq:PDsyst}) is
\begin{align} \label{eq:genCase:schurComplOfPDsyst}
\left( H + (X-L)\inv \Lambdait^l + (U-X)\inv \Lambdait^u \right) \Delta x^N & = - \nabla f(x) \nonumber \\
& \quad + \mu \left[ (X-L)\inv - (U-X)\inv \right] e.
\end{align}
For $i \in \mathcal{A}_x$ either $(u_i-x_i) \to 0$ or $(l_i-x_i) \to 0$ as $\mu \to 0$. In consequence, approximates of $\Delta x_i^N$, $i \in \mathcal{A}_x$, can be obtained from the Schur complement (\ref{eq:genCase:schurComplOfPDsyst}). These approximate solutions are given below in Proposition~\ref{prop:genCase:schurBased} which is the result analogous to Proposition~\ref{prop:schurBased}.
\begin{proposition}\label{prop:genCase:schurBased}
Under Assumption~\ref{ass1}, let $\mathcal{B}\left((x^*, \lambda^*),
  \delta\right)$ and $\muhat$ be defined by
Lemma~\ref{lemma:background:FpnonsingCont} and
Lemma~\ref{lemma:LipcPath} respectively. For $(x,\lambda) \in \mathcal{B}((x^*, \lambda^*), \delta)$, let $(\Delta x^N, \Delta \lambda^N)$ be the solution of (\ref{eq:PDsyst}) with $\mu^+ = \sigma \mu$, where $0< \sigma < 1$. If the search direction components are defined as
\begin{equation} \label{eq:prop:genCase:schurBased:dx}
\Delta x_i^S  = \frac{-1}{ \left[ \nabla^2 f(x)\right]_{ii} + \frac{\lambda_i^l}{x_i - l_i} + \frac{\lambda_i^u}{u_i-x_i}} \left( \left[ \nabla f(x) \right]_i  - \mu^+ \left[ \frac{1}{x_i -l_i} - \frac{1}{u_i- x_i}\right]  \right),
\end{equation}
for $i=1,\dots, n$, then
\begin{equation*}
\Delta x_i^S - \Delta x_i^N =  \frac{1}{ \left[ \nabla^2 f(x)\right]_{ii} + \frac{\lambda_i^l}{x_i - l_i} + \frac{\lambda_i^u}{u_i-x_i}}   \sum_{i \neq j} \left[ \nabla^2 f(x)\right]_{ij} \Delta x_j^N, \qquad i=1,\dots, n.
\end{equation*}
Assume in addition that $0 < \mu \le \muhat$ and $(x,\lambda)$ is sufficiently close to $(x^{\mu}, \lambda^{\mu})\in \mathcal{B}\left((x^*, \lambda^*), \delta\right)$ such that $\| F_{\mu} (x,\lambda) \| = \mathcal{O}(\mu)$. Then there exists $\mubar$, with
$0 <\mubar \le \muhat$, such that for $0 < \mu \leq \mubar$ it holds that
\begin{equation*}
\frac{1}{\left[ \nabla^2 f(x)\right]_{ii} + \frac{\lambda_i^l}{x_i - l_i} + \frac{\lambda_i^u}{u_i-x_i}} = \begin{cases} \mathcal{O}(\mu) & i \in \mathcal{A}_x, \\ \Theta(1) & i \in \mathcal{I}_x, \end{cases}
\end{equation*}
and
\begin{align*}
&| \Delta x_i^S - \Delta x^N_i | = \mathcal{O}(\mu^2),  &i \in \mathcal{A}_x.
\end{align*}
\end{proposition}
Next we give results corresponding to those in Proposition~\ref{prop:compBased}. As $\mu \to 0$ then  $\lambda_i^l \to 0$ for $i \in \mathcal{I}_{l}$ and $\lambda_i^u \to 0$ for $i \in \mathcal{I}_{u}$. Consequently, approximations based on the complementarity blocks of $F'(x,\lambda) (\Delta x^N, \Delta \lambda^N) = - F_{\mu}(x,\lambda)$  can be formed for $\Delta x_i^N, i \in \mathcal{A}_x$, $\Delta \lambda_i^{l,N}, i \in \mathcal{I}_{l}$ and $\Delta \lambda_i^{u,N}, i \in \mathcal{I}_{u}$.
\begin{proposition} \label{prop:genCase:compBased}
Under Assumption~\ref{ass1}, let $\mathcal{B}\left((x^*, \lambda^*),
  \delta\right)$ and $\muhat$ be defined by
Lemma~\ref{lemma:background:FpnonsingCont} and
Lemma~\ref{lemma:LipcPath} respectively. For $(x,\lambda) \in \mathcal{B}((x^*, \lambda^*), \delta)$, let $(\Delta x^N, \Delta \lambda^N)$ be the solution of (\ref{eq:PDsyst}) with $\mu^+ = \sigma \mu$, where $0< \sigma < 1$. If the search direction components are defined as
\begin{subequations} \label{eq:prop:genCase:compBased:dx}
\begin{align}
& \Delta x_i^C = - (x_i-l_i) + \frac{\mu^+}{\lambda_i^l}, & i \in \mathcal{A}_{l}, \> \> \label{eq:prop:genCase:compBased:dxLambdal} \\
&\Delta x_i^C = (u_i-x_i) - \frac{\mu^+}{\lambda^u_i}, & i \in \mathcal{A}_{u}, \> \label{eq:prop:genCase:compBased:dxLambdau} 
\end{align}
\end{subequations}
\begin{subequations} \label{eq:prop:genCase:compBased:dlambda}
\begin{align}
& \Delta \lambda_i^{l,C} = - \lambda_i^l + \frac{\mu^+}{x_i-l_i}, & i \in \mathcal{I}_{l}, \> \> \> \label{eq:prop:genCase:compBased:dlambdal} \\
& \Delta \lambda^{u,C}_i = - \lambda^u_i + \frac{\mu^+}{u_i-x_i}, & i \in \mathcal{I}_{u}, \> \>  \label{eq:prop:genCase:compBased:dlambdau} 
\end{align}
\end{subequations}
then
\begin{align*} 
& \Delta x_i^C - \Delta x_i^N =  \frac{x_i-l_i}{\lambda_i^l} \Delta \lambda_i^{l,N}, & i \in \mathcal{A}_{l}, \qquad \> \\
& \Delta x_i^C - \Delta x_i^N =  \frac{u_i-x_i}{\lambda_i^u} \Delta \lambda_i^{u,N}, & i \in \mathcal{A}_{u},\qquad  \\ 
& \Delta \lambda_i^{l,C} - \Delta \lambda_i^{l,N} = \frac{\lambda_i^l}{x_i-l_i} \Delta x_i^N, & i \in \mathcal{I}_{l}, \qquad \>\> \\
& \Delta \lambda_i^{u,C} - \Delta \lambda_i^{u,N} = -\frac{\lambda^u_i}{u_i-x_i} \Delta x_i^N, & i \in \mathcal{I}_{u}. \qquad \> 
\end{align*}
Assume in addition that $0 < \mu \le \muhat$ and $(x,\lambda)$ is sufficiently close to $(x^{\mu}, \lambda^{\mu})\in \mathcal{B}\left((x^*, \lambda^*), \delta\right)$ such that $\| F_{\mu} (x,\lambda) \| = \mathcal{O}(\mu)$. Then there exists $\mubar$, with
$0 <\mubar \le \muhat$, such that for $0 < \mu \leq \mubar$ it holds that
\begin{subequations} \label{eq:prop:genCase:compBased:Ordo}
\begin{align*} 
| \Delta x_i^C - \Delta x_i^N | & = \mathcal{O}(\mu^2), & i \in\mathcal{A}_x, \\
| \Delta \lambda_i^{l,C} - \Delta \lambda_i^{l,N} | & = \mathcal{O}(\mu^2), & i \in\mathcal{I}_{l}, \> \>  \\  
| \Delta \lambda_i^{u,C} - \Delta \lambda_i^{u,N} | & =  \mathcal{O}(\mu^2), & i \in\mathcal{I}_{u}. \> 
\end{align*} 
\end{subequations}
\end{proposition} 
Finally we give the general result for the approximate intermediate step, i.e., for the case with lower and upper bounds.
\begin{proposition} \label{prop:genCase:Astep}
Under Assumption~\ref{ass1}, let $\mathcal{B}\left((x^*, \lambda^*),
  \delta\right)$ and $\muhat$ be defined by
Lemma~\ref{lemma:background:FpnonsingCont} and
Lemma~\ref{lemma:LipcPath} respectively. For $(x,\lambda) \in \mathcal{B}((x^*, \lambda^*), \delta)$, define $(x_+^{N},\lambda_+^{N})=(x,\lambda)+ (\Delta x^N,\Delta \lambda^N)$ where $(\Delta x^N, \Delta \lambda^N)$ is the solution of (\ref{eq:PDsyst}) with $\mu^+ = \sigma \mu$, where $0< \sigma < 1$. Moreover, let $(x_+, \lambda_+) = (x,\lambda) + (\Delta x, \Delta \lambda)$ where 
\begin{equation} \label{eq:prop:genCase:partialAstep}
\Delta x_i = \begin{cases}(\ref{eq:prop:genCase:schurBased:dx}) \mbox{ or }  (\ref{eq:prop:genCase:compBased:dxLambdal}) & i \in \mathcal{A}_l, \\ 
(\ref{eq:prop:genCase:schurBased:dx}) \mbox{ or }  (\ref{eq:prop:genCase:compBased:dxLambdau}) & i \in \mathcal{A}_u, \\
0 & i \in \mathcal{I}_x,  \end{cases} \qquad \Delta \lambda_i = \begin{cases} 0 & i \in \mathcal{A}_x, \\
(\ref{eq:prop:genCase:compBased:dlambdal}) & i \in \mathcal{I}_l,\\
(\ref{eq:prop:genCase:compBased:dlambdau}) & i \in \mathcal{I}_u. \end{cases} 
\end{equation}
Assume that $0 < \mu \le \muhat$ and $(x,\lambda)$ is sufficiently close to $(x^{\mu}, \lambda^{\mu})\in \mathcal{B}\left((x^*, \lambda^*), \delta\right)$ such that $\| F_{\mu} (x,\lambda) \| = \mathcal{O}(\mu)$ and $\| (\Delta x_\mathcal{A}^N, \Delta \lambda_\mathcal{I}^N) \|
=\Omega ( \mu^\gamma )$ for $\gamma < 2$. Then there exists $\mubar$, with
$0 <\mubar \le \muhat$, such that for $0 < \mu \leq \mubar$ it holds that
\begin{equation*}
\| (x_+^N,\lambda_+^N) - (x_+,\lambda_+)  \|  \le \| (x_+^N,\lambda_+^N) - (x,\lambda)  \|.
\end{equation*}
\end{proposition} 

\subsubsection*{Full approximate solutions}
In this section we give results analogous to those given in Section~\ref{subsec:fullApp} together with some complementary remarks. 
Note that $\mathcal{I}_x \cap \mathcal{A}_{l} = \emptyset$ and $\mathcal{I}_x \cap \mathcal{A}_{u}= \emptyset$.  By partitioning \\$(\Delta x^N, \Delta \lambda^N) = (\Delta x_{\mathcal{A}_x}^N, \Delta x_{\mathcal{I}_x}^N, \Delta \lambda^{l,N}_{\mathcal{A}_{l}}, \Delta \lambda^{l,N}_{\mathcal{I}_{l}}, \Delta \lambda^{u,N}_{\mathcal{A}_{u}}, \Delta \lambda^{u,N}_{\mathcal{I}_{u}})$, (\ref{eq:PDsyst}) can be written as
\begingroup\makeatletter\def\f@size{7}\check@mathfonts
\def\maketag@@@#1{\hbox{\m@th\footnotesize\normalfont#1}}%
\begin{align}\label{eq:genCase:PDsyst_partitioned}
&\begin{bmatrix}
H_{\mathcal{A}_x \mathcal{A}_x}  & H_{\mathcal{A}_x \mathcal{I}_x}  & -I_{\mathcal{A}_x \mathcal{A}_{l}} & -I_{\mathcal{A}_x \mathcal{I}_{l}} & I_{\mathcal{A}_x \mathcal{A}_{u}} & I_{\mathcal{A}_x \mathcal{I}_{u}}  \\
H_{\mathcal{I}_x \mathcal{A}_x} & H_{\mathcal{I}_x \mathcal{I}_x}  &  & -I_{\mathcal{I}_x \mathcal{I}_{l}} &  & I_{\mathcal{I}_x \mathcal{I}_{u}} \\
\Lambdait^l_{\mathcal{A}_{l} \mathcal{A}_x} &  & (X-L)_{\mathcal{A}_{l} \mathcal{A}_{l}} &  & & \\
\Lambdait^l_{\mathcal{I}_l \mathcal{A}_x} & \Lambdait^l_{\mathcal{I}_{l} \mathcal{I}_x} & & (X-L)_{\mathcal{I}_{l} \mathcal{I}_{l}} &  & \\
-\Lambdait^u_{\mathcal{A}_{u} \mathcal{A}_x} &  & & & (U-X)_{\mathcal{A}_{u} \mathcal{A}_{u}} & \\
-\Lambdait^u_{\mathcal{I}_{u} \mathcal{A}_x} & -\Lambdait^u_{\mathcal{I}_{u} \mathcal{I}_x} & & &  & (U-X)_{\mathcal{I}_{u} \mathcal{I}_{u}} 
\end{bmatrix} \begin{bmatrix} \Delta x_{\mathcal{A}_x}^N \\ \Delta x_{\mathcal{I}_x}^N \\ \Delta \lambda^{l,N}_{\mathcal{A}_{l}} \\ \Delta \lambda^{l,N}_{\mathcal{I}_{l}} \\ \Delta \lambda^{u,N}_{\mathcal{A}_{u}} \\ \Delta \lambda^{u,N}_{\mathcal{I}_{u}}
\end{bmatrix} = \nonumber \\
 &  -\begin{bmatrix}
\nabla f(x)_{\mathcal{A}_x} - \lambda^l_{\mathcal{A}_x} + \lambda^u_{\mathcal{A}_x} \\
\nabla f(x)_{\mathcal{I}_x} - \lambda^l_{\mathcal{I}_x} + \lambda^u_{\mathcal{I}_x}\\
\Lambdait^l_{\mathcal{A}_{l} \mathcal{A}_{l}} (X-L)_{\mathcal{A}_{l} \mathcal{A}_{l}} e - \mu e \\
\Lambdait^l_{\mathcal{I}_{l} \mathcal{I}_{l}} (X-L)_{\mathcal{I}_{l} \mathcal{I}_{l}} e - \mu e \\
\Lambdait^u_{\mathcal{A}_{u}\mathcal{A}_{u}} (U-X)_{\mathcal{A}_{u} \mathcal{A}_{u}} e - \mu e \\
\Lambdait^u_{\mathcal{I}_{u} \mathcal{I}_{u}} (U-X)_{\mathcal{I}_{u} \mathcal{I}_{u}} e - \mu e
\end{bmatrix},
\end{align}\endgroup
Suppose that an approximate solution of $\Delta x^N_{\mathcal{A}_x}$ is given, e.g., (\ref{eq:prop:genCase:schurBased:dx}) or (\ref{eq:prop:genCase:compBased:dxLambdal}) and (\ref{eq:prop:genCase:compBased:dxLambdau}) of Proposition~\ref{prop:genCase:schurBased} and Proposition~\ref{prop:genCase:compBased} respectively. Insertion of an approximate $\Delta x_{\mathcal{A}_x}$ into (\ref{eq:genCase:PDsyst_partitioned}) yields
\begin{align}\label{eq:genCase:redPDsyst_partitioned}
&\begin{bmatrix}
H_{\mathcal{A}_x \mathcal{I}_x}  & -I_{\mathcal{A}_x \mathcal{A}_{l}} & -I_{\mathcal{A}_x \mathcal{I}_{l}} & I_{\mathcal{A}_x \mathcal{A}_{u}} & I_{\mathcal{A}_x \mathcal{I}_{u}}  \\
H_{\mathcal{I}_x \mathcal{I}_x}  &  & -I_{\mathcal{I}_x \mathcal{I}_{l}} &  & I_{\mathcal{I}_x \mathcal{I}_{u}} \\
& (X-L)_{\mathcal{A}_{l} \mathcal{A}_{l}} &  & & \\
\Lambdait^l_{\mathcal{I}_{l} \mathcal{I}_x} & & (X-L)_{\mathcal{I}_{l} \mathcal{I}_{l}} &  & \\
 & & & (U-X)_{\mathcal{A}_{u} \mathcal{A}_{u}} & \\
-\Lambdait^u_{\mathcal{I}_{u} \mathcal{I}_x} & & &  & (U-X)_{\mathcal{I}_{u} \mathcal{I}_{u}} 
\end{bmatrix} \begin{bmatrix} \Delta x^{ls}_{\mathcal{I}_x} \\ \Delta \lambda^{l,ls}_{\mathcal{A}_{l}} \\ \Delta \lambda^{l,ls}_{\mathcal{I}_{l}} \\ \Delta \lambda^{u,ls}_{\mathcal{A}_{u}} \\ \Delta \lambda^{u,ls}_{\mathcal{I}_{u}}
\end{bmatrix} = \nonumber \\
&  -\begin{bmatrix}
\nabla f(x)_{\mathcal{A}_x} - \lambda^l_{\mathcal{A}_x} + \lambda^u_{\mathcal{A}_x}  + H_{\mathcal{A}_x \mathcal{A}_x}  \Delta x_{\mathcal{A}_x} \\
\nabla f(x)_{\mathcal{I}_x} - \lambda^l_{\mathcal{I}_x} + \lambda^u_{\mathcal{I}_x} + H_{\mathcal{I}_x \mathcal{A}_x}  \Delta x_{\mathcal{A}_x} \\
\Lambdait^l_{\mathcal{A}_{l} \mathcal{A}_{l}} (X-L)_{\mathcal{A}_{l} \mathcal{A}_{l}} e - \mu e  + \Lambdait^l_{\mathcal{A}_{l} \mathcal{A}_x}  \Delta x_{\mathcal{A}_x} \\
\Lambdait^l_{\mathcal{I}_{l} \mathcal{I}_{l}} (X-L)_{\mathcal{I}_{l} \mathcal{I}_{l}} e - \mu e + \Lambdait^l_{\mathcal{I}_{l} \mathcal{A}_x}  \Delta x_{\mathcal{A}_x} \\
\Lambdait^u_{\mathcal{A}_{u}\mathcal{A}_{u}} (U-X)_{\mathcal{A}_{u} \mathcal{A}_{u}} e - \mu e - \Lambdait^u_{\mathcal{A}_{u} \mathcal{A}_x}  \Delta x_{\mathcal{A}_x}\\
\Lambdait^u_{\mathcal{I}_{u} \mathcal{I}_{u}} (U-X)_{\mathcal{I}_{u} \mathcal{I}_{u}} e - \mu e - \Lambdait^u_{\mathcal{I}_{u} \mathcal{A}_x}  \Delta x_{\mathcal{A}_x},
\end{bmatrix},
\end{align}
whose solution is labeled with ``$ls$'' since it will lead to least squares system, similarly as in Section~\ref{subsec:fullApp}.
The second, fourth and sixth block of (\ref{eq:genCase:redPDsyst_partitioned}) provide unique solutions of $\Delta x^{ls}_{\mathcal{I}_x}$, $\Delta \lambda^{l,ls}_{\mathcal{I}_l}$ and $\Delta \lambda^{u,ls}_{\mathcal{I}_u}$ which satisfy
\begin{align} \label{eq:genCase:Redx2syst}
\begin{bmatrix}
H_{\mathcal{I}_x \mathcal{I}_x}  &  -I_{\mathcal{I}_x \mathcal{I}_{l}} & I_{\mathcal{I}_x \mathcal{I}_{u}} \\
\Lambdait^l_{\mathcal{I}_{l} \mathcal{I}_x} & (X-L)_{\mathcal{I}_{l} \mathcal{I}_{l}}  \\
-\Lambdait^u_{\mathcal{I}_{u} \mathcal{I}_x} & & (U-X)_{\mathcal{I}_{u} \mathcal{I}_{u}} 
\end{bmatrix} \begin{bmatrix} \Delta x^{ls}_{\mathcal{I}_x} \\ \Delta \lambda^{l,ls}_{\mathcal{I}_{l}} \\ \Delta \lambda^{u,ls}_{\mathcal{I}_{u}}
\end{bmatrix} =  \nonumber \\
  -\begin{bmatrix}
\nabla f(x)_{\mathcal{I}_x} - \lambda^l_{\mathcal{I}_x} + \lambda^u_{\mathcal{I}_x} + H_{\mathcal{I}_x \mathcal{A}_x}  \Delta x_{\mathcal{A}_x} \\
\Lambdait^l_{\mathcal{I}_{l} \mathcal{I}_{l}} (X-L)_{\mathcal{I}_{l} \mathcal{I}_{l}} e - \mu e + \Lambdait^l_{\mathcal{I}_{l} \mathcal{A}_x}  \Delta x_{\mathcal{A}_x} \\
\Lambdait^u_{\mathcal{I}_{u} \mathcal{I}_{u}} (U-X)_{\mathcal{I}_{u} \mathcal{I}_{u}} e - \mu e - \Lambdait^u_{\mathcal{I}_{u} \mathcal{A}_x}  \Delta x_{\mathcal{A}_x}
\end{bmatrix}.
\end{align}
The solution of (\ref{eq:genCase:Redx2syst}) can be obtained by first solving with the Schur complement of $ (X-L)_{\mathcal{I}_{l} \mathcal{I}_{l}}$ and $(U-X)_{\mathcal{I}_{u} \mathcal{I}_{u}} $
\begin{align}  \label{eq:genCase:Redx2syst_schur}
&\left(H_{\mathcal{I}_x \mathcal{I}_x}  + I_{\mathcal{I}_x \mathcal{I}_{l}} (X-L)_{\mathcal{I}_{l} \mathcal{I}_{l}} \inv \Lambdait^l_{\mathcal{I}_{l} \mathcal{I}_x} +  I_{\mathcal{I}_x \mathcal{I}_{u}} (U-X)_{\mathcal{I}_{u} \mathcal{I}_{u}} \inv    \Lambdait^u_{\mathcal{I}_{u} \mathcal{I}_x} \right) \Delta x^{ls}_\mathcal{I} \nonumber \\
&= -\left( \nabla f(x)_{\mathcal{I}_x} + H_{\mathcal{I}_x \mathcal{A}_x}  \Delta x_{\mathcal{A}_x} \right) + I_{\mathcal{I}_x \mathcal{I}_{l}} (X-L)_{\mathcal{I}_{l} \mathcal{I}_{l}}\inv \left(   \mu e - \Lambdait^l_{\mathcal{I}_{l} \mathcal{A}_x}  \Delta x_{\mathcal{A}_x} \right)  \nonumber \\
& \> \quad - I_{\mathcal{I}_x \mathcal{I}_{u}} (U-X)_{\mathcal{I}_{u} \mathcal{I}_{u}} \inv  \left(   \mu e + \Lambdait^u_{\mathcal{I}_{u} \mathcal{A}_x}  \Delta x_{\mathcal{A}_x} \right), 
\end{align}
 and then 
\begin{subequations} \label{eq:genCase:fullApprox:dlambdaI}
\begin{align}
\Delta \lambda^{l,ls}_{\mathcal{I}_l} & = - \lambda^l_{\mathcal{I}_{l}} + (X-L)_{\mathcal{I}_{l} \mathcal{I}_{l}}\inv \left( \mu e - \Lambdait^l_{\mathcal{I}_{l} \mathcal{A}_x}  \Delta x_{\mathcal{A}_x} - \Lambdait^l_{\mathcal{I}_{l} \mathcal{I}_x} \Delta x^{ls}_{\mathcal{I}_x}  \right), \label{eq:genCase:fullApprox:dlambdalI} \\
\Delta \lambda^{u,ls}_{\mathcal{I}_u} & =  - \lambda^u_{\mathcal{I}_{u}}  + (U-X)_{\mathcal{I}_{u} \mathcal{I}_{u}}\inv\left( \mu e - \Lambdait^u_{\mathcal{I}_{u} \mathcal{A}_x}  \Delta x_{\mathcal{A}_x} + \Lambdait^u_{\mathcal{I}_{u} \mathcal{I}_x} \Delta x^{ls}_{\mathcal{I}_x} \right). \label{eq:genCase:fullApprox:dlambdauI}
\end{align}
\end{subequations}
Note that the matrix of (\ref{eq:genCase:Redx2syst_schur}) is by Assumption~\ref{ass1} a symmetric positive definite $ \left( |\mathcal{I}_x| \times |\mathcal{I}_x| \right)$-matrix. The remanding part of the solution of (\ref{eq:genCase:redPDsyst_partitioned}), that is $\Delta \lambda^{l,ls}_{\mathcal{A}_{l}}$ and $\Delta \lambda^{u,ls}_{\mathcal{A}_{u}}$ are then given by

\begin{align} \label{eq:genCase:LS:overdet}
&\begin{bmatrix}
 -I_{\mathcal{A}_x \mathcal{A}_{l}} & I_{\mathcal{A}_x \mathcal{A}_{u}} \\
(X-L)_{\mathcal{A}_{l} \mathcal{A}_{l}} & \\
 & (U-X)_{\mathcal{A}_{u} \mathcal{A}_{u}}  \\ 
\end{bmatrix} \begin{bmatrix} \Delta \lambda^{l,ls}_{\mathcal{A}_{l}} \\ \Delta \lambda^{u,ls}_{\mathcal{A}_{u}}
\end{bmatrix} = \nonumber \\
&  -\begin{bmatrix}
\nabla_x \mathcal{L}(x,\lambda)_{\mathcal{A}_x} + H_{\mathcal{A}_x \mathcal{A}_x}  \Delta x_{\mathcal{A}_x} + H_{\mathcal{A}_x \mathcal{I}_x}  \Delta x^{ls}_{\mathcal{I}_x} - I_{\mathcal{A}_x \mathcal{I}_{l}}   \Delta \lambda^{l,ls}_{\mathcal{I}_{l}} + I_{\mathcal{A}_x \mathcal{I}_{u}} \Delta \lambda^{u,ls}_{\mathcal{I}_{u}} \\
\Lambdait^l_{\mathcal{A}_{l} \mathcal{A}_{l}} (X-L)_{\mathcal{A}_{l} \mathcal{A}_{l}} e - \mu e  + \Lambdait^l_{\mathcal{A}_{l} \mathcal{A}_x}  \Delta x_{\mathcal{A}_x} \\
\Lambdait^u_{\mathcal{A}_{u}\mathcal{A}_{u}} (U-X)_{\mathcal{A}_{u} \mathcal{A}_{u}} e - \mu e - \Lambdait^u_{\mathcal{A}_{u} \mathcal{A}_x}  \Delta x_{\mathcal{A}_x}
\end{bmatrix},
\end{align}
where $\nabla_x \mathcal{L}(x,\lambda)_{\mathcal{A}_x} = \nabla f(x)_{\mathcal{A}_x} - \lambda^l_{\mathcal{A}_x} + \lambda^u_{\mathcal{A}_x}$.
If the approximate $\Delta x_{\mathcal{A}_x}$ is exact then so is $\Delta x^{ls}_{\mathcal{I}_x}$ by (\ref{eq:genCase:Redx2syst_schur}). In consequence, the over-determined system (\ref{eq:genCase:LS:overdet}) has a unique solution that satisfies all equations, i.e., $\Delta \lambda^{ls}_{\mathcal{A}_x}$, or equvalently $ \Delta \lambda^{l,ls}_{\mathcal{A}_{l}}$ and $\Delta \lambda^{u,ls}_{\mathcal{A}_{u}}$ since $ \mathcal{A}_x  =  \mathcal{A}_{l} \cup \mathcal{A}_{u}$, are the corresponding parts of the solution to (\ref{eq:PDsyst}).   The solutions corresponding to the first and second block equation of (\ref{eq:genCase:LS:overdet}) will be labeled with superscript ``$b$'' and ``$-$'' respectively. These are given by
\begin{align} \label{eq:genCase:fullApprox:dlambdaAsimple1}
\begin{bmatrix}
 -I_{\mathcal{A}_x \mathcal{A}_{l}} & I_{\mathcal{A}_x \mathcal{A}_{u}} \end{bmatrix}  \begin{bmatrix} \Delta \lambda^{l,b}_{\mathcal{A}_{l}} \\ \Delta \lambda^{u,b}_{\mathcal{A}_{u}}
\end{bmatrix} & = - \Big[
\nabla f(x)_{\mathcal{A}_x} - \lambda^{l}_{\mathcal{A}_x} + \lambda^{u}_{\mathcal{A}_x}  + H_{\mathcal{A}_x \mathcal{A}_x}  \Delta x_{\mathcal{A}_x} + H_{\mathcal{A}_x \mathcal{I}_x}  \Delta x^{ls}_{\mathcal{I}_x}  \nonumber \\
& \qquad \> \>  - I_{\mathcal{A}_x \mathcal{I}_{l}}   \Delta \lambda^{l,ls}_{\mathcal{I}_{l}} + I_{\mathcal{A}_x \mathcal{I}_{u}} \Delta \lambda^{u,ls}_{\mathcal{I}_{u}}  \Big],
\end{align}
and
\begin{align*}  
&\begin{bmatrix}
(X-L)_{\mathcal{A}_{l} \mathcal{A}_{l}} & \\
 & (U-X)_{\mathcal{A}_{u} \mathcal{A}_{u}}  \\ 
\end{bmatrix} \begin{bmatrix} \Delta \lambda^{l,-}_{\mathcal{A}_{l}} \\ \Delta \lambda^{u,-}_{\mathcal{A}_{u}}
\end{bmatrix} =  \\ &-\begin{bmatrix}
\Lambdait^l_{\mathcal{A}_{l} \mathcal{A}_{l}} (X-L)_{\mathcal{A}_{l} \mathcal{A}_{l}} e - \mu e  + \Lambdait^l_{\mathcal{A}_{l} \mathcal{A}_x}  \Delta x_{\mathcal{A}_x} \\
\Lambdait^u_{\mathcal{A}_{u}\mathcal{A}_{u}} (U-X)_{\mathcal{A}_{u} \mathcal{A}_{u}} e - \mu e - \Lambdait^u_{\mathcal{A}_{u} \mathcal{A}_x}  \Delta x_{\mathcal{A}_x}
\end{bmatrix}.
\end{align*}
Alternatively, $ \Delta \lambda^{l,ls}_{\mathcal{A}_{l}}$ and $\Delta \lambda^{u,ls}_{\mathcal{A}_{u}}$ can be obtained as the least squares solution of (\ref{eq:genCase:LS:overdet}) 
\begin{equation}  \label{eq:genCase:fullApprox:dlambdaAls_syst}
\begin{aligned}
&\begin{bmatrix}
I_{\mathcal{A}_{l} \mathcal{A}_{l}} + (X-L)_{\mathcal{A}_{l} \mathcal{A}_{l}}^2  & \\
 & I_{\mathcal{A}_{u} \mathcal{A}_{u}} + (U-X)_{\mathcal{A}_{u} \mathcal{A}_{u}}^2  \\ 
\end{bmatrix} \begin{bmatrix} \Delta \lambda^{l,ls}_{\mathcal{A}_{l}} \\ \Delta \lambda^{u,ls}_{\mathcal{A}_{u}}
\end{bmatrix} = \nonumber \\
 &
\left[\begin{matrix}
 I_{\mathcal{A}_x \mathcal{A}_{l}}^T \Big( \nabla f(x)_{\mathcal{A}_x} - \lambda^l_{\mathcal{A}_x} + \lambda^u_{\mathcal{A}_x}  + H_{\mathcal{A}_x \mathcal{A}_x}  \Delta x_{\mathcal{A}_x} + H_{\mathcal{A}_x \mathcal{I}_x}  \Delta x^{ls}_{\mathcal{I}_x} - I_{\mathcal{A}_x \mathcal{I}_{l}}   \Delta \lambda^{l,ls}_{\mathcal{I}_{l}}    \\
- I_{\mathcal{A}_x \mathcal{A}_{u}}^T \Big( \nabla f(x)_{\mathcal{A}_x} - \lambda^l_{\mathcal{A}_x} + \lambda^u_{\mathcal{A}_x}  + H_{\mathcal{A}_x \mathcal{A}_x}  \Delta x_{\mathcal{A}_x} + H_{\mathcal{A}_x \mathcal{I}_x}  \Delta x^{ls}_{\mathcal{I}_x} - I_{\mathcal{A}_x \mathcal{I}_{l}}   \Delta \lambda^{l,ls}_{\mathcal{I}_{l}}
\end{matrix}\right.\\
& \quad
\left.\begin{matrix}
+ I_{\mathcal{A}_x \mathcal{I}_{u}} \Delta \lambda^{u,ls}_{\mathcal{I}_{u}} \Big) - (X-L)_{\mathcal{A}_{l} \mathcal{A}_{l}}  \left(\Lambdait^l_{\mathcal{A}_{l} \mathcal{A}_{l}} (X-L)_{\mathcal{A}_{l} \mathcal{A}_{l}} e - \mu e  + \Lambdait^l_{\mathcal{A}_{l} \mathcal{A}_x}  \Delta x_{\mathcal{A}_x} \right)  \\
\quad  \> \> \> + I_{\mathcal{A}_x \mathcal{I}_{u}} \Delta \lambda^{u,ls}_{\mathcal{I}_{u}} \Big) - (U-X)_{\mathcal{A}_{u} \mathcal{A}_{u}} \left( \Lambdait^u_{\mathcal{A}_{u}\mathcal{A}_{u}} (U-X)_{\mathcal{A}_{u} \mathcal{A}_{u}} e - \mu e - \Lambdait^u_{\mathcal{A}_{u} \mathcal{A}_x}  \Delta x_{\mathcal{A}_x}\right)
\end{matrix}\right],
\end{aligned}
\end{equation}
since $I_{\mathcal{A}_x \mathcal{A}_{l}}^T I_{\mathcal{A}_x \mathcal{A}_{l}} = I_{\mathcal{A}_{l} \mathcal{A}_{l}}$,  $I_{\mathcal{A}_x \mathcal{A}_{u}}^T I_{\mathcal{A}_x \mathcal{A}_{u}} = I_{\mathcal{A}_{u} \mathcal{A}_{u}}$ and $I_{\mathcal{A}_x \mathcal{A}_{l}}^T I_{\mathcal{A}_x \mathcal{A}_{u}} = I_{\mathcal{A}_x \mathcal{A}_{u}}^T I_{\mathcal{A}_x \mathcal{A}_{l}}=0$. The equations can also be written as
\begin{subequations} \label{eq:genCase:fullApprox:dlambdaAls}
\begin{align} \label{eq:genCase:fullApprox:dlambdaAls_l}
 \Delta \lambda^{l,ls}_{\mathcal{A}_{l}} & = \left( I_{\mathcal{A}_{l} \mathcal{A}_{l}} + (X-L)_{\mathcal{A}_{l} \mathcal{A}_{l}}^2 \right) \inv \Big[ I_{\mathcal{A}_x \mathcal{A}_{l}}^T \Big( \nabla f(x)_{\mathcal{A}_x} - \lambda^l_{\mathcal{A}_x} + \lambda^u_{\mathcal{A}_x} \nonumber \\
 & \qquad + H_{\mathcal{A}_x \mathcal{A}_x}  \Delta x_{\mathcal{A}_x} + H_{\mathcal{A}_x \mathcal{I}_x}  \Delta x^{ls}_{\mathcal{I}_x} - I_{\mathcal{A}_x \mathcal{I}_{l}}   \Delta \lambda^{l,ls}_{\mathcal{I}_{l}} + I_{\mathcal{A}_x \mathcal{I}_{u}} \Delta \lambda^{u,ls}_{\mathcal{I}_{u}} \Big) \nonumber \\
 & \qquad - (X-L)_{\mathcal{A}_{l} \mathcal{A}_{l}}  \left(\Lambdait^l_{\mathcal{A}_{l} \mathcal{A}_{l}} (X-L)_{\mathcal{A}_{l} \mathcal{A}_{l}} e - \mu e  + \Lambdait^l_{\mathcal{A}_{l} \mathcal{A}_x}  \Delta x_{\mathcal{A}_x} \right) \Big],
\end{align}
\begin{align}\label{eq:genCase:fullApprox:dlambdaAls_u}
\Delta \lambda^{u,ls}_{\mathcal{A}_{u}} & =  - \left( I_{\mathcal{A}_{u} \mathcal{A}_{u}} + (U-X)_{\mathcal{A}_{u} \mathcal{A}_{u}} \right) \inv \Big[  I_{\mathcal{A}_x \mathcal{A}_{u}}^T \Big( \nabla f(x)_{\mathcal{A}_x} - \lambda^l_{\mathcal{A}_x} + \lambda^{u}_{\mathcal{A}_x} \nonumber \\
 & \qquad + H_{\mathcal{A}_x \mathcal{A}_x}  \Delta x_{\mathcal{A}_x} + H_{\mathcal{A}_x \mathcal{I}_x}  \Delta x^{ls}_{\mathcal{I}_x} - I_{\mathcal{A}_x \mathcal{I}_{l}}   \Delta \lambda^{l,ls}_{\mathcal{I}_{l}} + I_{\mathcal{A}_x \mathcal{I}_{u}} \Delta \lambda^{u,ls}_{\mathcal{I}_{u}} \Big) \nonumber \\
 & \qquad + (U-X)_{\mathcal{A}_{u} \mathcal{A}_{u}} \left( \Lambdait^u_{\mathcal{A}_{u}\mathcal{A}_{u}} (U-X)_{\mathcal{A}_{u} \mathcal{A}_{u}} e - \mu e - \Lambdait^u_{\mathcal{A}_{u} \mathcal{A}_x}  \Delta x_{\mathcal{A}_x}\right)   \Big].
\end{align}
\end{subequations}
Finally, we state the main result which is analogous to the result of Theorem~\ref{thm:simplifiedCase}.
\begin{theorem} \label{thm:genCase}
Under Assumption~\ref{ass1}, let $\mathcal{B}\left((x^*, \lambda^*),
  \delta\right)$ and $\muhat$ be defined by
Lemma~\ref{lemma:background:FpnonsingCont} and
Lemma~\ref{lemma:LipcPath} respectively. For  $0< \mu \leq \muhat$ and $(x,\lambda) \in \mathcal{B}((x^*, \lambda^*), \delta)$, let $(\Delta x^N, \Delta \lambda^N)$ be the solution of (\ref{eq:PDsyst}) with $\mu^+ = \sigma \mu$, where $0< \sigma < 1$. Moreover, let the search direction components be defined as
\begin{subequations}
\begin{equation*}
\Delta x_i = \begin{cases} (\ref{eq:prop:genCase:schurBased:dx}) \mbox{ or } (\ref{eq:prop:genCase:compBased:dxLambdal}) &\>\> i \in \mathcal{A}_{l}, \\
(\ref{eq:prop:genCase:schurBased:dx}) \mbox{ or } (\ref{eq:prop:genCase:compBased:dxLambdau}) & \>\>  i \in \mathcal{A}_{u}, \\
  (\ref{eq:genCase:Redx2syst_schur}) & \>\>   i \in \mathcal{I}_x,
\end{cases}
\end{equation*}
\begin{equation*}
\Delta \lambda_i^l = \begin{cases}
 (\ref{eq:genCase:fullApprox:dlambdaAls_l}) \mbox{ or }(\ref{eq:genCase:fullApprox:dlambdaAsimple1}) &   i \in \mathcal{A}_{l},\\ 
 (\ref{eq:prop:genCase:compBased:dlambdal}) \mbox{ or } (\ref{eq:genCase:fullApprox:dlambdalI}) &  i \in \mathcal{I}_{l}, 
\end{cases}
\end{equation*}
\begin{equation*}
\Delta \lambda_i^u = \begin{cases}
 (\ref{eq:genCase:fullApprox:dlambdaAls_u}) \mbox{ or }(\ref{eq:genCase:fullApprox:dlambdaAsimple1}) & \> i \in \mathcal{A}_{u},\\ 
 (\ref{eq:prop:genCase:compBased:dlambdau}) \mbox{ or } (\ref{eq:genCase:fullApprox:dlambdauI}) & \> i \in \mathcal{I}_{u}.  \\ 
\end{cases}
\end{equation*}
\end{subequations}
Assume that $0 < \mu \le \muhat$ and $(x,\lambda)$ is sufficiently close to $(x^{\mu}, \lambda^{\mu})\in \mathcal{B}\left((x^*, \lambda^*), \delta\right)$ such that $\| F_{\mu} (x,\lambda) \| = \mathcal{O}(\mu)$. Then there exists $\mubar$, with $0 <\mubar \le \muhat$, such that for $0 < \mu \leq \mubar$ it holds that
\begin{equation*}
\left\| (
\Delta x, \Delta \lambda
)  - (
\Delta x^N , \Delta \lambda^N
)  \right\|   = \mathcal{O}(\mu^2).
\end{equation*}
\end{theorem}
\end{footnotesize}
%
%

\newpage
\bibliographystyle{myplain}     
\bibliography{references,references2} 
\end{document}